\documentclass[a4paper,11pt]{amsart}
\let\oldtocsection=\tocsection

\let\oldtocsubsection=\tocsubsection

\let\oldtocsubsubsection=\tocsubsubsection

\renewcommand{\tocsection}[2]{\hspace{0em}\oldtocsection{#1}{#2}}
\renewcommand{\tocsubsection}[2]{\hspace{1em}\oldtocsubsection{#1}{#2}}
\renewcommand{\tocsubsubsection}[2]{\hspace{2em}\oldtocsubsubsection{#1}{#2}}

\usepackage{latexsym}
\usepackage{amssymb}
\usepackage{color}
\usepackage[
frame,cmtip,arrow,matrix,line,graph]{xy}

\definecolor{darkgreen}{RGB}{47,139,79}
\definecolor{darkblue}{RGB}{36,24,130}

\usepackage{hyperref}
\hypersetup{
    colorlinks=true, 
    linktoc=all,     
    linktocpage,
    citecolor=darkgreen,
    filecolor=black,
    linkcolor=darkblue,
    urlcolor=black
}

\usepackage{comment}

\input xy
\xyoption{all}

\usepackage{amsmath}
\usepackage{amsthm}
\usepackage{amsfonts}
\usepackage{amscd}

\usepackage{epsfig}
\usepackage{lpic,enumerate}

\newtheorem{thm}{Theorem}[section]
\newtheorem{lem}[thm]{Lemma}
\newtheorem{prop}[thm]{Proposition}

\newtheorem{ex}[thm]{Example}
\newtheorem{cor}[thm]{Corollary}
\newtheorem{Def}[thm]{Definition}

\newtheorem*{THM}{Main Theorem}

\newtheorem{rem}[thm]{Remark}

\newcommand{\vs}{\vspace{4mm}}

\newcommand{\A}{\mathcal{A}}
\newcommand{\Ann}{\mathcal{A}nn}
\newcommand{\Ai}{\mathcal{A}_\infty}
\newcommand{\Ass}{\mathcal{A}ss}

\newcommand{\C}{\mathcal{C}}

\newcommand{\bC}{\overline{C}}
\newcommand{\tCE}{\widetilde{\mathcal{E}}}
\newcommand{\tbCE}{\widetilde{\mathcal{E}}}
\newcommand{\D}{\mathcal{D}}

\newcommand{\e}{\mathcal{E}}

\newcommand{\FF}{\mathbb{F}}
\newcommand{\F}{\mathcal{F}}

\newcommand{\uk}{\underline{k}}
\newcommand{\LL}{\mathcal{L}}
\newcommand{\M}{\mathcal{M}}
\newcommand{\N}{\mathbb{N}}

\newcommand{\pp}{\mathcal{P}}
\newcommand{\OO}{\mathcal{O}}
\newcommand{\OC}{\mathcal{OC}}

\newcommand{\QQ}{\mathbb{Q}}
\newcommand{\RR}{\mathbb{R}}

\newcommand{\SD}{\mathcal{SD}}

\newcommand{\Z}{\mathbb{Z}}

\newcommand{\Dif}{\textrm{Diff}}
\newcommand{\Hom}{\textrm{Hom}}

\newcommand{\Comp}{\operatorname{Ch}}

\newcommand{\Img}{\operatorname{Im}}

\newcommand{\al}{\alpha}
\newcommand{\be}{\beta}
\newcommand{\tb}{\tilde\beta}
\newcommand{\ga}{\gamma}
\newcommand{\Ga}{\Gamma}

\newcommand{\De}{\Delta}

\newcommand{\la}{\lambda}

\newcommand{\Om}{\Omega}
\newcommand{\s}{\sigma}
\newcommand{\Si}{\Sigma}

\newcommand{\rar}{\longrightarrow}
\newcommand{\inc}{\hookrightarrow}
\newcommand{\sta}{\stackrel}
\newcommand{\arsim}{\sta{\simeq}{\rar}}

\newcommand{\x}{\times}

\newcommand{\ot}{\otimes}

\newcommand{\w}{\wedge}

\newcommand{\lgl}{\langle}
\newcommand{\rgl}{\rangle}

\newcommand{\del}{\partial}

\newcommand{\emp}{\varnothing}

\newcommand{\Fun}{\operatorname{Fun}}

\newcommand{\Obj}{\operatorname{Obj}}

\newcounter{samcounter}

\newcommand{\oc}[2]{[\begin{subarray}{c} #2 \\ #1 \end{subarray}]}

\begin{document}

\bibliographystyle{plain}

\title{Hochschild homology of structured algebras}

\author{Nathalie Wahl}
\author{Craig Westerland}

\date{\today}

\begin{abstract}
We give a general method for constructing explicit and natural 
operations on the Hochschild complex of algebras over any prop with $\Ai$--multiplication---we think of such algebras as 
 $\Ai$--algebras ``with extra structure''. 
As applications, we obtain an integral version of the
Costello-Kontsevich-Soibelman moduli space action on the Hochschild complex of open TCFTs, the Tradler-Zeinalian 
action of Sullivan diagrams on the Hochschild complex of strict Frobenius algebras, and give applications
to string topology in characteristic zero. 
Our main tool is a generalization of the Hochschild complex. 
\end{abstract}

\maketitle

The Hochschild complex of an associative algebra $A$ admits a degree 1 self-map, Connes-Rinehart's boundary operator $B$.  If $A$ is Frobenius, the (proven) cyclic Deligne conjecture says that $B$ is the $\Delta$--operator of a BV-structure on the
Hochschild complex of $A$.  In fact $B$ is part of much richer structure, namely 
an action by the chain complex of Sullivan 
diagrams on the Hochschild complex \cite{Tradler-Z,KaufI,KaufII,Kau10}. A weaker version of Frobenius algebras, called here $\Ai$--Frobenius algebras, yields instead an action by the chains on the moduli space of Riemann
surfaces \cite{costello07,KS06,KaufI,KaufII}. Most of these results use a very appealing recipe for constructing such operations  introduced by Kontsevich in \cite{Kon94}.  Starting from a model for the moduli of curves in terms of the combinatorial data of fatgraphs, the graphs can be used to guide the local-to-global construction of an operation on the Hochschild complex of an $\Ai$-Frobenius algebra $A$ -- at every vertex of valence $n$, an $n$-ary trace is performed. 

In this paper we develop a general method for constructing 
explicit operations on the Hochschild complex of  $\Ai$--algebras ``with extra structure'', which contains these theorems as special cases.  
In constrast to the above, our method is {\em global-to-local}: we give conditions on a composable collection of operations
that ensures that it acts on the Hochschild complex of algebras of a
given type; by fiat these operations preserve composition, something
that can be hard to verify in the setting of \cite{Kon94}.  
After constructing the operations globally, we then show how to read-off the action explicitly, so that formulas for individual operations can also be obtained. Doing this we recover the same formuli as in the local-to-global approach. 
Our construction can be seen as a formalization and extension of the method of \cite{costello07} which considered the case of $\Ai$--Frobenius algebras. 

Our main result, which we will explain now in more details, gave rise to new computations, including a complete description of the operations on the Hochschild complex of commutative algebras \cite{Kla13A}, a description of a large complex of operations on the Hochschild complex of commutative Frobenius algebras \cite{Kla13B} and a description of the universal operations given any type of algebra \cite{Wah12}.  

\smallskip

An $\Ai$--algebra can be described as an enriched symmetric
monoidal functor from a certain dg-category $\Ai$ to $\Comp$, the dg-category of chain complexes over $\Z$.
The category $\Ai$ is what is called a {\em dg-prop}, a symmetric monoidal dg-category with objects the natural numbers. 
We consider here more generally dg-props $\e$ equipped with a dg-functor 
$i:\Ai\to\e$. Expanding on the terminology of McClure--Smith \cite{McCSmi02}, we call such a pair $\e=(\e,i)$ a {\em prop with $\Ai$-multiplication}. 
An $\e$--algebra is a symmetric monoidal dg-functor $\Phi:\e\to\Comp$. When $\e$ is a prop with $\Ai$--multiplication, 
 any $\e$--algebra comes with a specified $\Ai$--structure by restriction along $i$, and hence we can talk about the Hochschild complex of $\e$--algebras. 

We introduce in the present paper a generalization of the Hochschild complex which assigns to {\em any} dg-functor $\Phi:\e\to \Comp$ 
 a certain new functor $C(\Phi):\e\to\Comp$. The assignment has the
property that, for $\Phi$ symmetric monoidal,
$C(\Phi)$ evaluated at 0 is the usual Hochschild complex of the underlying $\Ai$--algebra. 
(The evaluation of $C(\Phi)(n)$ can more generally be interpreted in terms of higher Hochschild homology as in \cite{Pir00} associated to the union of a circle and
$n$ points.) This Hochschild complex construction can be iterated, and 
for $\Phi$ split monoidal\footnote{i.e.~such that the maps $\Phi(n)\ot\Phi(m)\to\Phi(n+m)$
  are isomorphisms (also known as {\em strong monoidal})}, 
the iterated complex $C^n(\Phi)$ evaluated at $0$ is the $n$th tensor power $(C(\Phi)(0))^{\ot n}$. 

Our main theorem, Theorem~\ref{action},  says that if the iterated Hochschild complexes of the functors $\Phi = \e(e,-)$ admit a natural action of a
dg-prop $\D$  of the form 
$$ C^n(\e(e,-))\ot \D(n,m) \to C^m(\e(e,-))$$
then the classical Hochschild complex of any split monoidal functor $\Phi:\e\to \Comp$ is a $\D$--algebra, i.e. there are maps
 $$ \big(C(\Phi)(0)\big)^{\ot n}\ot \D(n,m) \to \big(C(\Phi)(0)\big)^{\ot m}$$
associative with respect to composition in $\D$. 
This action is given explicitly and is natural in $\e$ and $\D$.

\smallskip

Before stating the theorem in more detail, we describe some consequences.
Let $\OO$ denote the {\em open cobordism category}, whose objects are the natural numbers and whose morphisms from $n$ to $m$ are chains
on the moduli space of the Riemann surfaces that are cobordisms from $n$ to $m$ intervals (or ``open strings''). 
Taking $\e=\OO$ and $\D=\C$, the closed co-positive\footnote{where the components of morphism each
   have at least one {\em incoming} boundary}  boundary cobordism category, Theorem~\ref{action} gives an integral version of Costello's main theorem in
 \cite{costello07}, i.e.,  an action of the chains of the moduli space of Riemann surfaces on the Hochschild chain complex of any {\em $\Ai$--Frobenius
   algebra}\footnote{called an {\em extended
  Calabi-Yau $\Ai\!$ category} in \cite{costello07}}. 
(See Theorem \ref{CosThm} and Corollary \ref{detCor}.)
Reading off our action on the Hochschild chains,  we
recover the recipe for constructing such an action given by Kontsevich and Soibelman in \cite{KS06}, thus tying these two pieces of work together.
We also get a version for non-compact\footnote{loosely, these are non-counital $\Ai$--Frobenius algebras} $\Ai$--Frobenius algebras by replacing $\OO$ by the positive boundary\footnote{where the components of morphism
  each have at least one {\em outgoing} boundary} open cobordism category and $\C$ with the positive and co-positive boundary category. (See Corollary \ref{posCor}.)

Applying Theorem~\ref{action} to the category $\e=H_0(\OO)$, we obtain an action of the chain complex of Sullivan diagrams on the
Hochschild complex of strict symmetric Frobenius algebras, recovering, with very different methods and after dualization, 
the main theorem of Tradler-Zeinalian in \cite{Tradler-Z}, see also \cite{KaufI,KaufII,Kau10}.  (See Theorem \ref{strict}.) 
In particular, in genus 0, this gives the cyclic Deligne conjecture first proved in \cite{kaufmann}, see also \cite{paolo_cyclic}. (See Proposition~\ref{std_BV_prop}.) 
Again, there is a non-compact version, which includes the operations later constructed in \cite{Abb13B} using Kontsevich's method. 

A consequence of our naturality statement, Theorem~\ref{factorization}, is that the aforementioned HCFT structure constructed by Costello and
Kontsevich-Soibelman factors through an action of Sullivan diagrams, when the $\Ai$--Frobenius algebra happens to be
strict. Sullivan diagrams model the harmonic compactification of moduli space \cite[Prop 5.1]{EgaKup}, so one can say that the action of moduli space
compactifies in that case. New operations arise from the compactification, and we know that these act non-trivially already on very basic Frobenius
algebras \cite[Prop 4.1 and Cor 4.2]{Wah12}.  
On the other hand,  a significant part of the homology of moduli space dies in the compactification, in particular the stable classes, which implies a
significant collapse of the original structure when the algebra is strict. 
(See Proposition~\ref{strictvanish} and Corollary \ref{factor_action_cor}.) 

We apply the above to the case of string topology for a simply-connected manifold $M$ over a field of characteristic zero, using the
strict Frobenius model of $C^*(M)$ given by Lambrechts-Stanley \cite{lambrechts_stanley,felix_thomas}, 
and obtain an HCFT structure  on $H^*(LM,\QQ)$
factoring through an action of Sullivan diagrams. 
We show in Proposition \ref{string_prop} that our construction recovers the BV structure on $H_*(LM)$ originally
introduced by Chas-Sullivan. The vanishing of the action of the stable classes in the HCFT structure furthermore agrees with
Tamanoi's vanishing result in \cite{tamanoi}. These vanishing results should though be contrasted with the non-vanishing results of \cite{Wah12} for
classes coming from the compactification, with non-trivial higher operations existing already on $H^*(LS^n)$. 
A different approach to moduli space or Sullivan diagram actions on $H^*(LM)$ can be found in \cite{godin07,DPR15,PoiRou,Poi10} (see also \cite{ChaSul04} in the equivariant setting). 
The papers  \cite{Cha05,CohGod} construct string topology actions
using a more restricted definition of Sullivan diagrams. It is natural to conjecture that these geometrically defined string topology operations are likewise  compatible with ours under the characteristic zero assumption. Yet a different approach is given in \cite{KaufI,KaufII,Kau10}, where Kaufmann constructs an action of Sullivan diagrams on the $E^1$--page of a spectral sequence converging to $H_*(LM)$. 
(The prop of open-closed Sullivan diagrams defined in \cite{Kau10} has its closed part isomorphic to the Sullivan diagrams considered
here, see Remark~\ref{KaufSD}.)

Further applications of our methods in the case of commutative and commutative Frobenius algebras where obtained by Klamt in \cite{Kla13A,Kla13B}. 
Other interesting examples of families of algebras to consider would be algebras over Kaufmann's prop of open Sullivan diagrams \cite{Kau10} (see Section~\ref{KP_section}), Hopf algebras, Poisson algebras and $E_n$--algebras, to name a few.

\medskip

We now describe our set-up and tools in a little more detail and give a more precise formulation of the main theorem. 

\smallskip

Recall from above that $\e$ is a dg-category, in fact a dg-prop, equipped with a functor $i:\Ai\to \e$, which will always be assumed to be the
identity on the objects, the natural numbers. 
Recall also that the Hochschild complex of a functor $\Phi:\e\to\Comp$ is defined here as a new functor
$C(\Phi):\e\to\Comp$.

To any such dg-category $\e$, we associate in this paper a larger dg-category, its {\em Hochschild core category} $C\e$.  The category $C\e$  
has objects pairs of natural numbers $\oc{m}{n}$, 
has $\e$ as a full subcategory on the objects $\oc{m}{0}$, and with the
morphisms from 
$\oc{m_1}{0}$ to $\oc{m_2}{n}$ the iterated Hochschild complex $C^n(\e(m_1,-))$ evaluated at $m_2$. 
If $\e$ is the open cobordism category
$\OO$, then $C\e$ is the open-to-open and open-to-closed part of the open-closed cobordism category.  
Given a monoidal category $\tCE$ with the same objects as $C\e$, we call it an {\em extension
of $C\e$} if it agrees with $C\e$ on the morphisms with source $\oc{m}{n}$ when $n=0$. An extension of $C\e$ can be thought of as the full
open-closed cobordism category, also including the closed-to-closed and closed-to-open morphisms.

\begin{THM}[Theorem~\ref{action} for $\Phi$ split symmetric monoidal, $C$ unreduced]  
Let $(\e,i)$ be a prop with $\Ai$--multiplication and $C\e\inc \tCE$ 
an extension of $C\e$ in the above sense. Then $\tCE$ acts naturally on the Hochschild complex of $\e$--algebras:
For any $\e$--algebra $A$ with
$C(A,A)$ its Hochschild complex,   
there are chain maps 
$$C(A,A)^{\ot n_1}\ot A^{\ot m_1}\ot \tCE(\oc{m_1}{n_1},\oc{m_2}{n_2})\ \rar\ C(A,A)^{\ot n_2}\ot A^{\ot m_2}$$
which are natural in $A$ and associative with respect to composition in $\tCE$. 
\end{THM}

The same holds for a reduced version of the Hochschild complex. \\

\noindent {\bf How to apply the theorem in practice.} This theorem applies to any prop  with $\Ai$--multiplication $\e$ and chosen extension $\tCE$. 
In practice, one starts with such an $\e$, which is the prop describing the type of algebra one is interested in. One can then construct its Hochschild core category $C\e$. This category is a bi-colored prop built to act on the pair $(A,C_*(A,A))$ for any $\e$--algebra $A$. However, it only has non-trivial operations of the form $A^{\ot m_1}\to C(A,A)^{\ot n_2}\ot A^{\ot m_2}$, encoded as morphisms from $\oc{m_1}{0}$ to $\oc{m_2}{n_2}$. To construct operations on 
the Hochschild complex of $\e$--algebras, one then needs to enhance the bi-colored prop $C\e$ to a larger category $\tCE$ that also has non-trivial morphisms with source  $\oc{m_1}{n_1}$ for at least some $n_1\neq 0$. The theorem is thus saying that it suffices for the prop $\tCE$ to contain $C\e$ to ensure that $\tCE$ defines natural operations on the Hochschild complex of $\e$--algebras. 

For each of the applications discussed above we have explicit such
extension categories $\tCE$, and the prop $\D$ mentioned above in each case is the
 ``closed-to-closed'' part of $\tCE$. 
These categories $\tCE$ are constructed using ad hoc methods coming from the geometry of the
situation. 
Given any prop with $\Ai$--multiplication $(\e,i)$, there exists a universal 
extension  which is much larger than the extensions considered here (see \cite{Wah12}). In the particular cases where $\e=\OO$ or $H_0(\OO)$, it is however shown in \cite[Rem 2.4 and
Thm B,C]{Wah12} that the universal extension is quasi-isomorphic to the props constructed here, so on the level of homology our small models do actually give all the operations.

\smallskip

The proof of the main theorem, inspired by, though independent of, \cite{costello07}, 
uses simple properties of the double bar construction, and a quotiented version
of it to take care of the equivariant version of the theorem under the action of the symmetric groups. Our action is explicit thanks to the
construction of an explicit pointwise chain homotopy inverse to the quasi-isomorphism of functors $C(B(\Phi,\e,\e))\to C(\Phi)$.
(See Proposition~\ref{beta}.)
As an example of how our theory can be applied, we give in Section \ref{strict_section} explicit formulas for the product, coproduct, and $\Delta$--operator on the Hochschild complex of strict Frobenius
algebras.

\medskip

The paper is organized as follows: 
Section~\ref{FatTrees} introduces the chain complexes of graphs used throughout the paper. In particular, our graph model for the open-closed
cobordism category and a category of Sullivan diagrams are constructed and studied in this section. Section~\ref{Algsec} gives some background on types of algebras occurring in
the paper. The short Section~\ref{Bar} reviews a few
properties of the double bar construction and its quotiented analog.  
Section~\ref{HComp} then defines the Hochschild complex operator, examines its properties, and proves
the main theorem. 
Section~\ref{ex}  gives applications: 
Section \ref{Cos}
gives the application to Costello's theorem, and Section \ref{ks_section} describes how to deduce the Kontsevich-Soibelman approach
from it. Sections \ref{twist_section} and \ref{pos} take care of the twisting by the determinant bundle and the positive boundary
variation. In  Section \ref{strict_section}, we treat the case of strict Frobenius algebras and Sullivan diagrams, with the application to
string topology given in Section \ref{strings}. Section~\ref{KP_section} gives the relationship the Kaufmann--Penner model for {\em string interaction}.
Finally, Sections \ref{Ai} and \ref{AxP} consider $\Ai$ and $\Ass\times\pp$--algebras
for $\pp$ an operad. 
Section \ref{conventions} sets up some notation and the 
Appendix (Section \ref{signsub}) explains how to compute signs given operations represented by graphs. 

\vs

We would like to thank Alexander Berglund, Kevin Costello, Daniela Egas,  Richard Hepworth, Ralph Kaufmann,  Anssi Lahtinen and Bob Penner for helpful algebraic, geometric and twisted conversations.  The first author
was supported by the Danish National Sciences Research Council (DNSRC) and
the European Research Council (ERC), as well as by
the Danish National Research Foundation (DNRF) through the Centre for Symmetry and Deformation.   The second author was supported by the Australian Research Council (ARC) under the Future Fellowship and Discovery Project schemes.  He thanks the University of Copenhagen for its hospitality and support.

\newpage

\tableofcontents

\section{Conventions  and terminology}\label{conventions}

In the present paper, we work in the category $\Comp$ of chain complexes over $\Z$, unless otherwise specified. 
We use the usual sign conventions so that the differential $d_V+d_W$ on a tensor product $V\ot W$ is 
$(d_V+d_W)(v\ot w)=d_V(v)\ot w+(-1)^{|v|}v\ot d_W(w)$. 

By a {\em dg-category}, we mean a category $\e$ whose morphism sets are chain complexes and whose 
composition maps $\e(m,n)\ot\e(n,p)\to \e(m,p)$ are chain maps. 
A {\em dg-functor} $\Phi:\e\to \Comp$  is a functor such that the structure maps $$c_\Phi:\Phi(m)\ot\e(m,n)
\to \Phi(n)$$ are chain maps.\footnote{Equivalently, $\e$ is a category enriched in $\Comp$, and $\Phi$ is an enriched functor.} For example, given any $r\in \Obj(\e)$, the functor $\Phi(m)=\e(r,m)$ 
represented by $r$ is a dg-functor.

\section{Graphs and trees}\label{FatTrees}

In this section, we give the background definitions about graphs, chain complexes of graphs etc.~ necessary for the rest of
the paper. In particular, we define black and white graphs and use
them to give graph models of the moduli space of Riemann surfaces, and define the {\em open cobordism category} $\OO$, the {\em open-closed
cobordism category} $\OC$ and the {\em category of Sullivan diagrams} $\SD$. 

Fat graphs were defined to give a combinatorial model of Teichm{\" u}ller space and  moduli space \cite{Pen87, BowEps}. Originally, these were considered for surfaces with punctures and later adapted to also model surfaces with fixed boundary components \cite{Pen04, Cos07, God07}. To relate these to the open and open-closed cobordism categories, we need to additionally be able to model the maps induced on moduli spaces by the gluing of surfaces along boundary intervals and circles. The gluing along boundary circles is easiest to describe using an asymmetrical model: our incoming and outgoing boundaries will be specified very differently. Indeed, the outgoing boundary circles will be given by {\em white vertices} while the incoming boundary circles will be identified as edge cycles in the graph. (A symmetrized model of this cobordism category exists (see \cite[Thm 3.1]{Wah12}) but will not be relevant here.) 

The model we use here is that of Costello \cite{Cos07, costello07} and Kontsevich-Soibelman \cite{KS06} (slightly reformulated), as this is the one that naturally occurs when considering the Hochschild complex of open field theories. These authors' work naturally includes a model for gluing along boundary intervals; however they did not study the gluing of surfaces along boundary circles. For this, we will use the work of Egas \cite{Ega15}. 
Kaufmann-Penner have proposed in \cite{KauPen} a different (partially defined) gluing operation on moduli space to model open and closed string interactions.  We defer to Section~\ref{KP_section} a discussion of the  differences of these approaches.

We will also be interested in a category of Sullivan diagrams that arises when considering the Hochschild complex of strict Frobenius algebras. The work of Egas-Kupers \cite{EgaKup} shows that this category yields a model of the harmonic compactification of moduli space. Sullivan diagrams are classically defined as equivalence classes of fat graphs made out of circles and chords. We will show here that they can be seen as equivalence classes of black and white graphs, which will allow us to define our category of Sullivan diagrams as a quotient of our cobordism category. Sullivan diagrams can moreover be described, as in the work of Kaufmann \cite{KaufI, KaufII}, in terms of arc systems in surfaces (see Remark~\ref{KaufSD}).  

\medskip

We start the section by defining fat graphs, which will be used to model the open cobordism category, and then extend them to {\em black and white graphs}, that will model the open-closed cobordism category, both categories being defined subsequently. We define the subcategories $\Ai$ and $\Ai^+$ that will be used in Section~\ref{Algsec} to define $\Ai$-- and unital $\Ai$--algebras. We also define a subcategory of Annuli that will play a role in defining the Hochschild complex in Section~\ref{HComp}.
At the end of the section, we define a category of Sullivan diagrams and study its relationship to the open-closed cobordism category.

\subsection{Fat graphs}
By a {\em graph} $G$ we mean a tuple $(V,H,s,i)$ where $V$ is the set of {\em vertices}, $H$ the set of {\em half-edges}, $s:H\to V$ is
the {\em source map} and $i:H\to H$ is an involution. Fixed points of the involution are called {\em leaves}. A pair $\{h,i(h)\}$ with
$i(h)\neq h$ is called an {\em edge}. We will consider graphs with vertices of any valence, also valence 1 and 2. 

We allow the empty graph. We will also consider the following degenerate graphs which fail to fit the
above description:

\begin{itemize} 

\item The {\em leaf} 
consisting of a single leaf and no vertices.  
\item The {\em circle} with no vertices. 

\end{itemize}
The leaf will appear in two flavors: as a {\em singly labeled leaf} and as a {\em doubly labeled leaf}. The circle will arise from gluing the doubly
labeled leaf to itself.

\smallskip

A {\em fat graph} is a graph $G=(V,H,s,i)$ together with a cyclic ordering of each of the sets $s^{-1}(v)$ for
$v\in V$. 
The cyclic orderings define {\em boundary cycles} on the graph, which are sequences of consecutive half-edges
corresponding to the boundary components of the surface that can be obtained by thickening the graph.
Figure~\ref{graphsex}(a) shows an example of a
fat graph with two boundary cycles (the dotted and dashed lines), where the cyclic ordering at vertices is that inherited from the plane.
 (Formally, if $\s$ is the permutation of
 $H$ whose cycles are the cyclic orders at each vertex of the graph, then the boundary cycles of $G$ are the cycles of the permutation
$\s.i$ \cite[Prop 1]{God07}.)

\begin{figure}[h]
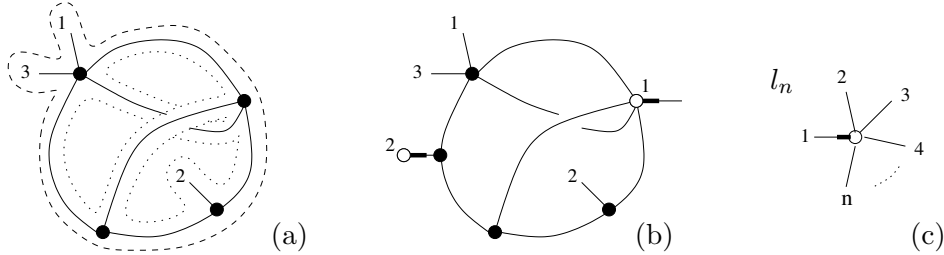

\begin{lpic}{graphs3(0.6,0.6)}
 \lbl[b]{62,0;(a)}
 \lbl[b]{142,0;(b)}
 \lbl[b]{202,0;(c)}
\lbl[b]{170,35;$l_n$}
\end{lpic}
\caption{Fat graph and black and white graphs.}\label{graphsex}
\end{figure}

\subsection{Orientation}
An {\em orientation} of a graph $G$ is a unit vector in $\operatorname{det}(\RR (V\sqcup H))$. 
The degenerate graphs have a canonical formal positive orientation. 
Note moreover that any odd-valent (in particular trivalent) graph has a canonical orientation
$$v_1\w h^1_1\w\dots\w h^1_{n_1}\w\dots\w v_k\w h^k_1\w\dots\w h^k_{n_k}$$
where $v_1,\dots,v_k$ is a chosen ordering of the vertices of the graph and $h^i_1,\dots,h^i_{n_i}$ is the set of half-edges at $v_i$
in their cyclic ordering.

\subsection{Black and white graphs}\label{bargraphs}
A {\em black and white graph} is a fat graph whose set of vertices is given as $V=V_b\coprod V_w$, with $V_b$ the set of {\em black vertices} and $V_w$ the set
of {\em white vertices}. 
The {\em white vertices} are labeled $1,2,\dots,|V_{w}|$ and are allowed to be of any valence (also 1 and 2). 
The {\em black vertices} are unlabeled and must be at least trivalent. 
Moreover, each white vertex is equipped with a choice of {\em start half-edge}, i.e.~a choice of an element in 
$s^{-1}(v)$ for each  $v\in V_w$.  Equivalently,  the set of half-edges $s^{-1}(v)$ at each white vertex $v$ has an
actual ordering, not just a cyclic ordering.  

We define a {\em $\oc{m}{p}$--graph} to be a black and white graph with $p$ white vertices and $m$ leaves labeled $\{1,\dots,m\}$. 
A   {\em $\oc{m}{p}$--graph} may have additional unlabeled leaves if they are the start half-edge of a white vertex. 
Figure~\ref{graphsex}(b) shows an example of a $\oc{3}{2}$--graph, with the start half-edges marked by
thick lines.

To define the Hochschild complex, we will use the $\oc{n}{1}$--graph, denoted $l_n$, depicted in Figure~\ref{graphsex}(c) 
which has a single vertex which is white, and $n$ leaves labeled
cyclically, with the first leave as start half-edge. (As $l_n$ has only one white vertex, we drop its label which is automatically $1=|V_w|$.)

A $\oc{m}{0}$--graph is just an ordinary fat graph whose vertices are at least trivalent and which has $m$ labeled leaves. 
Figure~\ref{graphsex}(a) gives an example of a $\oc{3}{0}$--graph. 

\smallskip

We will consider isomorphism classes of black and white graphs. If the graphs are oriented or have labeled leaves, we always assume this is preserved
under the isomorphism. 
Note that when two black and white graphs are isomorphic, the isomorphism is unique whenever  each component of 
the graph has at least one labeled leaf or at least one
white vertex: starting with the leaf or the start half-edge of the white vertex, and using that the cyclic orderings at vertices are preserved, one
can check by going around the corresponding component of the graph that the isomorphism is completely determined.

\subsection{Edge collapses and blow-ups}\label{col-blow}
For a black and white graph $G$ and an edge $e$ of $G$ which is {\em not a cycle} and {\em does not join two white vertices}, 
we denote $G/e$ the set of isomorphism classes of black and white graphs that can be obtained from $G$ by collapsing the edge
$e$, identifying its two end-vertices, declaring the new vertex to be white with the same label if one of the collapsed vertices was white---in
particular, the number of white vertices is constant under edge collapse.  
Graphs in $G/e$ have naturally induced cyclic orderings at their vertices. If the new vertex is black, the collapse is unique. 
If the new vertex is white, it has a well-defined start half-edge unless the start half-edge of the original white
vertex is collapsed with $e$, in which case  there is  a collection of possible collapses of $G$ along $e$, one
for each choice of placement of the start half-edge at the new white vertex among the leaves originating from the collapsed black vertex of the
original graph $G$. (See Figure~\ref{collapse} for an example.)
\begin{figure}[h]
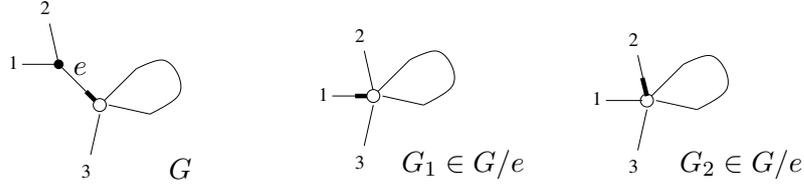

\begin{lpic}{collapse(0.6,0.6)}
 \lbl[b]{38,0;$G$}
 \lbl[b]{16,23;$e$}
\lbl[b]{100,0;$G_1\in G/e$}
\lbl[b]{161,0;$G_2\in G/e$}
\end{lpic}
\caption{The two possible collapses of $e$ in $G$}\label{collapse}
\end{figure}

If $G$ is oriented, the graphs in $G/e$ inherit an orientation as follows: If $e=\{h_1,h_2\}$ with $s(h_1)=v_1$, $s(h_2)=v_2$, and  
writing the orientation of $G$ in the form $v_1\w v_2\w h_1\w h_2\w x_1\w \dots\w x_k$,
we define the 
orientation of the collapsed graph to be  $v\w x_1\w \dots \w x_k$,
where $v$ is the vertex of the collapsed graph coming from identifying $v_1$ and $v_2$.

\medskip

For an (oriented) black and white graph $G$, 
we call an (oriented) black and white graph $\tilde{G}$ a {\em blow-up} of $G$ if there exists an edge $e$ of $\tilde{G}$ such that $G\in\tilde{G}/e$. 
The first line in Figure~\ref{dL3} shows all the possible blow-ups of the graph $l_3$.

\subsection{Chain complex of black and white graphs}\label{chaincpx}
Let $BW-{Graphs}$ denote the chain complex generated as a $\Z$--module 
by isomorphism classes of (not necessarily connected, possibly degenerate) oriented black and white graphs, 
modulo the relation that $-1$ acts by reversing the orientation.
The {\em degree} of a black and white graph is 
$$\deg(G)=\sum_{v\in V_b}(|v|-3)+\sum_{v\in V_w}(|v|-1),$$ where $|v|$ denotes the valence of $v$.  The degenerate graphs have degree
0.  
The map 
 $$\hat d\colon  G\mapsto \sum_{\begin{subarray}{l}\ \,(\tilde{G},e)\\ G\in\tilde{G}/e \end{subarray}} \tilde{G}$$ 
summing over all blow-ups of $G$ defines a differential on   $BW-{Graphs}$. Indeed, 
 we have 
$$(\hat d\,)^2(G)=\sum_{\begin{subarray}{l}\ \,(\tilde{G},e)\\ G\in \tilde{G}/e\end{subarray}} 
\Big(\sum_{\begin{subarray}{l}\ \,(\widehat{G},f)\\ \tilde{G}\in \widehat{G}/f\end{subarray}}
  \widehat{G}\Big)=
\sum_{\begin{subarray}{l}\ \,(\widehat{G},f,e)\\ G\in\widehat{G}/(f,e)\end{subarray}} 
   \widehat{G}$$ 
as any pair $(\tilde G,e),(\hat G,f)$ as above defines a triple $(\hat G,f,e)$ taking $e\in \hat G$ to be the inverse image of $e\in \tilde G$ under the collapse of $f$, and conversely, given a triple $(\hat G,f,e)$ as above, there is a unique $\tilde G$ in $\hat G/f$ with the property that $G\in \tilde G/e$. Indeed, if $\hat G/f$  contains several elements, they only differ by the placements of the start half-edge at a newly created white vertex, but only one of these placements can be  compatible with the start half-edge of $G$ at that vertex, also if $e$ defines a further collapse of a black vertex to that white vertex. The fact that $\hat d^2=0$ then follows from checking that the
orientations of $\widehat{G}/f/e$ and $\widehat{G}/e/f$ are opposite so that each term   $(\widehat{G},f,e)$ cancels with the term
$(\widehat{G},e,f)$. 

\smallskip

Let $\oc{m}{p}-{Graphs}$ now denote the chain complex generated as a $\Z$--module 
by isomorphism classes of (not necessarily connected, possibly degenerate) oriented $\oc{m}{p}$--graphs, 
modulo the relation that $-1$ acts by reversing the orientation. Recall that $\oc{m}{p}$--graphs are black and white graphs with $p$ white vertices
and $m$ labeled leaves, and that the only unlabeled leaves allowed in $\oc{m}{p}$--graphs are those which are start half-edge of a white vertex. 

A black and white graph $G$ with $p$ white vertices and $m$ labeled leaves has an underlying $\oc{m}{p}$--graph $\lfloor G\rfloor$ 
defined by 
 $\lfloor\tilde{G}\rfloor=\tilde{G}$ unless $\tilde{G}$ has unlabeled
leaves which are not the start half-edge of a white vertex. In such a leaf $l$ is attached at a  trivalent black vertex $v$, 
the vertex $v$ and the leaf are forgotten in $\lfloor\tilde{G}\rfloor$, and if such a leaf is attached at a white vertex (which will automatically be
at least bivalent) or at black vertex of valence at least 4,
we set $\lfloor\tilde{G}\rfloor=0$. 
The orientation of $\lfloor\tilde{G}\rfloor$ when $\lfloor\tilde{G}\rfloor\neq\tilde{G}$ (or 0) is obtained by first rewriting the
 orientation of $G$ in the form $v\w l\w h_1\w h_2\w\dots$ for $s^{-1}(v)=(l,h_1,h_2)$ in that cyclic ordering, and then removing the first 4
 terms.

We now define  the differential on  $\oc{m}{p}-{Graphs}$ as  $dG=\lfloor
\hat dG\rfloor$. 
Figure~\ref{dL3} shows three examples of differentials. 
\begin{figure}[h]
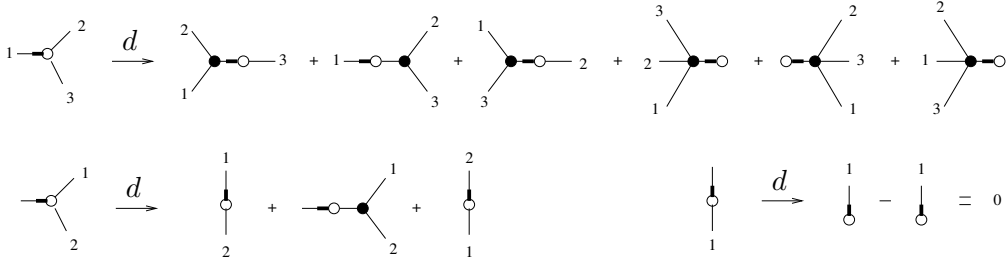

\begin{lpic}{L3.2(0.5,0.5)}
\lbl[l]{31,58;$d$}
\lbl[l]{32,19;$d$}
\lbl[l]{202,21;$d$}
\end{lpic}
\caption{Differential applied to the graph $l_3$ and to two graphs with an unlabeled start-leaf.}\label{dL3}
\end{figure}

\begin{lem}\label{floor}
$d\lfloor G\rfloor=\lfloor \hat dG\rfloor$.
\end{lem}

\begin{proof}
If $G$ has unlabeled leaves at trivalent black vertices which are forgotten in $\lfloor G\rfloor$, since trivalent black vertices cannot be expanded,
$\hat d G$ will also have  such unlabeled leaves and it does not alter the differential if they are forgotten before or after we sum over all
blow-ups. The case $\lfloor G\rfloor \neq 0$ follows. 

Suppose now that $\lfloor G\rfloor= 0$. If $G$ has more than one unlabeled leaf at a high valent black vertex or non-start half-edge at a white
vertex, it is immediate that $\lfloor \hat d G\rfloor$ is also 0 as all terms in $\hat d G$ will also have unlabeled leaves of that type. 
If $G$ has a single unlabeled leaf at a high valent black or white vertex, then $\hat d G$ will have exactly two terms $G_l,G_r$ such that the unlabeled leaf is
at a black trivalent vertex, namely the blow-ups of that the vertex blowing out the leaf together with its left and its right neighbor
respectively. One has that $\lfloor G_l\rfloor =\lfloor G_r\rfloor$ with opposite orientations. 
\end{proof}

\begin{prop}
The map $d$ is a differential. 
\end{prop}

\begin{proof}
This follows directly from the lemma using the fact that $\hat d$ is a differential: $d^2G=d \lfloor \hat d G\rfloor=\lfloor (\hat d\,)^2 G\rfloor=0$.
\end{proof}

\subsection{The open cobordism category $\OO$}\label{OO} 
For us, the open cobordism category is a dg-category with objects the natural numbers, thought of as representing disjoint unions of intervals, and morphism
given by chain complexes with homology that of the moduli spaces of Riemann cobordisms between the intervals. 
The composition is induced by the composition of cobordisms, i.e.~gluing along the intervals. Here, we follow the terminology of Moore-Segal \cite{moore-segal}, amongst others.

Fat graphs, without leaves, were invented to 
define a cell decomposition of Teichmuller space (see the work of Bowditch-Epstein \cite{BowEps}, Harer \cite{Har84},
Penner \cite{Pen87,Pen88}), and the chain complex $\oc{0}{0}-Graphs$ defined above 
is the corresponding cellular complex of the quotient of Teichmuller
space by the action of the mapping class group, namely the coarse moduli
space of Riemann surfaces. Similarly, fat graphs with leaves define a chain
complex for the moduli space of surfaces with fixed
boundaries, or with fixed intervals in their boundaries 
(see Penner \cite{Pen04,Pen06},
Godin \cite{God07}, Costello \cite[Sect.~6]{costello07} and \cite{Cos07}). 
As already remarked in \ref{bargraphs}, graphs with labeled leaves have no symmetries. 
The same holds for Riemann surfaces as soon as part of the boundary of the surface is assumed to have a fixed Riemann structure.  
It follows that the moduli space, being the quotient of Teichmuller space by a free action of the mapping class group of the surface, is a classifying space for that mapping
class group, and the chain complex of $\oc{m}{0}$--graphs of that surface type when $m>0$ computes the homology of the (now fine) moduli
space as well as the homology of the corresponding mapping class group. 

Let $S$ be a surface and $I$ a collection of intervals in its boundary. If we
denote by $\M(S,I)$ the moduli space of Riemann surfaces with a fixed
structure on an $\varepsilon$--neighborhood of $I$ (with the convention that $\M(S^1\x I,\emp)=*=\M(D^2,\emp)$, and the moduli space is the coarse moduli space for other surfaces with
no intervals in their boundary), we have the following:

\begin{thm}\label{OOmod}
There is an isomorphism 
$$H_*(\oc{m}{0}-Graphs)\ \ \cong\ \bigoplus_{(S,I)}
H_*(\M(S, I))$$
where $(S, I)$ ranges over all (possibly disconnected) oriented surfaces $S$ 
with $I$ a collection of $m$ labeled intervals in $\partial S$.  Here, each component of $S$ must  have nonempty boundary. 
\end{thm}

While the many references indicated above give similar such combinatorial models for moduli space, one may explicitly extract this result from
\cite{Cos07}, via the enumeration of the cells in Costello's cellular model for moduli space after Proposition 2.2.3 in \cite{Cos07} with $s=0$.  An alternative reference, with a different proof, is the restriction to the open part of 
\cite[Thm A]{Ega15} as $\oc{m}{0}-Graphs$ are the same as ``open fat graphs'' in the terminology of that paper. 
We can thus use $\oc{m}{0}$--graphs to provide a model for the open cobordism category, which we do now.

Let $\OO$ be the symmetric monoidal dg-category with objects the natural numbers (including 0) and morphisms 
from $m$ to $n$ the chain complex $$\OO(m,n):=\oc{m+n}{0}-Graphs$$ of fat graphs with $m+n$ labeled leaves. 
See Figure~\ref{OOA} for examples of morphisms in $\OO$. 

\begin{figure}[h]
\begin{lpic}{OOandAi.2(0.7,0.7)}
 \lbl[l]{66,0;(a)}
 \lbl[l]{162,0;(b)}
 \lbl[l]{45,30;\small{$4\ =1_{out}$}}
 \lbl[l]{45,23;\small{$7\ =4_{out}$}}
 \lbl[l]{45,17;\small{$5\ =2_{out}$}}
 \lbl[l]{45,11;\small{$6\ =3_{out}$}}
 \lbl[l]{45,5;\small{$8\ =5_{out}$}}
 \lbl[r]{-1,24;\small{$1$}}
 \lbl[r]{-1,34;\small{$2$}}
 \lbl[r]{-1,15;\small{$3$}}
 \lbl[r]{98,14;\small{$1$}}
 \lbl[r]{98,28;\small{$2$}}
 \lbl[r]{98,0;\small{$3$}}
 \lbl[r]{98,7;\small{$4$}}
 \lbl[r]{98,35;\small{$5$}}
 \lbl[r]{98,22;\small{$6$}}
 \lbl[l]{141,20;\small{$7\ =1_{out}$}}
 \lbl[l]{141,4;\small{$8\ =2_{out}$}}
\end{lpic}
\caption{Morphisms of $\OO(3,5)$ and $\Ai(6,2)$.}\label{OOA}
\end{figure}

Relabeling the $(m+j)$th leaf of a graph in $\OO(m,n)$ by $j_{out}$ as in the figure, 
composition $G_2\circ G_1$ is defined by gluing the leaf $j_{out}$ of $G_1$ with the  $j$th leaf of $G_2$, so that the two
leaves form an edge in the glued graph.  (More formally, we compose graphs by unioning vertices and half edges and altering the involution so that
the glued leaves are mapped to each other under the involution.) The orientation is obtained by juxtaposition (wedge product). 
The rule for gluing the exceptional graphs is as follows: \\
- Gluing a leaf labeled on one side has the effect of removing the corresponding leaf of
the other graph if this is a degree 0 operation (i.e. if the leaf was attached to a trivalent
vertex)---otherwise the gluing just gives 0.  If the trivalent vertex is $v$ with half edges $h_1,h_2,h_3$ attached to it in that cyclic order, and
the graph has orientation $v\w h_1\w h_2\w h_3\w x_1\w\dots\w x_k$, then the glued graph has orientation $x_1\w\dots\w x_k$.\\
- Gluing a doubly labeled leaf has the effect of relabeling the leaf
of the other graph if the labels of the leaf are incoming and outgoing. If both labels are incoming or outgoing, it
attaches the corresponding leaves of the other graph together so they form an edge. 

\smallskip

The fact that this gluing is compatible with the gluing of Riemann
surfaces along the intervals $I$ is \cite[Prop 6.1.5]{costello07}  (see also \cite[Thm 3.30]{Ega15} for a different proof). This can be understood
as follows: Fat graphs come from a cell decomposition of Teichm\"uller space, a fat graph in a surface defining a dual decomposition of the surface
into polygons. When the fat graph has leaves, the endpoints of the leaves should be placed in the boundary of the surface. Such graphs are dual to a polygonal decomposition of the surface with an interval around each leaf in the boundary of the surface
being always part of the decomposition. The gluing along leaves then corresponds in Teichm\"uller space to gluing such polygonal decompositions along such specified intervals, remembering the interval in the decomposition of the glued surface. 

\begin{rem}{\rm 
We note that this gluing along open boundaries does not agree with the one defined in \cite{Kau10,KauPen}, at least not under the most natural equivalence between the arc model used in those papers (restricting to the case of a single brane) and the dual fat graph model we use. Indeed, the gluing there is defined along intervals between marked points defined by the leaves, instead of around such points as we do.  
These boundary intervals between leaves correspond in the graph to sequences of edges between leaves, and the gluing is thus a gluing along sequences of edges. Despite the different starting point, we expect that their gluing, where it is defined, also models the gluing of Riemann surfaces along boundary intervals. We refer to Section~\ref{KP_section} for a further discussion of the Kaufmann-Penner model. 
}\end{rem}

\smallskip

The symmetric monoidal structure of $\OO$ is defined by taking disjoint union of graphs.
The identity morphisms and the symmetries in the category are given by (possibly empty) unions of doubly labeled leaves.

\subsection{The categories  $\Ai$ and $\Ai^+$}\label{Aisec} 
We let $\Ai$ denote the subcategory of directed forests in $\OO$, i.e.~$\Ai$ has the same objects as $\OO$, the natural numbers, 
and the chain complex $\Ai(m,n)$ of morphisms from
$m$ to $n$ is generated by graphs which are disjoint unions of $n$ trees with a total of $m_1+\dots+m_n=m$ incoming leaves, with each $m_i>0$,  
in addition to the root
of the tree which is labeled as an outgoing leaves. Here we allow the degenerate graphs consisting of single leaves labeled both sides (as one input
and one output), as well as the empty graph defining the identity morphism on $0$. 
We let $\Ai^+$ denote the slightly larger category where also the leaf labeled on one side
as an output is allowed. 
See Figure~\ref{OOA} (b) for an example of a morphism in $\Ai$. 

In \ref{Aialg}, we will relate these categories to $\Ai$-- and unital $\Ai$--algebras.

\subsection{The open-closed cobordism category $\OC$}\label{OCsec}
The open-closed cobordism category is a dg-category with objects pairs of natural numbers, thought of as representing disjoint unions of intervals  ({\em open
  boundaries}) and
circles ({\em closed boundaries}), and with morphisms defined as chain complexes on the moduli spaces of Riemann cobordisms between the collections of intervals and circles. Composition is again induced by composing cobordisms, i.e.~gluing surfaces. 
We give here a model of this category with morphisms sets given by chain complexes of black and white graphs. The white vertices will model outgoing closed boundary components while incoming closed boundaries will be modeled by cycles of edges in the graph starting at leaves (and determined by such leaves), and open boundaries will be modeled by leaves elsewhere in the graph. This asymmetric description of the morphisms is necessary for being able to define the composition---we need the boundary circles on one side to be disjointly embedded in some way, a property achieved by the white vertices which are by definition disjoint when distinct. But the particular choice of model made here really comes from studying the Hochschild complex of the open category $\OO$ (most particularly Lemma \ref{OOext}), and this is why the same model appears both in the work of  Costello and of Kontsevich-Soibelman in \cite{costello07,KS06}.

\smallskip

Let $\OC$ denote the dg-category with objects pairs of natural numbers $\oc{m}{n}$, for $m,n\ge 0$ representing $m$ intervals and $n$ circles, 
and with morphisms 
$$\OC(\oc{m_1}{n_1},\oc{m_2}{n_2})\ \subset\ \oc{n_1+m_1+m_2}{n_2}-Graphs$$ 
the subcomplex of $\oc{n_1+m_1+m_2}{n_2}$--graphs with the
first $n_1$ leaves sole labeled leaves in their boundary cycle, representing cobordisms from $m_1$ intervals and $n_1$ circles to
$m_2$ intervals and $n_2$ circles. Theorem \ref{cos_thm} below says that the chain complex $\OC(\oc{m_1}{n_1},\oc{m_2}{n_2})$ does indeed compute the homology of the
moduli space of Riemann structures on such cobordisms. Note that these moduli spaces are classifying spaces for the mapping class groups of the corresponding surfaces fixing marked circles and intervals in their boundary. Now the mapping class group of a surface fixing an interval in some boundary component (or several intervals) is isomorphic to the mapping class group fixing the whole boundary. This is a way of understanding how the moduli of surfaces with a fixed boundary circle can be given by the graph complex for a surface with a single leaf in its boundary component, given that we already know from the open cobordism category that leaves model fixed intervals.

One can think of the composition of two graphs on closed boundaries as being induced by ``gluing'' the outgoing circles represented by white vertices of the one graph along the cycles in the other graph representing incoming circles.  In many respects, this resembles the composition law in the cactus operad. In practice, this means that for each white vertex, we will attach the half-edges at that white vertex in all possible ways (giving the right degree) along the corresponding cycle of the other graph, with the start half-edge at the ``start leaf'' of that cycle (i.e.~the leaf defining it), and respecting the cyclic ordering.  Formally, 
given graphs $G_1\in\OC(\oc{m_1}{n_1},\oc{m_2}{n_2})$ and $G_2\in\OC(\oc{m_2}{n_2},\oc{m_3}{n_3})$, their 
composition  is defined as the sum 
$G_2\circ G_1=\sum \lfloor G\rfloor$ over
all possible black and white graphs $G$ that can be obtained from $G_1$ and $G_2$ by: 
\begin{enumerate}
\item removing the $n_2$ white vertices of $G_1$,
\item identifying the start half-edge of the $i$th white vertex $v_i$ of $G_1$ with the $i$th leaf $\la_i$ of $G_2$,
\item attaching the remaining leaves in $s^{-1}(v_i)$ to vertices of the boundary cycle of $G_2$ containing $\la_i$, respecting the
  cyclic ordering of the leaves,
\item attaching the last $m_2$ labeled leaves of $G_1$ to the leaves of $G_2$ labeled $n_2+1,\dots,n_2+m_2$, respecting the order, 
\end{enumerate}
where $\lfloor G\rfloor$ is defined as in Section~\ref{chaincpx}. 
(Unlabeled leaves are produced during the gluing operation in the following situation: if the $i$th white vertex of $G_1$ has an unlabeled start
half-edge, the $i$th leaf $\la_i$ of $G_2$ becomes unlabeled in the glued graph.) 

The orientation of $G_2\circ G_1$ is obtained by juxtaposition after removing the white vertices $v_i$ and their start half-edges $h_i$ from the
orientation of $G_1$ ordered as pairs $v_i\w h_i$, and then removing quadruples $v\w l\w h_1\w h_2$ as in \ref{chaincpx} for each forgotten unlabeled leaf. 
Figure~\ref{gluing} give two examples of compositions in $\OC$.
\begin{figure}[h]
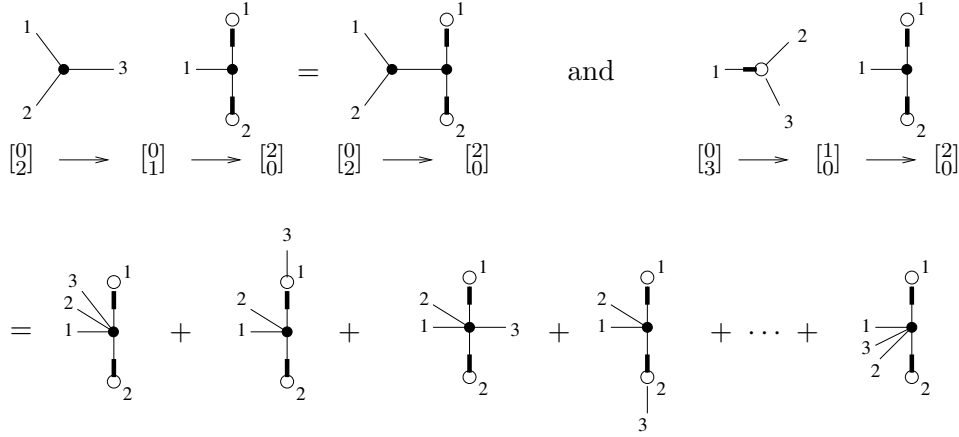

\begin{lpic}{gluing(0.6,0.6)}
 \lbl[b]{0,57;$\oc{2}{0}$}
 \lbl[b]{29,57;$\oc{1}{0}$}
 \lbl[b]{55,57;$\oc{0}{2}$}
 \lbl[b]{72,57;$\oc{2}{0}$}
 \lbl[b]{63,78;$=$}
 \lbl[b]{100,57;$\oc{0}{2}$}
 \lbl[b]{151,57;$\oc{3}{0}$}
 \lbl[b]{177,57;$\oc{0}{1}$}
 \lbl[b]{203,57;$\oc{0}{2}$}
 \lbl[b]{125,78;and}
 \lbl[b]{0,21;$=$}
 \lbl[b]{35,20;$+$}
 \lbl[b]{72,20;$+$}
 \lbl[b]{118,20;$+$}
 \lbl[b]{163,20;$+\ \cdots\ +$}
\end{lpic}
\caption{Compositions in $\OC$. (Note that we interpret the same graph in two different ways in the first and second composition.)}\label{gluing}
\end{figure}

\begin{lem}
The composition of graphs defined above is a chain map
$$\OC(\oc{m_1}{n_1},\oc{m_2}{n_2})\ot\OC(\oc{m_2}{n_2},\oc{m_3}{n_3})\rar \OC(\oc{m_1}{n_1},\oc{m_3}{n_3}).$$
Moreover it is associative. 
\end{lem}

\begin{proof}
Given $G_1,G_2$ as above, we need to check that $$d(G_2\circ G_1)=G_2\circ dG_1 + (-1)^{|G_1|} dG_2\circ G_1.$$
Recall from \ref{chaincpx} that $dG=\lfloor \hat dG\rfloor$, for $\hat d$ the differential in black and white graphs. 
We have similarly $G_2\circ G_1=\lfloor G_2\hat\circ G_1\rfloor$ where 
$G_2\hat\circ G_1$ denotes the composition of graphs as black and white graphs, without taking the underlying $\oc{m}{p}$--graphs. 

We first check that $\hat d(G_2\hat \circ G_1)=G_2\hat\circ \hat d G_1 + (-1)^{|G_1|} \hat d G_2\hat \circ G_1.$
Call a vertex of $G_2\hat\circ G_1$ {\em special} if it comes from a vertex of one of the first $n_2$ boundary cycles of $G_2$.  
The left-hand side has terms coming from  
\begin{enumerate}
\item blowing up at a non-special vertex,
\item blowing up at a special vertex in such a way that the newly created vertices are either 
\begin{tabular}[t]{l}
white, \\
black with no half-edges of $G_2$, or\\
black with at least two half-edges of $G_2$,
\end{tabular}
\item blowing up at a special vertex in such a way that one of the newly created vertices is black with exactly one half-edge of $G_2$
  attached to it. 
\end{enumerate} 
The terms of type (1) and (2) are exactly the terms occurring in $G_2\hat\circ \hat dG_1 + (-1)^{|G_1|} \hat dG_2\hat \circ G_1$ as black and white graphs, i.e.~before
taking the underlying $\oc{m}{p}$--graphs $\lfloor G\rfloor$. 
Indeed, type (1) terms correspond to blowing up at vertices of $G_1$ or $G_2$ which are not affected by the gluing, and type (2) terms
correspond either to blowing up a vertex of $G_2$ on a incoming cycle and then attach edges of $G_1$, or, in the case where one of the
vertices is black with no half-edges of $G_2$ attached to it, this correspond to blowing-up at a white vertex of $G_1$ and then glue
the resulting graph to $G_2$. This covers all the possibilities. 

The fact that the signs agree follows from the fact that the parity of the degree of a
graph is the same as the parity of the number of vertices and half-edges in the graph, i.e. that 
$(-1)^{|G_1|}=(-1)^{|V_1|+|H_1|}=(-1)^{|V_1|+|H_1|-2|(V_{1})_{w}|}$ for $(V_{1})_{w}$ the set of white vertices of $G_1$. 
Indeed, a vertex contributes with an odd degree
precisely when it has even valence, that is when the vertex plus its half-edges give an odd number. 

We are left to check that the terms of type (3) cancel in pairs.  
A ``bad'' newly created vertex has exactly two half-edges attached to it which are not from $G_1$: one from $G_2$ and one newly created
half-edge. Any
such graph occurs a second time as a term of type (3) with the role of these two edges exchanged and one checks that the signs cancel.

\smallskip 

Now $d(G_2\circ G_1)=d\lfloor G_2\hat \circ G_1\rfloor= \lfloor \hat d(G_2\hat \circ G_1)\rfloor$ by Lemma~\ref{floor}. 
Using the above calculation, we thus get 
 $d(G_2\circ G_1)=\lfloor G_2\hat\circ \hat d G_1\rfloor + (-1)^{|G_1|} \lfloor \hat d G_2\hat \circ G_1 \rfloor$. 
Now $\lfloor G_2\hat\circ \hat d G_1\rfloor  =G_2\circ dG_1$ as unlabeled leaves attached to trivalent black vertices of $G_1$ will still be attached
at trivalent black vertices in the composition, and those attached to higher valent black vertices or to white vertices will still be attached to
such.  We are left to check that 
$ \lfloor \hat d G_2\hat \circ G_1 \rfloor= dG_2\circ G_1$. In this case if $G_2$ has an unlabeled leaf at a trivalent black vertex of an incoming
cycle, there will be terms in $\hat d G_2\hat \circ G_1$ with this leaf is attached to a higher valent black vertex, namely the terms where leaves of
$G_1$ are attached at that vertex. These terms vanish in $ \lfloor \hat d G_2\hat \circ G_1 \rfloor$ and are not present in $dG_2\circ G_1$ as the
vertex is forgotten in $dG_2$. 

\medskip

We check associativity. Suppose $G_1,G_2,G_3$ are three composable graphs and consider the compositions $G_3\circ (G_2\circ G_1)$ and $(G_3\circ
G_2)\circ G_1$. The identifications of leaves representing open boundaries will be the same in both cases. For closed boundaries, one checks that each
term  in the first composition corresponds exactly to a term in the second composition and vice versa: The identification of start-leaves is fixed,
and the same in both cases. If we first remove the white vertices of $G_1$,
some leaves of $G_1$ might be attached to white vertices of $G_2$. When those white vertices are removed in the further composition with $G_3$, the
leaves of $G_1$ that were attached to a white vertex of $G_2$ will be attached in all possible ways, respecting their position in between leaves of
$G_2$, to the corresponding boundary circle in $G_3$. If we start by composing $G_2$ and $G_3$, the incoming boundary cycles of $G_2$ with white
vertices will become incoming boundary cycles of $G_3\circ G_2$ partially in the old $G_3$. Attaching now $G_1$ along such a boundary cycle, we see
exactly all the terms that occured before. Indeed, the leaves of $G_1$ will either be attached only to black vertices of the old $G_2$, or some of
them might be attached to vertices of $G_3$, in all possible ways, in between old edges of $G_2$ that where previously attached to a white
vertex. Left is to check that taking the underlying black-and-white graph gives the same result in both cases: If $G$ in the composition $G_2\circ
G_1$ satisfies $\lfloor G\rfloor=0$, then a start-leaf of $G_1$ was attached to a higher valence black vertex of $G_2$ or a white vertex of $G_2$. Any
graph $G'$ obtained from $G$ by further attaching $G_3$ will then have that start-leaf attached to a higher valence black vertex or a white vertex 
of the composed
graph, and hence also give 0.  On the other hand, if $G'$ in the composition $G_3\circ
G_2$ satisfies $\lfloor G'\rfloor=0$, then a start-leaf of $G_2$ was attached to a higher valence black vertex of $G_3$ or a white vertex of $G_3$,
and it will remain attached to such a vertex after any further gluing of $G_1$.
\end{proof}

Finally, we verify that the morphism complexes in $\OC$ do indeed compute the homology of the moduli space of open-closed Riemann cobordisms.

\begin{thm}\label{cos_thm}
$\OC(\oc{m_1}{n_1},\oc{m_2}{n_2})$  is the cellular complex of a space weakly homotopy equivalent to the disjoint union of coarse moduli
spaces\footnote{Here we employ again the convention that the moduli space of a disk with a single free boundary is a point, as is the moduli of an annulus with two free boundaries.} of Riemann surfaces of every genus, with

\begin{itemize}

\item $m_1+m_2$ labeled open boundary components,
\item $n_1+n_2$ labeled closed boundary components, 
\item any number of free boundary components (at least one per component with no open or incoming closed boundary)

\end{itemize}

Moreover, the composition of graphs defined above is compatible under this equivalence with the maps of moduli spaces induced by gluing surfaces along open and closed boundary components. 
\end{thm}

In particular, this theorem identifies the components of $\OC$ with the topological types of open-closed cobordisms.  

The equivalence to moduli space in the case $n_1=0$ was proved by Costello \cite[Prop.~6.1.3]{costello07} using certain flow on moduli space. 
Costello denotes this chain complex $\mathcal{G}$ in \cite{costello07} and describes it in terms of discs (corresponding here to black vertices) and annuli with marked
points (corresponding here to white vertices with start half-edges). The description in terms of graphs can be found in \cite{Cos07} for the
category $\OO$ after Proposition 2.2.3, though in \cite{Cos07} the white vertices are used to model punctures and do not have start half-edges. 

We extract this result instead from the work of Egas \cite{Ega15}, who proves the complete statement using, instead of Costello's flow,  a direct relationship between fat graphs and black and white graphs. 

\begin{proof}
Theorem  B of \cite{Ega15} says that the chain complex of black and white graphs has homology  $\coprod_S H_*(B\operatorname{Mod}(S))$, where $S$ runs over all open-closed cobordisms $S$ with at least one boundary which is not outgoing closed, and $\operatorname{Mod}(S)$ denotes the mapping class group of $S$. As $\operatorname{Mod}(S)$ has the homotopy type of the coarse moduli space, by the contractibility of Teichm{\"u}ller space, this yields the equivalence claimed. Now  \cite[Thm 4.41]{Ega15} shows that the composition of graphs induces the composition of moduli space induced by gluing surfaces.  
\end{proof}

Note also that Theorem 3.1 of \cite{Wah12} shows that $\OC$ identifies as a quasi-isomorphic 
subcategory of the category of formal operations on the Hochschild complex of $\OO$--algebras,
where a formal operation is defined as a natural transformation of the iterated generalized Hochschild complex functor, and composition is
composition of natural transformations. This shows that the gluing defined here also identifies with the composition of the universal formal operations on the Hochschild
complex of $\OO$--algebras.

\subsection{Annuli}\label{annulisec}

We introduce in this section a chain complex of annuli that will play an important role in our definition of the Hochschild complex: tensoring with this chain complex will be used to give the degree shift in the Hochschild complex and define the Hochschild differential. This will facilitates both keeping track of the signs, and making sure the actions we define are given by chain maps. 

This chain complex of annuli  is closely related to the annular part of the category $\mathcal{D}$ described in \cite[Sec 6]{costello07}.  Further, it resembles a chain complex used in \cite[Sec 11.2]{KS06} to describe the action of the Hochschild cochains of an algebra on its Hochschild chains. We will however {\em not} consider the Hochschild cochains in this paper, and the complex of annuli used here should rather be thought of as simply modeling the map 
$A^{\ot m}\to C_*(A,A)=\bigoplus_{n\ge 1}A^{\ot n}$. 

\medskip

Let $\OC_A(\oc{m}{0},\oc{0}{1})\subset \OC(\oc{m}{0},\oc{0}{1})$ 
denote the component of the annuli with $m$ open incoming boundaries on one side, and one closed outgoing boundary on the
other side. Each generating graph in this chain complex is build from a white vertex (the outgoing circle) by attaching trees, with possibly one unlabeled leaf as start
half-edge for the white vertex. 
Inside this chain complex, we can consider the sub-chain complex $\mathcal{L}(m)$ of graphs with no unlabeled leaf. 
\begin{figure}[h]
\begin{lpic}{Ann(0.7,0.7)}
\end{lpic}
\caption{Annulus}\label{Ann}
\end{figure}
Figure~\ref{Ann} shows an example of an graph in $\mathcal{L}(12)$. 
Let $L_n=\lgl l_n \rgl$ denote the free graded $\Z$-module on a single generator  in degree $n-1$, the graph $l_n$ of Figure~\ref{graphsex}(c). 
By cutting the graphs around their white vertex, the complex $\mathcal{L}(m)$  can be described as
$$\mathcal{L}(m)=\bigoplus_{n\ge 1}\Ai(m,n)\ot L_n$$
with differential $d=d_{\Ai}+d_L$, where $d_{\Ai}$ is the differential of $\Ai$ and 
$$d_L:\Ai(m,n)\ot L_{n}\rar \Ai(m,n)\ot \bigoplus_{1\le k< n} \Ai(n,k)\ot L_{k}\rar \bigoplus_{1\le k< n} \Ai(m,k)\ot L_{k},$$
where the first map takes the differential of $l_n$ in $\OC(\oc{m}{0},\oc{0}{1})$ and reads off the blown-up graphs as elements of $\Ai(n,k)\ot L_{k}$
for various $k<n$, and the second map is composition in $\Ai$. 

We will use the notation $$d_L(l_n)=\sum_{k<n}f_{n,k}\ot l_k\in \bigoplus_{1\le k< n} \Ai(n,k)\ot L_{k}$$ 
for this decomposition of the differential of $l_n$ in $\OC$.

\subsection{The category of Sullivan diagrams $\SD$}\label{SDsec}
Sullivan chord diagrams are usually defined as fat graphs built from a disjoint union of circles by attaching chords, or trees, which should be
thought of as ``length 0'' edges, in such a way that the original circles are still cycles in the resulting graph. One has to be aware that authors
sometimes restrict to {\em non-degenerate diagrams}, those such that collapsing the chords does not change the homotopy type of the
graph (as in for example \cite{Cha05,CohGod,KaufI}). There are also marked and unmarked versions, and there can be variations in the way the markings
are handled (as in e.g.~\cite{PoiRou}). We consider
here a chain complex of general Sullivan diagrams, also degenerate ones. Our definition is in the spirit of \cite{ChaSul04} and
 agrees with that of \cite{Tradler-Z} as well as the normalized ``closed'' Sullivan diagrams of \cite{Kau10}. 
Such Sullivan diagrams  model a harmonic compactification of moduli space (see \cite{EgaKup}).  

We start the section by giving a formal definition of Sullivan diagrams, relate it to the informal definition above, and build a chain complex
of such diagrams. Then we will show that
Sullivan diagram can be identified with a quotient complex of black and white graphs, which is the way they occur in the present paper. 
This will enable us to define an ``open-closed'' category $\SD$ of
Sullivan diagrams directly as a quotient of the category $\OC$. 
We will then prove a few facts about the map quotient $\OC\to \SD$, as well as 
explain how our category of Sullivan diagrams relates to the one defined in \cite{Kau10}. 

\medskip

We call a fat graph {\em $p$--admissible} (in the spirit of \cite{godin07}) if $p$ of its boundary cycles are disjoint embedded circles
in the graph. We call these $p$ special cycles {\em admissible cycles} and represent them as round circles when drawing such a graph.  
\begin{Def}\label{SDdef}
An {\em (oriented) $\oc{m}{p}$--Sullivan diagram} is an equivalence class of (oriented) 
$p$--admissible fat graphs with $p+m$ leaves, where the first $p$
leaves are distributed as 1 per
admissible boundary cycle and the remaining $m$ leaves lie in the other cycles. Two such graphs $G_1,G_2$ are equivalent if they are
connected by a zig-zag of edge collapses between $p$--admissible fat graphs, collapsing edges which are not in the $p$ admissible cycles, for edge collapses as defined in \ref{col-blow}. 
(Figure \ref{SDc} shows four equivalent $\oc{2}{1}$--Sullivan diagrams.)
\end{Def}
\begin{figure}[h]
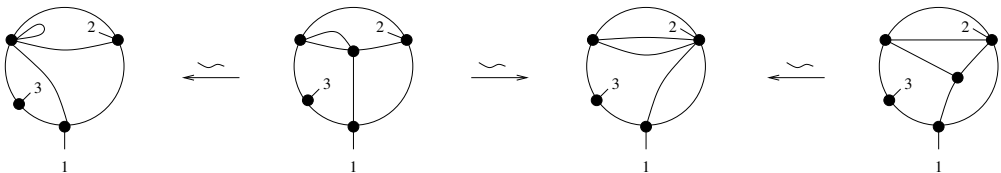

\begin{lpic}{SDclassic2(0.5,0.5)}
\end{lpic}
\caption{Equivalent Sullivan diagrams, with one admissible cycle which in each case is the outside of the round circle}\label{SDc}
\end{figure}

To a fat graph, one can associate a surface by thickening the graph. As edge collapses respect the topological type of the associated surface, a
Sullivan diagram still has an associated topological type. 

\smallskip

For a  $\oc{m}{p}$--Sullivan diagram $G$, we let $E_a$ denote the set of edges lying on the admissible cycles of $G$. 
The {\em degree} of $G$ is then defined as 
$$\deg(G)=|E_a|-p.$$ 
For example, the Sullivan diagrams of Figure~\ref{SDc} are of degree $4-1=3$, while the left picture in Figure~\ref{SDequiv} is a
Sullivan diagram of degree $6-2=4$. 

Let $\oc{m}{p}-SD$ denote the chain complex generated as a graded $\Z$--module by 
all oriented $\oc{m}{p}$--Sullivan diagrams, modulo the relation that $-1$ acts by reversing the orientation. The 
boundary map in $\oc{m}{p}-SD$ is defined on generators by 
$$dG=\sum_{e\in E_a} G/e,$$ 
the sum of all collapses of $G$ along edges in the admissible boundary cycles\footnote{It is worth noting that we have already effectively collapsed
  the remaining edges by the equivalence relation.} of $G$, with $G/e$ defined in \ref{col-blow}.

This chain complex is isomorphic to the complex {\em Cyclic Sullivan Chord
  Diagrams} considered by Tradler-Zeinalian in \cite[Def 2.1]{Tradler-Z}. 
Their diagrams are build from disjoint circles (our admissible cycles) by attaching trees (the {\em chords} or non-admissible edges), 
whereas we allow the chords to be unions
of graphs.  However, collapsing non-admissible edges, one can alway push the vertices of the representing graph to only lie on the admissible cycles. Hence a
Sullivan diagram in our sense is always equivalent to one as in \cite{Tradler-Z} which is a union of admissible cycles together with chords which
are edges attached directly to the cycles.
The equivalence relation in \cite{Tradler-Z} corresponds this way to the one defined here. 

\begin{rem}\label{cactirem}{\rm
The chain complex of $\oc{1}{p}$--Sullivan diagrams of topological
type a surface of genus 0 with $p+1$ boundary components is a cellular
complex for the $p$th space of the normalized cactus operad
\cite{Vor05,Kau05}. Indeed, such a Sullivan diagram is made out of $p$
circles attached to each other in a  tree-like fashion, exactly
representing a cell in the normalized cactus operad. (See Figure \ref{cacti}.) This statement can also be found in \cite[Sec 3.1.1]{KauSch10} or \cite{Kau07} in the ``spineless'' version (corresponding to not having start half-edges at the white vertices in our language), and \cite{War12} with the spines, both in terms of {\em bipartite black and white trees}, which give a slightly different description of the same chain complex. (See also \cite{Kau05}.) 
\begin{figure}[h]
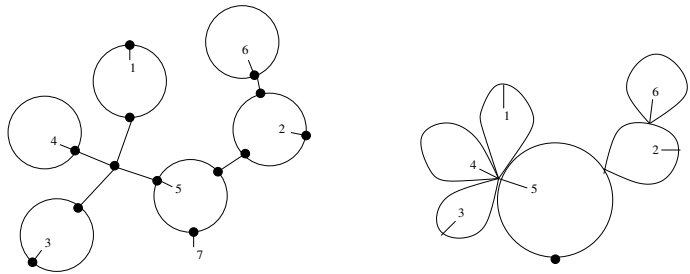

\begin{lpic}{cacti(0.4,0.4)}
\end{lpic}
\caption{A Sullivan diagram with 6 admissible cycles modelling a 6-lobed cactus.}\label{cacti}
\end{figure}
}\end{rem}

The following theorem relates Sullivan diagrams to black and white graphs. In the proof, 
we will use the equivalence relation in Sullivan diagrams in the opposite way from what we used to relate our definition to that of \cite{Tradler-Z}: we will represent
Sullivan diagrams by the graphs
in their equivalence class with the maximum number of vertices {\em not} on the admissible cycles.

\begin{thm}\label{SDequivthm}
The chain complex $\oc{m}{p}-SD$ of $\oc{m}{p}$--Sullivan diagrams is the quotient of  $\oc{m}{p}-Graphs$ by the graphs with black vertices of valence
at least 4 and by the boundaries of such graphs.
\end{thm}

\begin{proof}
We first note that every Sullivan diagram can be represented by a graph with only trivalent vertices, except for the
vertices where the leaf of an admissible cycle is attached, which may be 4-valent. Indeed, if the graph has a
higher valence vertex away from the admissible cycles, one can blow it up in any manner one likes and obtain an
equivalent graph with trivalent vertices replacing the higher valence vertex. If there is a higher valence vertex 
on an admissible cycle, it has exactly 
two contiguous half-edges of that admissible cycle attached to it, unless the leaf of the admissible cycle is at that
position, in which case it has three such.  Any blow-up of that vertex 
which keeps the half-edges of the admissible cycle together produces
an equivalent graph with the property we want. 
For the purpose of the proof, we call such graphs {\em
  essentially trivalent}. 

Two essentially trivalent Sullivan diagrams are equivalent if and only if they are
equivalent through such Sullivan diagrams and diagrams 
with exactly one 4-valent vertex which is away from the cycles: a single valence 4 vertex at a time
suffices since we can do collapses and blow-ups one at a time, 
and no additional valence
4 (or 5) vertices on the admissible cycles are necessary because there is only one way of blowing-up
such a vertex if the two (or three) half-edges of the admissible cycles have to stay together, up to collapses and blow-ups away from
the admissible cycle. 

Given an essentially trivalent $\oc{m}{p}$--Sullivan diagram, we get a $\oc{m}{p}$--graph 
by collapsing the admissible cycles to white vertices. If the leaf of the $i$th admissible cycle is at a 3-valent
vertex, we place an unlabeled start-leaf at that position on the $i$th white vertex, and if it is at a 4-valent vertex, we remove
it and define the remaining half-edge after the collapse to be start half-edge. (See Figure~\ref{SDequiv} for an example.) 
\begin{figure}[h]
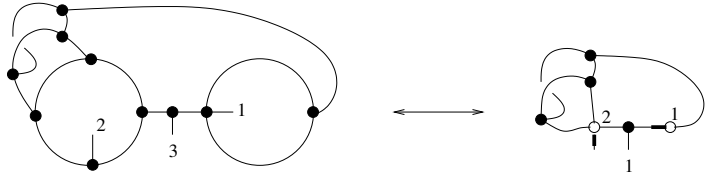

\begin{lpic}{SDequiv.2(0.5,0.5)}
\end{lpic}
\caption{Essentially trivalent Sullivan diagram (with admissible cycles the inside of the round circles) and the corresponding black and white graph}\label{SDequiv}
\end{figure}

Given a $\oc{m}{p}$--graph, one can similarly  obtain an
essentially  
trivalent $\oc{m}{p}$--Sullivan diagram by expanding the white vertices to circles and placing leaves at the spots corresponding
to start-edges.  These two maps are inverses of one  another, and the equivalence relations agree under
the maps by the above remarks. 

Note moreover that the degrees agree: 
the degree of a $\oc{m}{p}$--graph $G$ is $\sum_{v\in V_b}|v|-3 + \sum_{v\in V_w}|v|-1$. As all black vertices of the graphs occurring here are
trivalent, the first sum gives 0. On the other hand, the valence of a white vertex in $G$ is the number of admissible edges on the corresponding
admissible cycle of the associated Sullivan diagram.

We are left to check that the boundary maps also agree. Given an essentially trivalent Sullivan diagram, the boundary
map in $\oc{m}{p}$--SD is a sum of Sullivan diagrams, each  with a higher valence vertex on an admissible
cycle. Blowing up that vertex in the only possible manner to obtain an essentially trivalent graph corresponds exactly under the equivalence above to
a term in the differential of the associated  $\oc{m}{p}$--graph, and all the terms of the differential of this $\oc{m}{p}$--graph that do not have
valence 4 or more black vertices will occur this way.  
\end{proof}

Recall from \ref{OCsec} that the open-closed category $\OC$ has objects pairs of natural numbers $\oc{m}{n}$ and morphisms complexes 
$\OC(\oc{m_1}{n_1},\oc{m_2}{n_2})\subset\oc{n_1+m_1+m_2}{n_2}-Graphs$, the subcomplex of graphs with the first $n_1$ leaves alone in their boundary cycles. 
As composition in $\OC$ can only increase the valence of black vertices, it still gives a  well-defined composition when quotienting out by 
the graphs with black vertices of valence $4$ or more. 
Hence, using the above theorem, we can simply  define the category of Sullivan diagrams as a quotient category of $\OC$:

\begin{Def}

Let $\SD$ be the category with objects pairs of natural numbers $\oc{m}{n}$, with $m,n\ge 0$, and morphisms from $\oc{m_1}{n_1}$ to 
$\oc{m_2}{n_2}$ the quotient of $\OC(\oc{m_1}{n_1},\oc{m_2}{n_2})$ by the graphs having black vertices of
valence higher than 3 and by the boundary of such graphs.

\end{Def}

Note that in terms of ``classical'' Sullivan diagrams, as in Definition~\ref{SDdef}, admissible cycles are considered here as outgoing boundary
circles, while incoming boundary circles are ordinary cycles in the graph. The composition of Sullivan diagrams $G_1,G_2$ is  defined in classical terms by gluing the $i$th
admissible cycle $G_1$ to the $i$th incoming cycle of a graph $G_2$ by attaching the edges which had boundary points on this admissible cycle of
$G_1$ to edges of admissible cycles of $G_2$ lying on its $i$th incoming cycle, in all possible way respecting the cyclic ordering. This is because,
in terms of black and white graphs, composition is defined by attaching the edges of the first graph at vertices of the second along the corresponding
cycle in all possible ways, but attaching edges at black vertices creates vertices of valence 4 or higher, and hence is trivial in Sullivan
diagrams. On the other hand, attaching edges at white vertices corresponds to attaching at admissible edges in the classical picture.

\smallskip

By definition, we have a quotient functor 
$$\pi:\OC\to\SD$$
A direct consequence of  Theorem~\ref{SDequivthm} is the following:

\begin{prop}
The quotient functor $\pi\colon\OC\to\SD$ induces an isomorphism 
$$H_0(\OC(\oc{m_1}{n_1},\oc{m_2}{n_2}))\cong H_0(\SD(\oc{m_1}{n_1},\oc{m_2}{n_2}))$$
for each $n_1,m_1,n_2,m_2$.  
\end{prop}

From Remark~\ref{cactirem}, we have in addition that the component of $\SD(\oc{0}{1},\oc{0}{p})$ of Sullivan diagrams of underlying topological type a genus 0 surface with $p+1$ boundary components is quasi-isomorphic to the same component in $\OC(\oc{0}{1},\oc{0}{p})$, as both have homology that of the framed little discs, i.e.~the $p$th component of the BV operad. As is to be expected, the map $\pi$ on this component is a quasi-isomorphism.

\begin{prop}\label{non-vanish}
On the component of a surface of genus 0 with $p+1$ boundary components, the map $\pi:\OC(\oc{0}{1},\oc{0}{p})\to \SD(\oc{0}{1},\oc{0}{p})$  is a quasi-isomorphism. 
\end{prop}

This result is closely related to  \cite[Thm 6.6]{War12} and \cite[Prop 3.15]{KauSch10} is its spineless version. As this result is easiest proved once our machinery is set-up, we postpone its proof to Section~\ref{strict_section} at which time we will have all the ingredients in place.

\medskip

In general, the map $\pi$ is however not a quasi-isomorphism. 
Given a topological type of surface, the chain complex of Sullivan diagrams of that topological type computes the homology of a certain harmonic compactification of moduli spaces of Riemann surfaces of that type. More precisely, $\SD$ is a cellular chain complex for a space of metric Sullivan diagrams $SD$ (see \cite[Def 3.16]{EgaKup}) and the following holds for this space

\begin{thm}\label{compactification}
$SD$ is homeomorphic to the unimodular harmonic compactification of moduli space and the map $\OC\to \SD$ models the inclusion of moduli space in its compactification. 
\end{thm}

\begin{proof} As in \cite{Ega15,EgaKup}, let $\mathcal{F}at^{ad}$ denote the category of admissible fat graphs under edge collapses and $\mathcal{MF}at^{ad}$ the space of metric admissible fat graphs. 
Then the map in the statement identifies as a chain version of the map $|\mathcal{F}at^{ad}|\arsim \mathcal{MF}at^{ad}\to SD$ in \cite[Thm 1.1]{EgaKup} under the identification 
\cite[Thm 4.38]{Ega15} of $\OC$ as a chain model for $|\mathcal{F}at^{ad}|$. 
Indeed, Theorem 4.38 of \cite{Ega15} uses a filtration of $|\mathcal{F}at^{ad}|$ by mixed degree (see Definition 4.13 in that paper) to show that black and white graphs define a quasi-cell decomposition of $|\mathcal{F}at^{ad}|$. The mixed degree is also well-defined for Sullivan diagrams seen as equivalence classes of admissible fat graphs (as in our Definition~\ref{SDdef}), and the quasi-cell decomposition becomes an actual cell decomposition by $\SD$ in the case of Sullivan diagrams. The quotient map from admissible fat graphs to Sullivan diagrams respects these decompositions and is, in terms of the decomposition, the map in the statement. 
\end{proof}

As we will see below, the case described in Proposition~\ref{non-vanish} above is rather particular, and in fact this map annihilates a large part of the known homology of moduli space in positive genus.  Despite this, this map stays of main interest to us both because it is a compactification, and because it will play a role in the present paper when studying the action of the homology of moduli space on the Hochschild complex of strict Frobenius algebras (see Corollary \ref{factor_action_cor}). 
Before getting to our vanishing result, we first given an example that illustrates that this map is also not be surjective in homology in general.

\medskip

The chain complex of
Sullivan diagrams is a lot smaller than that of all fat graphs, or all $\oc{m}{p}$--graphs, and hence computations of its homology are
more approachable. It is for example not hard to compute that the component of the pair of pants in $\SD(\oc{0}{2},\oc{0}{1})$ is a complex that computes the
homology of $S^3\x S^1$. The corresponding component of $\OC$ computes the homology of the framed disk operad $fD(2) \simeq S^1\x S^1\x S^1$.  
The map $\OC\to\SD$ in this case is induced by 
the canonical embedding of the first two $S^1$--factors as a standard torus in the 3-sphere.  

Note that Sullivan diagrams are more fundamentally asymmetric in their inputs/outputs than black and white graphs. Indeed, from Proposition~\ref{non-vanish} we have that  
$\SD(\oc{0}{1},\oc{0}{2})\simeq\OC(\oc{0}{1},\oc{0}{2})$, and hence
$$H_*(\SD(\oc{0}{1},\oc{0}{2}))\cong H_*(\OC(\oc{0}{1},\oc{0}{2})) \cong H_*(\OC(\oc{0}{2},\oc{0}{1})) \not\cong  H_*(\SD(\oc{0}{2},\oc{0}{1})).$$

We refer to \cite[Sec 4]{Wah12} and \cite[Paper C]{Ega14Abis} for further computations of homology of Sullivan diagrams.

\smallskip

The fact that $\SD$ is a (quite drastic) quotient of $\OC$ makes one expect that, in homology, the projection $\pi$ kills many classes. We make this precise by analysing the map componentwise: denote by 
$$\pi_S:\OC_S(\oc{m_1}{n_1},\oc{m_2}{n_2}))\rar \SD_S(\oc{m_1}{n_1},\oc{m_2}{n_2}))$$ 
the functor $\pi:\OC\to\SD$ restricted to the component of morphisms of type $S$,
where $S$ is a generator in $H_0\OC_S(\oc{m_1}{n_1},\oc{m_2}{n_2}))\cong H_0\SD_S(\oc{m_1}{n_1},\oc{m_2}{n_2}))$, i.e.~a topological type of cobordism. 

We have the following general vanishing result: 

\begin{prop}\label{strictvanish}
Suppose $m_1+m_2+n_1>0$ and $S$ is a generator of $H_0\OC(\oc{m_1}{n_1},\oc{m_2}{n_2}))$ which is a connected surface of genus $g$. 
Then there exists
$S'\in H_0\OC(\oc{m_1}{n_1},\oc{m_2+1}{0}))$ and a
map $f:\OC_{S'}(\oc{m_1}{n_1},\oc{m_2+1}{0}))\to \OC_S(\oc{m_1}{n_1},\oc{m_2}{n_2}))$ which is an isomorphism in homology in
degrees $*\le\frac{2g}{3}$ and such that the image of the composition
$$\OC_{S'}(\oc{m_1}{n_1},\oc{m_2+1}{0}))\sta{f}{\rar} \OC_S(\oc{m_1}{n_1},\oc{m_2}{n_2}))\sta{\pi}{\rar} \SD_{S}(\oc{m_1}{n_1},\oc{m_2}{n_2}))$$
is concentrated in
degree 0. In particular, the stable classes of positive degree map to 0 under the map 
 $H_*(\pi):H_*(\OC)\to H_*(\SD)$. 
\end{prop}

Here by a {\em stable class}, we mean a class in that lives in 
$H_*(\OC_S(\oc{m_1}{n_1},\oc{m_2}{n_2}))$ in the range of degrees $*\le\frac{2g}{3}$ for $g$ the
genus of the component of lowest genus in $S$. The terminology {\em stable} is justified by the fact that the map 
$H_*(\OC_S(\oc{m_1}{n_1},\oc{m_2}{n_2}))\to H_*(\OC_{P\circ S}(\oc{m_1}{n_1},\oc{m_2}{n_3}))$ 
induced from composition in $\OC$ with a chosen element $[P]\in H_0(\OC_T(\oc{m_2}{n_2},\oc{m_2}{n_3}))$ for $P$ a pair of pants (union some identities) glued along one of two circles, 
induces an isomorphism in this range of degrees. 
This is a consequence of Harer's stability theorem  (\cite{Har85},
with the improved range of \cite{Bol09,RW09} and \cite{Wah13} for punctures in $S$). 
Indeed, $\OC_S(\oc{m_1}{n_1},\oc{m_2}{n_2})$ is a chain complex computing the homology of the mapping class groups $\operatorname{Mod}(S):=\pi_0\Dif(S\ {\rm rel}\ \del_{in}\cup \del_{out})$ and composition with $[P]$ corresponds to the map induced on the homology of the mapping class groups by gluing $P$ and extending the diffeomorphisms $S$ to $P\circ S$ by the identity on $P$, which is the way the stabilization maps are classically defined.

The map $f$ in the proposition in the case $m_1=m_2=n_1=0$ is a little more subtle. An analogous statement can be made though using in place of $f$ a map that replaces
a fixed boundary by a free boundary, which is not an isomorphism in homology stably. 

\begin{proof}
The idea of the proof is as follows: Sullivan diagrams have their non-zero degree concentrated at white vertices, i.e.~outgoing closed boundaries, as black vertices in Sullivan diagrams can only be of valence 3. Now any surface with $n_2$ outgoing closed boundary components and $b+1$ other fixed boundary components can be constructed from a surface with $b+1$ boundary components by attaching an $n_2$--legged pair of pants. Homological stability says that adding such an $n_2$--legged pair of pants {\em in degree 0} induces an isomorphism in homology of the corresponding moduli spaces. Hence any homology class in the stable range can be written ``without using white vertices'' and hence must map to zero in Sullivan diagrams unless it is of degree 0. 

We make this argument precise now. 

Suppose first that $m_1+m_2>0$. Then $S'$ can be obtained from $S$ by gluing discs on the $n_2$ closed outgoing boundaries of $S$
and adding a open outgoing boundary on a boundary component containing some other open boundary. We can reconstruct the topological
type of $S$ from
$S'$ by gluing a $n_2$--legged pair of pants $P$ along an open boundary as shown in Figure \ref{SiP}. 
\begin{figure}[h]
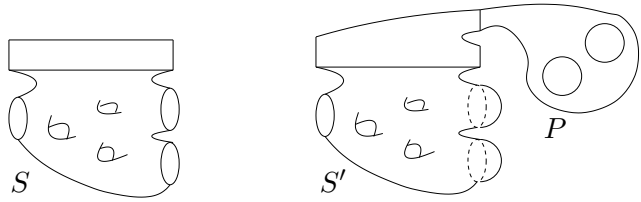

\begin{lpic}{SiP(0.4,0.4)}
\lbl[b]{4,2;$S$}
\lbl[b]{107,2;$S'$}
\lbl[b]{180,20;$P$}
\end{lpic}
\caption{The surfaces $S$ and $P\circ S'\cong S$}\label{SiP}
\end{figure}
Choosing a degree 0 representative of $P$ in $\OC(\oc{1}{0},\oc{0}{n_2}))$, the map $f$ above is just induced by composition with $P$ in
$\OC$. The fact that it is an isomorphism in homology in the given range follows from the stability theorem of \cite{Bol09,RW09} (and \cite{Wah13} for a version with punctures). Indeed, a neighborhood of $P$ union the boundary component of $S'$ it is attached to is an $(n_2+1)$--legged pair of pants. As the boundary of $S'$ that $P$ is attached to has a fixed interval, it may as well be assumed to be completely fixed. Then attaching $P$ along an interval is seen to be equivalent to attaching an $(n_2+1)$--legged pair of pants along the whole boundary, which is a composition of stabilization maps. 

The fact that the composition $\pi_S\circ f$ lands in degree 0 follows from the following two facts: \\
1) the diagram 
$$\xymatrix{ \OC(\oc{m_1}{n_1},\oc{m_2+1}{0}))\ar[rr]^-{f=P\circ\ - } \ar[d]^\pi & & \OC(\oc{m_1}{n_1},\oc{m_2}{n_2}))\ar[d]^\pi\\
\SD(\oc{m_1}{n_1},\oc{m_2+1}{0}))\ar[rr]^-{\pi(P)\circ\ -}&& \SD(\oc{m_1}{n_1},\oc{m_2}{n_2}))
}$$ commutes as $\pi$ is a functor, \\
2) the complex $\SD(\oc{m_1}{n_1},\oc{m_2+1}{0}))$ is concentrated in degree 0 as graphs in this complex have no white vertices. 

For $n_1>0$, we have an isomorphism $\OC_S(\oc{m_1}{n_1},\oc{m_2}{n_2}))\cong \OC_{\bar S}(\oc{m_1+1}{n_1-1},\oc{m_2}{n_2}))$,
and similarly for $\SD$, for $\bar S$ the surface obtained from $S$ by replacing an incoming closed boundary by an incoming open
boundary, alone on that component. This reduces the case $m_1+m_2=0$ with $n_1>0$ to the previous one. 
\end{proof}

We end by mentioning an alternative definition of Sullivan diagrams.

\begin{rem}[Sullivan diagrams as arcs in a surface]\label{KaufSD}{\rm
In \cite[Sec 2.3, 6.3]{Kau10}, Kaufmann defines an open-closed category of {\em arc families of  Sullivan types} $\operatorname{Sull}_1^{\rm c/o}$ (with the ``1'' indicating a normalized version). We explain here how this category relates to the category  
 $\SD$ we defined in this section.  Briefly, we have that the cellular chain complex of $\operatorname{Sull}_1^{\rm c/o}$ is isomorphic to a subcategory of 
$\SD$ when restricted to the closed part and switching the role of
incoming and outgoing boundaries. We explain this in more details now. 

For simplicity, we restricted to the case of a single brane. A {\em windowed surface} $F$ with $m_1$ (resp.~$n_1$) incoming open (resp.~closed) boundaries and $m_2$ (resp.~$n_2$) outgoing open (resp.~closed) boundaries is a surface with $m_1+m_2+n_1+n_2$ marked points in its boundary such that the last $n_1+n_2$ are alone on their boundary component, together with a labelling as $n_1$ ``in'' and $n_2$ ``out'' of those boundaries and of $m_1$ ``in'' and $m_2$ ``out'' of the arcs {\em in between} the other marked points in the boundary of $F$.    
An arc family of Sullivan type with $m_1$/$n_1$ incoming open/closed boundaries and $m_2$/$n_2$ outgoing open/closed boundaries is then defined as a weighted collection of arcs in such a surface $F$, where the arcs start at the incoming boundaries and end at the outgoing boundaries, and
such that the sum of the weights at each incoming boundary is equal to 1. (See the first picture in Figure~\ref{KaufSDpic} for an example with one incoming and two outgoing closed boundaries). These arc systems are considered modulo the action of the mapping class group of $F$. 
When $m_1=m_2=0$, one can reconstruct a classical metric Sullivan diagram (an element of the space $SD$ of Theorem~\ref{compactification}) from such a
collection of arcs by having a circle for each incoming closed boundary with a edge of length the associated weight for each arc starting at that
boundary. The chords are then obtained by choosing a fat graph representative of each component of the surface $F$ cut along the arcs, with a leaf for each of the marked points.  (See
Figure~\ref{KaufSDpic} for an example.) A collection of $k_1+\dots+k_{n_2}$ arcs corresponds to a cell $\De^{k_1-1}\times\dots\times\De^{k_{n_2}-1}$
represented by a diagram of degree  $k_1+\dots+k_{n_2}-n_2$ in our definition above. The boundary of the cells in the arc description is defined by forgetting arcs. This corresponds to gluing surfaces in the complement of the arcs, which corresponds to collapsing an edge in the classical description. 
\begin{figure}[h]
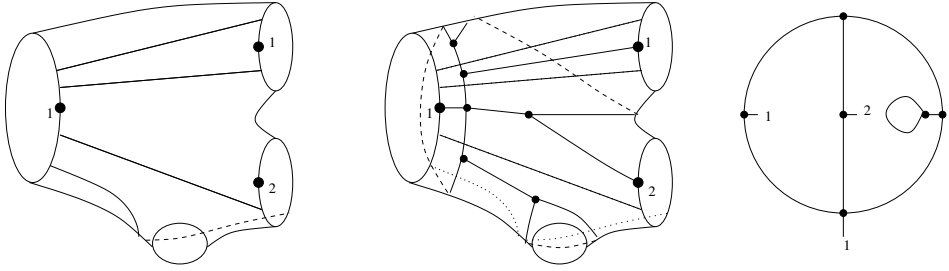

\begin{lpic}{KaufSD(0.45,0.45)}
\end{lpic}
\caption{A Sullivan diagram described by a system of arcs in a surface, the dual Sullivan diagram in the surface, and the same Sullivan diagram without the surface}\label{KaufSDpic}
\end{figure}

This shows how to get from an arc family to a Sullivan diagram in the case of closed boundaries. The map is injective but not quite surjective. Indeed,
our closed Sullivan diagrams are slightly more general in that, in terms of arc families, we allow arcs that end at unmarked boundary components. 
In Section~\ref{KP_section},  we will discuss the full category of open-closed arc families of Sullivan type, in their relation to the Kaufmann-Penner model of open-closed string interactions. 
}\end{rem}

\section{Algebras}\label{Algsec}

In this section, we describe the main types of algebras we will consider in the present paper. We use the formalism of props of MacLane \cite[\S
24]{McL65}, and describe
algebraic structures via symmetric monoidal functors from given symmetric monoidal categories: recall that a prop, {\em product and permutation
  category}, in the category $\Comp$ 
is a symmetric monoidal dg-category with objects the natural numbers, and an algebra over that prop is a symmetric monoidal functor from that category
to $\Comp$. We describe in this section the main props we will use, and give descriptions of  their algebras. (A good introductory reference for props
and operads is \cite{Val12}).

\smallskip

If $\e$ is a symmetric monoidal category and $\Phi:\e\to\Comp$ is a functor, we say that $\Phi$ is {\em symmetric monoidal} if there are maps 
$\Phi(n)\ot\Phi(m)\to\Phi(n+m)$ natural in $n$ and $m$ and compatible with the symmetries of $\Comp$ and $\e$. 
We say that $\Phi$ is {\em split monoidal} if these maps are isomorphisms and
{\em h-split}  if they are quasi-isomorphisms.

\subsection{$\Ai$--algebras}\label{Aialg}
Recall from \ref{Aisec} the symmetric monoidal dg-category $\Ai$. This category is freely generated as a symmetric monoidal category by the morphisms from
$k$ to $1$, for each $k\ge 2$, represented by a tree (or rather a {\em corolla}) 
$m_k$ of degree $k-2$ with a single vertex with $k$ incoming and $1$ outgoing leaves. 
A symmetric monoidal functor $$\Phi:\Ai\to\Comp$$ corresponds precisely to
giving an $A_\infty$--structure on $\Phi(1)$ with multiplication and higher multiplications 
$$\mu_k:\Phi(1)^{\ot k}\to \Phi(k)\sta{\Phi(m_k)}{\rar} \Phi(1)$$ 
for each $k\ge 2$, where the first map uses the monoidal structure of $\Phi$. 
The fact that this defines an $\Ai$--structure comes from the fact
that planar, or equivalently ``fat'' trees
define a cellular decomposition of Stasheff's polytopes. See for example
\cite[C.2, 9.2.7]{LodVal}.

There is an additional generating map
$u:0\to 1$ of degree 0 in the category $\Ai^+$, a singly labeled outgoing leaf, which behaves as a unit for the multiplication $\mu_2$. 
So if $\Phi$ extends to a symmetric monoidal functor with source $\Ai^+$, the $\Ai$--algebra $\Phi(1)$ is equipped with a unit 
for the multiplication $\mu_2$.  This is what we will mean by a {\em unital $\Ai$--algebra}.

More generally, we will consider in this paper symmetric monoidal dg-categories $\e$ equipped with a symmetric monoidal functor $i:\Ai\to\e$, so that 
$\e$--algebras, i.e.~symmetric monoidal functors $\e\to\Comp$ have an underlying $\Ai$--algebras by precomposition with $i$. 
We will call such a pair $(\e,i)$ a {\em prop with $\Ai$--multiplication}. If $\e$ admits a functor $i:\Ai^+\to\Comp$, we call the pair $(\e,i)$ a 
{\em prop with unital $\Ai$--multiplication}.

\subsection{Frobenius and $\Ai$--Frobenius algebras}\label{Frobsec}

By a {\em symmetric Frobenius algebra}, or just {\em Frobenius algebra} for short, 
we mean a dg-algebra with a non-degenerate symmetric pairing. A Frobenius algebra can alternatively be defined as
a chain complex with is a unital algebra and a counital coalgebra, such that the multiplication and coproduct satisfy the Frobenius identity:
$$\nu(ab) = \sum_i a_i' \otimes a_i'' b = \sum_j a b_j' \otimes b_j''$$
where $a, b$ are elements of the algebra, $\nu$ is the coproduct, $\nu(a) = \sum_i a_i' \otimes a_i''$, and $\nu(b) = \sum_j b_j' \otimes b_j''$. 
(See \cite[2.2,2.2.9,2.3.24]{Koc04} for the various equivalent definitions of (symmetric) Frobenius algebras.)

The cohomology of a closed manifold is an example of a Frobenius algebra, though with a pairing of degree $-d$ for $d$ the dimension of a manifold. 
Because of this, Frobenius algebras are sometimes called a {\em Poincar\'e duality
algebra} (see e.g.~\cite[Def 2.1]{lambrechts_stanley} in the commutative setting). 

Recall from \ref{OO} the open cobordism category $\OO$ with objects the natural numbers and morphisms the chain complexes of moduli spaces of open
cobordisms. We denote by $H_0(\OO)$ the dg-category with the same objects but with morphisms from $n$ to $m$ concentrated in degree 0, given by 
$H_0(\OO)$. In other words, the morphisms from $n$ to $m$ is the free module on the topological types of cobordisms from $n$ to $m$
intervals. Corollary 4.5 of \cite{LauPfe} says that split symmetric monoidal functors $\Phi:H_0(\OO)\to\Comp$ are in 1-1 correspondence with symmetric
Frobenius algebras.  We note that split monoidality is in some sense an analogue the assumption of \emph{cyclicity} for algebras over cyclic operads, since $H_0(\OO)$ has a built-in cyclic symmetry.

\smallskip

Replacing $H_0(\OO)$ in the above by the original open cobordism category $\OO$, we get an $\Ai$--version of Frobenius algebras: 
We call a split symmetric monoidal
functor $$\Phi: \OO \to \Comp$$ (or by abuse of language its value at 1, $\Phi(1)$) an
\emph{$\Ai$--Frobenius algebra}.  If $\Phi$ is  $h$--split, $\Phi$ could be
called an \emph{extended $\Ai$--Frobenius algebra}, following
\cite[7.3]{costello07}.  In either case, note that by restriction along $i: \Ai \to \OO$,
$\Phi$ equips $\Phi(1)$ with the structure of an $\Ai$--algebra (in fact an $\Ai^+$--algebra).  

In addition to the $\Ai$--structure, the morphism $tr:1 \to 0$ in $\OO$ given by a single incoming
labeled leaf (the mirror of the unit $u$) gives a map $tr:\Phi(1) \to \Phi(0)$.  
When $\Phi$ is $h$--split, $\Phi(0)$ is quasi-isomorphic to $\Z$
(concentrated in degree 0).  The map induced by the trace in homology
$$tr: H_*(\Phi(1)) \to H_*(\Phi(0)) = \Z,$$
which, along with the associative multiplication coming from the $\Ai$ structure,
equips $H_*(\Phi(1))$ with the structure of a Frobenius algebra.  When $\Phi$
is \emph{split}, $\Phi(0) = \Z$, so one gets a trace defined on $\Phi(1)$,
which is non-degenerate.

The structure of an $\Ai$--Frobenius algebra is generated by this $\Ai$--structure together with the trace; that is, all chain level operations from the moduli of surfaces in the open category can be derived from these operations, as is indicated in section 7.3 of \cite{costello07}.  Roughly speaking, having a non-degenerate trace allows one to construct the pairing and the copairing.  Together with the $\Ai$--structure, one can recover any fat graph.  We expand upon this in the following section.

\subsection{Positive boundary or ``noncompact''  $\Ai$--Frobenius algebras}\label{posalg}

Define the \emph{positive boundary open cobordism category} $\OO^b$ to be the subcategory of $\OO$ with the same objects and whose morphisms are given by the subcomplex of fat graphs whose associated topological type is a disjoint union of surfaces, all of which have at least one outgoing boundary.

There are certain morphisms in $\OO^b$ whose role should be highlighted.  Certainly, $\OO^b$ contains all of the category $\Ai^+$, and in particular the corollas $m_k: k \to 1$.  It also contains the \emph{coproduct} $\nu$ -- the morphism from 1 to 2 given by the corolla with one incoming and two outgoing leaves.

\begin{prop} \label{positive_prop}

The category $\OO^b$ is generated as a symmetric monoidal category by its subcategory $\Ai^+$ and the coproduct $\nu$.

\end{prop}

\begin{proof}

First, define the \emph{copairing} $C := \nu \circ u :0 \to 2$; this is an exceptional graph with no vertices.  Composing a disjoint
union of $n-1$ copies of $C$ with $m_{k+n-1}$ gives the corolla\footnote{We should be careful to indicate the labeling of the leaves in
  $c_{k,n}$, but since we will consider the \emph{symmetric} monoidal category generated by these, any choice will suffice.} $c_{k, n}: k \to
n$ for any $k \geq 0$ and $n\geq 1$.  Note that we can write $m_k = c_{k, 1}$, $u = c_{0,1}$, $\nu =
c_{1, 2}$, and $C = c_{0, 2}$.  Then the symmetric monoidal subcategory generated by $\Ai$, $u$,  and $\nu$ is the same as the one
generated by all of the $c_{k, n}$.
\begin{figure}[h]
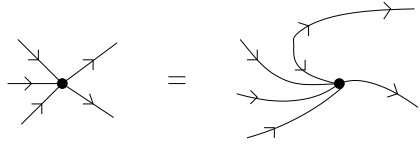

\begin{lpic}{cnk(0.45,0.45)}
\end{lpic}
\caption{the corolla $c_{3,2}$ as a composition $m_4\circ (C\sqcup id)$}\label{cnk}
\end{figure}

Now let $\Gamma:m \to n$ be an arbitrary graph in $\OO^b$; we may assume that $\Gamma$ is connected and non-empty, and so $n\geq 1$.  Pick a maximal
tree $T$ of edges of $\Gamma$ and choose an outgoing leaf of $\Ga$ attached at a vertex $v$ (which is included in $T$ by maximality). 
There is a unique way to orient the edges of $T$ to make it rooted at $v$.  Extend that orientation (arbitrarily) to an orientation of $\Gamma$,
though keeping the ``in'' and ``out'' orientations of the leaves.  Since $T$ includes all of the vertices of $\Gamma$, there is always at least one outgoing edge (or leaf) at each vertex.  Thus the star of each vertex is $c_{k, n}$ for some value of $k$ and $n$.  Consequently $\Gamma$ is obtained as an iterated composition of (disjoint unions of) the $c_{k, n}$, and so is in the symmetric monoidal subcategory generated by them.
\end{proof}

The relations between these generators can be summarized (in a pithy if not particularly helpful way) by saying that two compositions of generators are equal if the fat graphs that they define are the same.  For instance, the Frobenius relation
$$(\mbox{coproduct} \sqcup id) \circ (id \sqcup \mbox{product}) =  (id \sqcup \mbox{coproduct}) \circ (\mbox{product} \sqcup id)$$
expresses the fact that the fat graphs in Figure~\ref{relation} are isomorphic.
\begin{figure}[h]
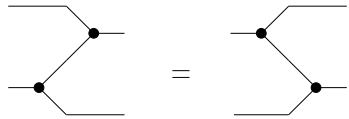

\begin{lpic}{relation(0.45,0.45)}
\end{lpic}
\caption{Frobenius relation}\label{relation}
\end{figure}

Noting that $\OO^b$ contains a copy of $\Ai^{op}$, extending the coproduct (though with no counit!), Proposition \ref{positive_prop} gives us:

\begin{cor}

A split symmetric monoidal functor $\Phi: \OO^b \to \Comp$ makes $A:=\Phi(1)$ into a unital $\Ai$--algebra and non-counital $\Ai$--coalgebra. 

\end{cor}

\section{Bar constructions}\label{Bar}

In this section, we define the classical double bar construction, as studied by many authors, and a quotient version of it by symmetries occurring in
\cite{costello07}. This less well-know bar construction has the advantage of providing resolutions of symmetric monoidal functors. (See Proposition~\ref{Bmon}.)

\medskip

Given a dg-category $\C$ and dg-functors $\Phi:\C\to \Comp$ (which we can think of as a {\em $\C$--module}) and $\Psi:\C^{op}\to \Comp$ 
(a {\em $\C^{op}$--module}), define  the $p$th simplicial level of the double bar construction 
$$B_p(\Phi,\C,\Psi)=
\bigoplus_{\begin{subarray}{r}m_0,\dots,m_p\\ \in \Obj(\C)\end{subarray}}\Phi(m_0)\otimes \C(m_0,m_1)\otimes \dots\ot \C(m_{p-1},m_p)\otimes \Psi(m_p).$$
If $\C$ is symmetric monoidal with objects the natural numbers under addition, 
let $\Si\cong \coprod \Si_n$ denote the subcategory of $\C$ with the same objects and with morphisms the
symmetries in $\C$. 
Then we can define similarly 
$$B_p^\Si(\Phi,\C,\Psi)=
\bigoplus_{\begin{subarray}{r}m_0,\dots,m_p\\ \in \Obj(\C)\end{subarray}}\Phi(m_0)\otimes_{\Si} \C(m_0,m_1)\otimes_{\Si} \dots\ot_{\Si} \C(m_{p-1},m_p)\otimes_{\Si} \Psi(m_p)$$
where $X\ot_{\Si}Y$ denotes the quotient of $X\ot Y$ by $x.f\ot y\sim x\ot f.y$ for any $f\in \Si$ with $f$ acting by
pre- or post-composition on the middle factors and via $\Phi(f)$ and $\Psi(f)$ on the first and last factors. 

\smallskip

Denoting elements of $B_p(\Phi,\C,\Psi)$ by $a\otimes b_1\otimes \dots\otimes b_p\otimes c$, let $d_i: B_p\to B_{p-1}$,
the $i$th face map, be defined by 

$d_0(a\otimes b_1\otimes \dots\otimes b_p\otimes c)=\Phi(b_1)(a)\otimes b_2 \dots\otimes b_p\otimes c$

$d_i(a\otimes b_1\otimes \dots\otimes b_p\otimes c)=a\otimes b_1\otimes\dots \otimes b_{i+1}\circ b_{i} \otimes\dots\otimes b_p\otimes c$ for $0< i<p$

$d_p(a\otimes b_1\otimes \dots\otimes b_p\otimes c)=a\otimes b_1 \dots\otimes b_{p-1}\otimes \Psi(b_p)(c)$.\\
This makes $B(\Phi,\C,\Psi)=\oplus_{p\ge 0}B_p(\Phi,\C,\Psi)$, the {\em double bar construction}, into a semi-simplicial chain
complex, and a chain complex with differential $D_p=(-1)^p\delta+d$ where $\delta$ denotes the differential of $B_p(\Phi,\C,\Psi)$ as a tensor
product of chain complexes, and $d=\sum_{i=0}^p(-1)^id_i$ denotes the simplicial differential.

As all the face maps are well-defined over $\Si$, we have that $B^\Si(\Phi,\C,\Psi)=\oplus_{p\ge 0}B^\Si_p(\Phi,\C,\Psi)$ is also a semi-simplicial
chain complex. (In fact, $B(\Phi,\C,\Psi)$ is a simplicial chain complex, in that it admits well-defined degeneracies, but this is not
true for $B^\Si(\Phi,\C,\Psi)$.)

\medskip

Taking $\Psi=\C(-,m)$ to be the $\C^{op}$--module represented by an object $m$ of $\C$, we note moreover that the bar
construction $B(\Phi,\C,\C(-,m))$ is natural in $m$, i.e.~ we get a functor $B(\Phi,\C,\C):\C\to\Comp$ with value 
$B(\Phi,\C,\C(-,m))$ at $m\in \Obj(\C)$.

\begin{prop}\label{BCC}
For any functor $\Phi:\C\to\Comp$ 
there are quasi-isomorphisms of functors 
$$\al\colon B(\Phi,\C,\C)\ \arsim \  \Phi \ \ \ \ \textrm{and}\ \ \ \ \al^\Si\colon \ B^\Si(\Phi,\C,\C)\sta{\simeq}{\rar} \Phi$$
In particular, $B(\Phi,\C,\C(-,m))\simeq B^\Si(\Phi,\C,\C(-,m))$ for each $m$.  
\end{prop}

The result is well-known for the usual bar construction $B$.  We recall the proof here and show that it also applies to
$B^\Si$.

\begin{proof}
Let $\al=\oplus_p\al_p:B(\Phi,\C,\C(-,m))=\oplus_pB_p(\Phi,\C,\C(-,m))\rar \Phi(m)$ be defined by   
$\al_0(a\ot c)=\Phi(c)(a)$ and $\al_p=0$ for $p>0$.  
This is natural in $m$. Let
 $\beta: \Phi(m)\to B(\Phi,\C,\C(-,m))$ be defined by $\beta(a)=a\otimes 1_m\in \Phi(m)\otimes\C(m,m)$, 
where $1_m$ here denotes the identity on $m$. 
We have $\al\circ\beta=id$ and $\beta\circ\al\simeq id$; an explicit chain homotopy is given by 
$h_i=s_p\circ\dots\circ s_{i+1}\circ \eta\circ d_{i+1}\circ\dots\circ d_p$, where $s_i$ is the $i$th degeneracy,
introducing an identity at the $i$th position, and $\eta$ is the ``extra degeneracy'' which introduces an identity at
the right-most spot. Explicitly, $h_i$ takes $a\otimes b_1\otimes \dots\otimes b_p\otimes c$ to 
$a\otimes b_1\otimes \dots\otimes b_i\ot (c\circ b_{p}\circ\cdots\circ b_{i+1})\ot 1_m\ot\dots\ot 1_m$. 
Hence $\al$ gives a natural transformation by quasi-isomorphisms between the functors 
$B(\Phi,\C,\C)$ and $\Phi$. 

For $B^\Si$, we now just note that the maps $\al$,$\beta$ and $h_i$ are well-defined over $\Si$. (For $h_i$, the degeneracies $s_j$ are
not well-defined but the above composition with $\eta$ is.)
\end{proof}

\begin{rem}\label{bar}{\rm 
More generally, one can show that $B(M,\C,N)\simeq B^\Si(M,\C,N)$ if $M$ or $N$ is quasi-free (i.e., free as a $\C$--module, if one ignores the
differential). 
}\end{rem}

\begin{prop}\label{Bmon}
If $\C$ is (symmetric) monoidal and $\Phi:\C\to \Comp$ is monoidal, then $B(\Phi,\C,\C)$ and $B^\Si(\Phi,\C,\C)$ are 
monoidal. 
If $\Phi$ is symmetric monoidal, then so is 
$B^\Si(\Phi,\C,\C)$.  Moreover, if $\Phi$ is h-split, $B(\Phi,\C,\C)$ and $B^\Si(\Phi,\C,\C)$ are both h-split. 
\end{prop}

\begin{proof}
The monoidal structure of $B^{(\Si)}(\Phi,\C,\C)$ comes directly from that of $\Phi$ and $\C$, taking 
$(a\ot f_1\ot \dots\ot f_{p+1})\ot(a'\ot f'_1\ot \dots\ot f'_{p+1})$ to 
$(a\boxplus a')\ot (f_1\boxtimes f'_1)\ot \dots\ot (f_{p+1}\boxtimes f'_{p+1})$, where $\boxplus$ denotes the monoidal structure of
$\Phi$ and $\boxtimes$ that of $\C$. 

We want to check that  $B^\Si(\Phi,\C,\C)$ is in fact symmetric monoidal, i.e. that the diagram 
$$\xymatrix{B^\Si(\Phi,\C,\C(-,n))\ot B^\Si(\Phi,\C,\C(-,m))\ar[r]\ar[d]_{\tau_\ot} 
  & B^\Si(\Phi,\C,\C(-,n+m))\ar[d]^{\tau_\C}\\
B^\Si(\Phi,\C,\C(-,m)) \ot B^\Si(\Phi,\C,\C(-,n))\ar[r]& B^\Si(\Phi,\C,\C(-,m+n))}$$
commutes, where $\tau_\ot$ denotes the symmetry in the category of chain complexes and $\tau_\C$ the symmetry of $\C$. 
This means that we need 

$(a'\boxplus a)\ot_\Si (f_1'\boxtimes f_1)\ot_\Si \dots\ot_\Si (f'_{p+1}\boxtimes f_{p+1})$ 
equal to 

$(a\boxplus a')\ot_\Si (f_1\boxtimes f_1')\ot_\Si \dots\ot_\Si (f_p\boxtimes f_p')\ot_\Si ((f_{p+1}\boxtimes f'_{p+1})\circ\tau_\C)$. \\
This holds because $(f_i\boxtimes f'_i)\circ\tau_\C=\tau_\C\circ(f'_i\boxtimes f_i)$ in $\C$ and  
 $\Phi(\tau_\C)(a\boxplus a')=a'\boxplus a$ as $\Phi$ is symmetric monoidal.

The fact that $\Phi$ is h-split implies  $B(\Phi,\C,\C)$ and $B^\Si(\Phi,\C,\C)$ are h-split; this follows from the commutativity of the following
diagram: 
$$\xymatrix{B(\Phi,\C,\C)(n)\ot B(\Phi,\C,\C)(m) \ar[r]\ar[d]^\al_\simeq & B(\Phi,\C,\C)(n+ m) \ar[d]^\al_\simeq \\
\Phi(n)\otimes \Phi(m) \ar[r]^\simeq  & \Phi(n+ m).}$$
\end{proof}

Note that in the above proposition, strengthening the assumption on $\Phi$ to be split still only yields $B^{(\Si)}(\Phi,\C,\C)$ h-split.

\section{Hochschild complex operator}\label{HComp}

Let $\e$ be a symmetric monoidal dg-category
which admits a symmetric monoidal functor $i:\Ai\to \e$, for $\Ai$ the category defined in \ref{Aisec}. 
For simplicity, and because all our examples are of this sort, we assume that $i$ is the identity on objects, i.e. that $\e$ is a prop with $\Ai$--multiplication. 
Recall from \ref{Aialg} that $\e$--algebras, i.e.~symmetric monoidal functors $\e\to\Comp$, have an underlying $\Ai$--algebra structure by
precomposition with $i$, and hence have a well-defined Hochschild complex. 
We define in this section a generalization of the Hochschild complex in the form of an 
operator $C$  on dg-functors $\Phi:\e\to \Comp$ with the property that, if $\Phi$ is symmetric monoidal,  the value of 
$C(\Phi)$ at 0 is the usual Hochschild complex of the underlying $\Ai$--algebra. The value of $C(\Phi)$ at $n$ can more generally be identified with
the higher Hochschild homology \`a la Pirashvili \cite{Pir00} associated to the simplicial set which is a union of a circle and $n$ points. 

 In \ref{Hprop} we study the  basic properties of our Hochschild complex operator  
and in \ref{Hact} we prove our main theorem, Theorem~\ref{action}, which
gives a way of constructing actions on Hochschild complexes.

\medskip 

Recall from \ref{annulisec} the functor $\LL:\Ai^{op}\to \Comp$ defined by $$\LL(k)=\bigoplus_{n\ge 1}\Ai(k,n)\ot L_n$$ 
for $L_n=\lgl l_n\rgl$. 

Let $\e$ be a monoidal dg-category. 
Given a functor $\Phi:\e\to\Comp$ and an object $m\in \e$, we can define a new functor $$\Phi(-+m):\e\to\Comp$$ by setting 
$\Phi(-+m)(n)=\Phi(n+m)$ and $\Phi(-+m)(f)=\Phi(f+id_m)$. Note that for any morphism $g\in \e(m,m')$, $\Phi(id+g)$ induces a natural
transformation $\Phi(-+m)\to \Phi(-+m')$. 

Given functors $F:\C\to\Comp$ and $G:\C^{op}\to\Comp$, we denote by 
$$F\ot_{\C}G=\bigoplus_{k \in Obj(\C)} F(k)\ot G(k)/\sim$$
 the tensor product of $F$ and $G$, where the equivalence relation is given by $f(x)\ot y\sim x\ot f(y)$ for any
$x\in F(k)$, $y\in G(l)$ and $f\in \C(k,l)$. This is a chain complex with differential $d=d_F+d_G$ (with the usual Koszul sign convention).

\begin{Def}[Hochschild complex]\label{CDef} Let  $(\e,i)$ be a prop with $\Ai$--multiplication.  
For a functor $\Phi:\e\to \Comp$, define its {\em Hochschild complex} as a functor $C(\Phi):\e\to\Comp$ given on objects by 
$$C(\Phi)(m):=i^*\Phi(-+m)\ot_{\Ai}\LL$$
and on morphisms by  
$$C_*(\Phi)(f):=i^*\Phi(id+f)\ot id.$$
\end{Def}

Note that 
$\LL$ is free as a functor to graded vector spaces, so as a graded vector space, 
$$\begin{array}{rcccl}C(\Phi)(m) &\cong& \bigoplus_{n\ge 1} \Phi(n+m)\ot L_n &\cong&  \bigoplus_{n\ge 1} \Phi(n+m)[n-1]
\end{array}$$
where the second isomorphism comes from the fact that each $L_n$ is generated by a single element in degree $n-1$. 
The differential is given, for $x\in\Phi(n+m)$, by 
$$d(x\ot l_n)=d_\Phi x\ot l_n + (-1)^{|x|}\sum_{k=1}^{n-1} \Phi(i(f_{n,k})+id_m)(x)\ot l_k$$
with $f_{n,k}$  the terms of the differential of $L_n$ as defined in \ref{annulisec}.  

\smallskip

The construction is natural in $\Phi$ and $\e$ in the following sense: Given a factorization of $i$ as 
$\Ai\sta{i'}{\to}\e'\sta{j}{\to}\e$ and a
functor $\Phi:\e\to \Comp$, we have $C(j^*\Phi)\cong j^*C(\Phi)$, and given two functors $\Phi,\Psi:\e\to\Comp$ and a natural
transformation $\eta:\Phi\to\Psi$, we get a natural transformation $C(\eta):C(\Phi)\to C(\Psi)$.

\begin{rem}{\rm 
The operator $C$ generalizes the usual Hochschild complex of $A_\infty$--algebras in the sense that 
for $\Phi:\Ai\to\Comp$ symmetric monoidal, $C_*(\Phi)(0)$
is the usual Hochschild complex of the $A_\infty$--algebra $\Phi(1)$ as in e.g.~\cite[7.2.4]{KS06}. 
In the case of a strict graded algebra, 
taking as generator of $L_n$ the graph $l_n$ of Figure~\ref{graphsex} with orientation $v\w h_1\w\dots\w h_n$ and
using the sign convention for the product given in Figure~\ref{prodsign}, our differential is explicitly given by the following formula: 
for a $n$--chain $a_0\ot \dots\ot a_n$ of the Hochschild complex of an algebra $A$, we have \\
$\begin{array}{rl}d(a_0\ot \dots\ot a_n)=&\sum_{i=0}^n (-1)^{a_0+\dots+a_{i-1}}a_0\ot\dots \ot da_i\ot\dots \ot a_n\\
& \\
 + & (-1)^{a_0+\dots+a_n}\sum_{i=0}^{n-1} (-1)^{i+1}a_0\ot\dots\ot a_ia_{i+1}\ot\dots\ot a_n\\
&\\
+ & (-1)^{n+1+(a_n+1)(a_0+\dots+a_{n-1})+a_n}a_na_0\ot a_1\ot\dots\ot a_{n-1},
\end{array}$\\
where $a_i$ in a superscript denotes the degree of $a_i$.  
}\end{rem}

Note though that we have defined the Hochschild complex for any functor $\Phi:\e\to\Comp$, not just for monoidal ones. In particular, we will apply the
Hochschild constructions to the (in general non-monoidal) representable functors $\Phi(m)=\e(m,-)$, which can be thought of as 
``generalized free $\e$--algebras''. Also, even for $\Phi$ monoidal, $C(\Phi)$ will in general not be monoidal, but we can nevertheless iterate the
construction and talk about $C(C(\Phi))=C^2(\Phi)$, $C^3(\Phi)$, etc.

\begin{Def}[Reduced Hochschild complex]\label{CbDef}
Let $(\e,i)$ be a prop with unital $\Ai$--multiplication and $\Phi:\e\to\Comp$ a functor. 
Define the {\em reduced Hochschild complex of $\Phi$} as the quotient functor $\bC(\Phi)=C(\Phi)/U:\e\to\Comp$ given on object by 
$$\overline{C}(\Phi)(m)\ =\ \bigoplus_{n\ge 1} \Phi(n+m)/_{U_{n}}\ot L_{n}$$ 
where $U_{n}=\sum_{i=2}^n\Img(\Phi(i(u_i)+id_m))\subset \Phi(n+m)$ 
with $u_i=1\ot\dots\ot u\ot \dots\ot 1$ in $\Ai^+(n-1,n)$  the morphism that inserts a unit at the $i$th position.
\end{Def}

As the quotient does not affect the variable part of $C(\Phi)$, it is clear that $\bC(\Phi)$ is still defines a functor $\e\to\Comp$. On the other
hand, we need to check that the differential is well-defined on the quotient, which is done in the following lemma:

\begin{lem}\label{reddiff}
The differential of $C(\Phi)(m)$ induces a  well-defined differential on $\bC(\Phi)(m)$ for each $m$.
\end{lem}

\begin{proof}
Let $U_n\le \Phi(n)$ be as in Definition~\ref{CbDef}. 
We first note that $U_n$ is mapped to itself by $d_\Phi$ because the structure map $c_\Phi$ of $\Phi$ 
is by chain maps and $d(u_i)=0$.   
We need to see that the same holds for the Hochschild part of the differential. This follows from the commutativity of the following diagram
(written in the case $m=0$ for readability) 
$$\xymatrix{\Phi(n\!-\!1)\ot\lgl u_i\rgl\ot  L_{n}\ar[d]^-{d_L} \ar[r]^-{c_\Phi}&  \Phi(n)\ot L_{n}\ar[d]^-{d_L}\\
\displaystyle{\bigoplus_{\begin{subarray}{c}2\le r\le n\\ 1\le j\le n\end{subarray}}} 
   \Phi(n\!-\!1)\ot\lgl u_i\rgl\ot\lgl m^j_r\rgl\ot L_{n+1-r}\ar[d]^{c_{\Ai}} \ar[r]^-{c_\Phi}&
\displaystyle{\bigoplus_{\begin{subarray}{c}2\le r\le n\\ 1\le j\le n\end{subarray}}} 
   \Phi(n)\ot\lgl m^j_r\rgl\ot L_{n+1-r}\ar[d]^{c_\Phi} \\ 
\displaystyle{\bigoplus_{r,j}} \Phi(n\!-\!1)\ot \lgl u_i\oplus m_r^j\rgl\ot L_{n+1-r} \ar[r]^-{c_\Phi}&
   \bigoplus_{k\ge 1}\Phi(k)\ot L_k}$$
where $m_r^j=1\oplus\dots\oplus m_r\oplus \dots\oplus 1$ denotes the multiplication
$m_r$ of the entries $j,\dots,j+r-1$ (mod $n$). The target of the map $c_{\Ai}$ is justified as follows. 
There are two cases when composing $u_i$ and $m_r^j$: either $i\notin\{j,\dots,j+r-1\}$ so that the composition
$m^j_r\circ u_i$ is of the form $u_i\oplus m^j_r$. Otherwise, the composition $m^j_r\circ u_i$ is
the identity map when $r=2$ and $0$ when $r>2$. 
In the case $r=2$, the term $m^{i-1}_2\circ u_i$ cancels with $m^i_2\circ u_i$. (The sign comes from the differential in $L$).
\end{proof}

Let $\e,\F$ be dg-categories and suppose that $\Phi:\e\to \Comp$ in fact extends to a bifunctor 
$\Phi:\F^{op}\x \e\to \Comp$. In this case, we also call $\Phi$ an  {\em $(\F^{op},\e)$--bimodule}\footnote{Here, to correctly work out the signs in the differential, we take the structure maps of the bimodule to be in the form $\F(m_1, m_2) \times \Phi(m_2, n_1) \times \e(n_1, n_2) \to \Phi(m_1, n_2)$ and apply the usual sign convention.}.

\begin{prop}\label{bimodule} Let $(\e,i)$ be a prop with (unital) $\Ai$--multiplication and 
suppose $\Phi$ is an $(\F^{op},\e)$--bimodule. Then the Hochschild complexes 
$C(\Phi(a,-))$ and $\bC(\Phi(a,-))$ built using the $\e$--structure of $\Phi$ pointwise on objects $a$ of $\F$ assemble again to $(\F^{op},\e)$--bimodules. 
\end{prop}

\begin{proof}
Given $f:m_1\to m_2$ in $\e$ and $g:a_2\to a_1$ in $\F^{op}$, $C(\Phi)(g,f)$ on the summand $\Phi(a_2,n+m_1)\ot L_n$ is the map 
$(-1)^{(n-1)|f|}(g,id_n+f)$. This is well-defined as the Hochschild part of the differential commutes with such maps. 
\end{proof}

\begin{ex}\label{bimodex}{\rm 
The example we are interested in is the $(\e^{op},\e)$--bimodule $\e$. By the proposition, its Hochschild and iterated
Hochschild complexes $C(\e),C^n(\e)$, and reduced versions when relevant, are again $(\e^{op},\e)$--bimodules. 
Given any $\Phi:\e\to\Comp$, this allows to consider the double bar construction $B(\Phi,\e,C^n\e)$ (as in Section~\ref{Bar}), which in fact identifies with   
$C^n(B(\Phi,\e,\e))$ as both have value at $m$ given by 
$$\bigoplus_{\begin{subarray}{r}p\ge 0, \ n\ge 1 \\ m_0,\dots,m_p\ge 0\end{subarray}}\Phi(m_0)\otimes \e(m_0,m_1)\otimes \dots\ot\e(m_{p},n+m)\otimes L_n$$
(and similarly for the reduced constructions). 
}\end{ex}

\subsection{Properties of the Hochschild operator}\label{Hprop}

We prove in this section that the Hochschild complex operator is homotopy invariant and we describe its behavior under
iteration.  Throughout the section, we assume that $(\e,i)$ is a prop with $\Ai$--multiplication when we consider the Hochschild complex $C$, and
that $(\e,i)$ is a prop with unital $\Ai$--multiplication when we consider its reduced version $\bC$.

\medskip

Recall that by a quasi-isomorphism of functors $\Phi\arsim \Phi':\e\to \Comp$, we mean a natural transformation by
quasi-isomorphisms $\Phi(m)\arsim\Phi'(m)$.

\begin{prop}\label{qi}
Let $\Phi,\Phi':\e\to \Comp$. 
A quasi-isomorphism of functors $\Phi\arsim \Phi'$  induces quasi-isomorphisms of functors 
$C_*(\Phi)\arsim C_*(\Phi')$ and 
  $\overline{C}_*(\Phi)\arsim \overline{C}_*(\Phi')$. 
\end{prop}

For the reduced part of the proposition, we need the following lemma. 

\begin{lem}\label{quot}
Suppose $\Phi\arsim \Phi':\e\to \Comp$ are quasi-isomorphic functors. 
For any $J\subset\{1,\dots,n\}$, let 
$U_J=\sum_{j\in J} \operatorname{Im}\big(\Phi(i(u_j))\big)\subset \Phi(n)$, 
and similarly for $\Phi'$. 
Then $$\Phi(n)/U_J\arsim\Phi'(n)/U_J.$$ 
If $\Phi\cong\Phi'$, these maps are also isomorphisms.  
\end{lem}

\begin{proof}
We prove the lemma by induction on the cardinality of $J$, for any $n$, starting with the case $J=\emp$ which is trivial. 

Fix $J=\{j_1\le\dots\le j_s\}\subset \{1,\dots,n\}$ and    
denote by $U_i,U_i'$ the image of $i(u_i)$ in $\Phi(n)$ and $\Phi'(n)$ respectively. 
We want to show that $\Phi(n)/(U_{j_1}+\dots+U_{j_s})\arsim \Phi'(n)/(U'_{j_1}+\dots+U'_{j_s})$. 

There is a short exact sequence 
 $$\xymatrix{\Phi(n\!-\!1)/(U_{j_1}+\dots+U_{j_{s-1}})\ar[r]^-{i(u_{j_s})}  & \Phi(n)/(U_{j_1}+\dots+U_{j_{s-1}}) \ar[d] \\
&  \Phi(n)/(U_{j_1}+\dots+U_{j_s}).}$$
Indeed $u_{j_s}$ is injective on $\Phi(n\!-\!1)/(U_{j_1}+\dots+U_{j_{s-1}})$ with left inverse $i(m_2^{j_s})$ (where
$m_2^{j_s}$ multiplies $j_s$ and $j_s+1$ modulo $n$). 
The result then follows by induction by 
considering the map of short exact sequences induced by $\Phi\to \Phi'$. 
\end{proof}

\begin{proof}[Proof of the Proposition] 
We filter the complexes 
$C_*(\Phi)(m)=\oplus \Phi(k+m)\ot L_k$ and $\overline{C}_*(\Phi)(m)=\oplus \Phi(k+m)/U_k\ot L_k$ by $k$ and consider the resulting
spectral sequence. In both cases the differential is $d_\Phi + d_H$ where $d_H$ decreases the filtration
grading and $d_\Phi$ does not. Hence the $E^1$--terms of the spectral sequences are $E^1_{p,q}=H_p(\Phi(q+1+m))\ot L_{q+1}$ and 
$E^1_{p,q}=H_p(\Phi(q+1+m)/U_{q+1})\ot L_{q+1}$ in the reduced case. A quasi-isomorphism of functors induces a map of spectral sequences
which is an isomorphism on the $E^1$--term by the assumption in the unreduced case and by Lemma~\ref{quot} in the reduced case. 
\end{proof}

Applying Proposition \ref{qi} to the map $\al:B(\Phi,\e,\e)\sta{\simeq}{\to}\Phi$ of Proposition~\ref{BCC}, we get
a quasi-isomorphism $$C(\al): C(B(\Phi,\e,\e))\arsim C(\Phi).$$
The proof of Proposition \ref{BCC} gives a pointwise  homotopy inverse $\beta$ to $\al$ which is not a natural transformation, so we cannot
apply Proposition~\ref{qi} to it. (In fact $C(\beta)$ does not define a chain map.) 
Instead, we construct now an explicit pointwise homotopy inverse $\tilde \beta$ to $C^n(\al)$, for any $n$, as this will be useful later to produce
explicit actions on the Hochschild complex of $\e$--algebras.  

\begin{prop}\label{beta}
For any $n$ and $m$, there is a quasi-isomorphism  of chain complexes 
$$\tilde\beta:C^n(\Phi)(m)\sta{\simeq}{\rar} C^n(B(\Phi,\e,\e))(m)$$
natural both with respect to natural transformations $\Phi\to\Phi'$ and with respect to functors 
$j:\e\to \e'$ with $i'=j\circ i:\Ai\to\e'$. Moreover, $\tilde\beta$ is a right inverse to $C(\al)$ for $\al$ as in Proposition~\ref{BCC}. 
\end{prop}

\begin{proof}
We first define
$\tilde\beta$ in the case $\e=\Ai$, and using the identification $$C^n(B(\Phi,\Ai,\Ai))(m)\cong B(\Phi,\Ai,C^n(\Ai)(m))$$
of Example~\ref{bimodex}.  
 The map $\tilde\beta$  for a general  $\e$ and $\Phi:\e\to\Comp$ is then obtained by post-composition with the quasi-isomorphism  
$$C^n(B(i^*\Phi,\Ai,\Ai))(m)\to C^n(B(\Phi,\e,\e))(m)$$ 
induced by $i:\Ai\to \e$. The naturality of $\tb$ in $\e$ follows from the naturality of that second map.

Recall from \ref{annulisec} the map 
 $$d_L:L_{k}\to \bigoplus_{1\le j< k} \Ai(k,j)\ot L_{j}.$$
We consider here more generally the map 
$$d_L:L_{k_1}\ot\dots\ot L_{k_n}\to \bigoplus_{1\le k< n} \Ai(k_1+\dots+k_n,k_1'+\dots+k'_n)\ot L_{k'_1}\ot\dots\ot L_{k'_n}$$
induced by the differential of the $\oc{k}{n}$--graph which is the union $l_{k_1}\sqcup\dots\sqcup l_{k_n}$, where $k=k_1+\dots+k_n$. 
We let $\tilde\beta:=\sum_{p\ge 0} (d_L)^p$, where we interpret $(d_L)^p$ as the composition
\begin{align*}
\Phi(k+m)\ot L_{\uk}\sta{(d_L)^p}{\rar} \bigoplus_{j_i} \Phi(k+m)\ot\Ai(k,j_1)\ot\dots\ot\Ai(j_{p-1},j_p)\ot L_{\underline{j}_p}\\
\sta{+id_m}{\rar} \bigoplus_{j_i} \Phi(k+m)\ot\Ai(k+m,j_1+m)\ot\dots\ot\Ai(j_{p-1}+m,j_p+m)\ot L_{\underline{j}_p}
\end{align*}
with image in the $p$th simplicial level of $B(\Phi,\Ai,C^n(\Ai)(m))$, where $L_{\uk}=L_{k_1}\ot\dots\ot L_{k_n}$ 
and $L_{\underline{j}_p}=L_{j^p_1}\ot\dots\ot L_{j^p_n}$ is identified with $\lgl id_{j_p+m}\rgl\ot L_{\underline{j}_p}$ in 
$\Ai(j_p+m,j_p+m)\ot L_{\underline{j}_p}$ in $C^n(\Ai)(m)$. 
Note that the sum is always finite as $(d_L)^p$ applied to $L_k$ is 0 for all $p\ge k$.

We will show that the relation $d\tb=\tb d$ holds on each component as maps 
$$\bigoplus_{(k)=(k_1,\dots,k_n)}\!\!\!\!\!\! \Phi(k+m)\ot L_{k_1}\ot\dots\ot L_{k_n}\ \ \rar\ \ \bigoplus_p B_p\big(\Phi,\Ai,C^n(\Ai)(m)\big)$$
i.e. 
that for each fixed $(k)$,  
the images of $d\tb$ and $\tb d$ agree on the component of simplicial degree $p$. We first consider $\tb d$.

As $d=d_\Phi + c_\Phi d_L$, we have 
on the $(k)$th component
$$_{(k)}(\tb d)=\sum_{i=0}^{K-1}(d_L)^id_\Phi + \sum_{i=0}^{K-2} (d_L)^ic_\Phi d_L$$
with $K=\max(k_1,\dots,k_n)$, 
which can be rewritten as 
$$_{(k)}(\tb d)=d_\Phi\sum_{i=0}^{K-1}(-1)^i(d_L)^i + d_0 \sum_{i=0}^{K-2} (d_L)^{i+1}$$
as $d_Ld_\Phi=-d_\Phi d_L$ and $\bar d_L^i c_\Phi \bar d_L=d_0 \bar d_L^{i+1}$ with $d_0$ the $0$--th face map in\\ 
$B_{i}(\Phi,\Ai,C^n(\Ai)(m))$.  
Hence the component of $_{(k)}(\tb d)$ of simplicial degree $p$ is 
$$ _{(k)}(\tb d)_p =(-1)^pd_\Phi (d_L)^p + d_0(d_L)^{p+1}.$$

\smallskip

On the other hand, we have $_{(k)}(d\tb)=d(_{(k)}\tb)$ where the differential on the $p$th component of $_{(k)}\tb$ is
$(-1)^p(d_\Phi+(d_\A)_1+\dots+(d_\A)_{p} + \widetilde{d_L}) + \sum_{i=0}^p(-1)^id_i $, where $(d_\A)_i$ denotes the
differential of the $i$th factor $\Ai(-,-)$ and $\widetilde{d_L}$ the map  
$d_{p+1}d_L$ which applies the differential to the factors $L$ without increasing the simplicial degree. 
As the face maps $d_i$ reduce the simplicial degree, we have 
$$_{(k)}(d\tb)_p=(-1)^p\Big(d_\Phi+(d_\A)_1+\dots+(d_\A)_{p} + d_{p+1}d_L\Big) (d_L)^p + \Big(\sum_{i=0}^{p+1}(-1)^id_i\Big)(d_L)^{p+1}.$$ 
This is a sum of two compositions whose respective first terms are exactly $_{(k)}(\tb d)_p$, and whose last terms cancel.  
Hence $$ _{(k)}(d\tb)_p - \, _{(k)}(\tb d)_p= 
\sum_{i=1}^{p} \Big((-1)^p(d_\A)_i(d_L)^p+(-1)^i d_i(d_L)^{p+1}\Big).$$
The $i$th term in the sum can be rewritten as 
$$(-1)^i(d_L)^{p-i}\Big((d_\A)_i+d_i d_L\Big)(d_L)^i$$ 
which is 0 as the middle part $((d_\A)_i+d_i d_L) d_L$ is the square of a differential
in the graph complex, which gives the desired equality.

\medskip

As  
$C^n(\al)(m)\circ \tb$ is the identity and $C^n(\al)(m)$ is a quasi-isomorphism by Propositions~\ref{BCC} and \ref{qi}, 
$\tb$ is also a quasi-isomorphism. 
The map $\tb$ is natural in $\Phi$ as $d_L$ is natural in $\Phi$. 
\end{proof}

Next we describe how the Hochschild operator behaves under iteration.  
Recall from Section~\ref{Algsec} that 
 a monoidal functor $\Phi:\e\to \Comp$ is {\em h-split} if the maps $\Phi(n)\ot\Phi(m)\to \Phi(n+m)$ 
are quasi-isomorphisms, and  {\em split} if the maps are isomorphisms.

For $\Phi:\e\to \Comp$, we can consider the iterated Hochschild functor 
$C^n(\Phi)=C(C(\dots C(\Phi)\dots))$. When $\Phi$ is h-split monoidal, it computes the tensor powers of the
Hochschild complex:

\begin{prop}\label{monoidal}
If $\Phi:\e\to \Comp$ is monoidal, then there are natural maps 
$$\la:C(\Phi)(0)^{\ot n}\ot\Phi(1)^{\ot m}\rar C^n(\Phi)(m)$$
and $$\bar \la:\bC(\Phi)(0)^{\ot n}\ot\Phi(1)^{\ot m}\rar \bC^n(\Phi)(m).$$ 
These maps are quasi-isomorphisms if $\Phi$ is h-split, and isomorphisms if $\Phi$ is split.

Moreover, there exists an action of $\Si_n$ on $C^n(\Phi)$ such that if $\e, \Phi$ and $i$ are symmetric monoidal, 
these maps are $\Si_n\x\Si_m$--equivariant (where $\Si_m$ acts on $C^n(\Phi)(m)$ via the symmetries of $\e$). 
\end{prop}

\begin{proof}
$C_*(\Phi)(0)^{\ot n}=(\oplus_{k_1} \Phi(k_1)\ot L_{k_1})\ot\dots\ot(\oplus_{k_n}\Phi(k_n)\ot L_{k_n})$ and \\
$C_*^n(\Phi)(m)=\oplus_{k_n}(\dots (\oplus_{k_1} \Phi(k_1+\dots+k_n+m)\ot L_{k_1})\ot\dots\ot L_{k_n})$. 
The maps $\la$ and $\bar \la$ are then defined by appropriately permuting the factors and then using the monoidal structure of
$\Phi$. These maps are isomorphisms/quasi-isomorphisms in the unreduced case if the structure maps of $\Phi$ have that property. 
For the reduced complexes, we need\\
$\Phi(k_1)/U_{k_1}\ot\dots\ot\Phi(k_n)/U_{k_n}\ot\Phi(m)\to \Phi(k_1+\dots +k_n+m)/U_{k_1}/\dots /U_{k_n}$ 
to be an isomorphism when $\Phi$ is split and
a quasi-isomorphism when $\Phi$ is h-split. This follows from an iteration of 
Lemma~\ref{quot}: Consider the restriction of the natural transformation $\Phi\ot\dots\ot\Phi\to\Phi(\ +\dots +\ )$
to the first variable and apply the lemma with $J_1=\{2,\dots,k_1\}$. This gives a quasi-isomorphism 
$\Phi(k_1)/U_{k_1}\ot\Phi(k_2)\ot\dots\Phi(m)\to \Phi(k_1+\dots+m)/U_{k_1}$. This quasi-isomorphism is functorial in the
variables $k_2,\dots,m$ and we can repeat the process until we obtain the desired result. 
\end{proof}

\subsection{Action on Hochschild complexes}\label{Hact}

Given a monoidal dg-category $\D$ with objects pairs of natural numbers $\oc{m}{n}$, we say that a pair of chain complexes $(V,W)$ is a {\em $\D$--module} if
there is a split monoidal dg-functor $\Psi:\D\to\Comp$ with $\Psi(\oc{0}{1})=V$ and $\Phi(\oc{1}{0})=W$, i.e. if  
there are chain maps 
$$\big(V^{\otimes n_1}\ot W^{\otimes m_1}\big)\ot \D(\oc{m_1}{n_1},\oc{m_2}{n_2}) \rar V^{\otimes n_2}\ot W^{\otimes m_2}$$
compatible with composition in $\D$.    
We say that $(V,W)$ is a {\em homotopy $\D$--module} if the compatibility condition is only satisfied up to homotopy, that is if 
$\Psi$ is only a functor up to homotopy, satisfy the equation $\Psi(f\circ g)\simeq \Psi(f)\circ\Psi(g)$ for any pair of composable morphisms $f,g$ in
$\D$. 
In particular, taking homology with field coefficients (or general coefficients
but restricting to the ``operadic part'' with $\oc{m_2}{n_2}=\oc{0}{1}$ or $\oc{1}{0}$), we get in both cases an honest action of $H_*(\D)$ on $(H_*(V),H_*(W))$. 

If $\D$ is symmetric monoidal,
we say that the module structure is {\em  $\Si$--equivariant} if the functor $\Psi$ is symmetric monoidal.

\medskip

Proposition \ref{bimodule} in the case where $\Phi$ is the $(\e,\e^{op})$--bimodule $\e$
can be reinterpreted as follows: Given $\e$, we can define its {\em Hochschild core category} $C\e$ with objects 
$$\oc{m}{n}=(m,n)\in \Obj(\e)\x\N \ (=\N\x\N),$$ for $\N$ the natural numbers including 0, and
morphisms
$$C\e(\oc{m_1}{n_1},\oc{m_2}{n_2})=\left\{\begin{array}{ll} C^{n_2}(\e(m_1,-))(m_2) & n_1=0\\
0& n_1\neq 0 \end{array}\right.$$
where $C^0$ means the identity operator, so that $C\e(\oc{m_1}{0},\oc{m_2}{0})=\e(m_1,m_2)$. 
The only possible non-trivial compositions in $C\e$ are given by the bimodule
structure of $C^n(\e(m,-))$ described in Proposition~\ref{bimodule}.
Moreover, $C\e$ is monoidal via the maps 
$$C^{n}(\e(m_1,-))(m_2)\ot C^{n'}(\e(m'_1,-))(m'_2) \to C^{n+n'}(\e(m_1+m_1',-))(m_2+m_2')$$
as in Proposition~\ref{monoidal}, and $C\e$ is symmetric monoidal when the same is true for $\e$.

We call a monoidal category $\tCE$ with objects $\N\x\N$ an {\em extension} of $C\e$ if there is a monoidal inclusion 
$C\e\inc\tCE$ with
$\tCE(\oc{m_1}{n_1},\oc{m_2}{n_2})=C\e(\oc{m_1}{n_1},\oc{m_2}{n_2})$ when $n_1=0$.  
We define the {\em reduced Hochschild core category} $\bC\e$ and its extensions in the same way, replacing $C$ by $\bC$.  

Our main result says that if $\tCE$ is an extension of $C\e$ (or $\bC\e$), then $\tCE$ acts on the
(reduced) Hochschild complex of split monoidal functors $\Phi:\e\to \Comp$ in the following sense:

\begin{thm}\label{action}
Let $(\e,i)$ be a prop with $\Ai$--multiplication and $\tCE$ an extension of  $C\e$.  
Then for any monoidal functor $\Phi:\e\to\Comp$, 
there is a diagram 
 $$\xymatrix{\Big(C(\Phi)(0)^{\ot n_1}\ot\Phi(1)^{\ot m_1}\Big) \ \textstyle{\bigotimes}\  \tCE(\oc{m_1}{n_1},\oc{m_2}{n_2})
\ar[r]^-\ga & C^{n_2}(\Phi)(m_2)\\
& C(\Phi)(0)^{\ot n_2}\ot\Phi(1)^{\ot m_2} \ar[u]_{\la}
}$$
natural in $\Phi$,  
with $\la$ as in  Proposition~\ref{monoidal}. 
If $\Phi$ is split, the composition $\la^{-1}\circ \ga$ makes the pair 
$(C(\Phi)(0),\Phi(1))$ into a $\tCE$--module, and a homotopy $\tCE$--module for any choice of $\la^{-1}$ if $\Phi$ is h-split. 
Moreover, if $\e,\Phi,i$ and $\la^{-1}$ are symmetric monoidal, the module structure is $\Si$--equivariant.

If $(\e,i)$ is a prop with unital $\Ai$--multiplication, the same holds for the reduced case, replacing $C$ by $\bC$. 
\end{thm}

An extension $\tCE$ of $C\e$ can be thought of as a way to encode a natural action on the Hochschild complex of the representable functors $\e(n,-)$,
and the above theorem is only non-trivial when the complex $\tCE(\oc{m_1}{n_1},\oc{m_2}{n_2})$ are not identically 0  
for $n_1\neq 0$. Thinking of the representable functors as generalized free algebras, the theorem can be interpreted as saying that an
natural/compatible action on the Hochschild complex of free algebras induces an action on the Hochschild complex of all algebras.

The map $\ga$ in the statement is explicit, given by the big diagram in the proof of the theorem below. 
This allows to write down formulas for operations given cycles in the extension category (see
Section~\ref{ks_section} and the end of Section~\ref{strict_section}).

Note that restricting to $n_2=1$ and $m_2=0$ avoids having to invert $\la$,  and  
restricting further to $n_1=1$ and $m_1=0$ avoids needing $\la$ at all. In particular,  
$C(\Phi)(0)$ is a $\tCE(\oc{0}{1},\oc{0}{1})$--module without any monoidal assumption on $\Phi$. 
Alternatively, one can use $C^n(\Phi)(m)$ as a model of $C(\Phi)(0)^{\ot n}\ot \Phi(1)^{\ot m}$ which admits an action of $\tCE$
without reference to 
$\lambda$, as in the following: 
\begin{cor}
Let $(\e,i)$ be a prop with (unital) $\Ai$--multiplication. 
For any $\Phi:\e\to\Comp$ and any extension $\tCE$ of $C\e$, taking $C_\Phi(\oc{m}{n})=C^n(\Phi)(m)$ defines a dg-functor 
$C_\Phi:\tCE\to\Comp$ extending $\Phi$ on $\e$ (and the same in the reduced case). Moreover, the association $\Phi\mapsto C_\Phi$ defines a functor 
$\Fun(\e,\Comp)\to \Fun(\tCE,\Comp)$. 
\end{cor}

This corollary is a direct corollary of the proof of Theorem~\ref{action}.

\begin{proof}[Proof of Theorem~\ref{action}]
The action is defined by the following diagram: 
{\small $$\xymatrix{C(\Phi)(0)^{\ot n_1}\ot\Phi(1)^{\ot m_1}\ot \tCE(\oc{m_1}{n_1},\oc{m_2}{n_2}) \ar[d]^{\la\ot id} \ar@{-->}[dr]^{\hspace{5mm}\ga}  
  & C(\Phi)(0)^{\ot n_2}\ot\Phi(1)^{\ot m_2}\ar[d]_\la\\
C^{n_1}(\Phi)(m_1)\ot \tCE(\oc{m_1}{n_1},\oc{m_2}{n_2}) \ar[d]^{\tb\ot id}_\simeq & C^{n_2}(\Phi)(m_2)\\
C^{n_1}\big(B(\Phi,\e,\e)\big)(m_1)\ot \tCE(\oc{m_1}{n_1},\oc{m_2}{n_2}) \ar[d]_\cong  & C^{n_2}\big(B(\Phi,\e,\e)\big)(m_2) \ar[u]^{C(\al)}_\simeq\\
B\big(\Phi,\e,C^{n_1}(\e)(m_1)\big)\ot \tCE(\oc{m_1}{n_1},\oc{m_2}{n_2}) \ar[d]_= & B\big(\Phi,\e,C^{n_2}(\e)(m_2)\big) \ar[u]_\cong\\
B\big(\Phi,\e,\tCE(\oc{-}{0},\oc{m_1}{n_1})\big)\ot \tCE(\oc{m_1}{n_1},\oc{m_2}{n_2}) \ar[r] & B\big(\Phi,\e,\tCE(\oc{-}{0},\oc{m_2}{n_2})\big) \ar[u]_=
}$$}
The map $\tilde\beta$ is that of Proposition~\ref{beta} and the map $\alpha$ is that of  
Proposition~\ref{BCC}. They are quasi-isomorphisms for any $\Phi$. 
The map $\la$ is that of Proposition~\ref{monoidal}. It is an isomorphism whenever $\Phi$ is split and a quasi-isomorphism whenever
$\Phi$ is h-split.
The bottom horizontal arrow is induced by composition in $\tCE$. 

Consider the composition with a further morphism in $\tCE(\oc{m_2}{n_2},\oc{m_3}{n_3})$. 
Note now that the failure of $\tb\circ C^{n_2}(\al)$ to be the identity 
lies in the non-zero simplicial degrees of $B\big(\Phi,\e,\tCE(\oc{-}{0},\oc{m_2}{n_2})\big)$. As the simplicial degree is constant when applying the composition with $\tCE(\oc{m_2}{n_2},\oc{m_3}{n_3})$, this difference is killed when
we apply $C^{n_3}(\al)$ at the end of the action. Hence, when $\Phi$ is split monoidal, the action is strictly associative.

\smallskip

Let $B^\Si$ denote the quotiented bar construction defined in Section~\ref{Bar}. 
If $\e, i$  and $\Phi$ are symmetric monoidal, then using $B^\Si$ instead of $B$, replacing $\tb$ with its composition with the
quotient map $B\to B^\Si$,  makes 
the diagram above equivariant under the action of $\Si_{m_1}\x\Si_{n_1}$, by Proposition~\ref{Bmon} and ~\ref{monoidal}, and the fact that
this action is given by morphisms of $\tCE$. 

\smallskip

For the reduced version, we need to check that this composition of maps is well-defined. (The map $\tb$ is in fact not well-defined in
that case.) 

Consider 
the action of some $f\in  \tCE(\oc{m_1}{n_1},\oc{m_2}{n_2})$ on some $x\ot l_{\uk}\in C^{n_1}(\Phi)(m_1)$ with $x\ot l_{\uk}$ identified with $0$ in
$\bC^{n_1}(\Phi)(m_1)$, i.e. $$x\ot l_{\uk}=c_\Phi(y\ot u_j)\ot l_{k_1}\ot\dots\ot l_{k_{n_1}}$$ for $y\in \Phi(k-1+m_1)$,  with 
$k=k_1+\dots+k_{n_1}$ and 
$u_j=i(u_j)\in \e(k-1+m_1,k+m_1)$ introducing a unit in the $j$th position for $j\in\{2,\dots,k_1,k_1+2,\dots,k_{n_1}\}$. 

Following the diagram defining the action, we have
\begin{align*} 
(x\ot l_{k_1}\ot\dots\ot l_{k_{n_1}})\ot f & \sta{\tb}{\mapsto} x \ot (id_{k+m_1}\ot l_{k_1}\ot\dots\ot l_{k_{n_1}}) \ot f + \textrm{higher order} \\
& \sta{c_{\tCE}}{\mapsto} x \ot (\sum g \ot l_{k'_1}\ot\dots\ot l_{k'_{n_2}})  + \textrm{higher order} \\
& \sta{\al}{\mapsto} \sum c_\Phi(x \ot g) \ot l_{k'_1}\ot\dots\ot l_{k'_{n_2}}
\end{align*}
for some maps $g\in \e(k+m_1,k'+m_2)$. 
Now $$c_\Phi(x \ot g)=c_\Phi(c_\Phi(y\ot u_j) \ot g)=c_\Phi(y \ot c_{\e}(u_j \ot g))$$
so it is enough to know that $\sum c_{\e}(u_j \ot g)$ is of the form $\sum c_{\e}(g'\ot u_{j'})$ for some $g',j'$ whenever $g$ comes
from a composition as above. 
We have (in abreviated notation)
$$\sum c_{\e}(u_j \ot g)\ot l_{\uk'}=c_{\tCE}(u_j\ot c_{\tCE}((id_{k+m_1} \ot l_{\uk}) \ot f))=c_{\tCE}((u_j \ot l_{\uk}) \ot f)$$
by definition and associativity of composition in $\tCE$. As $u_j\ot l_{\uk}$ is identified with 0 in 
$\bC^{n_1}(\e(k-1+m_1,-))(m_1)=\tCE(\oc{k-1+m_1}{0},\oc{m_1}{n_1})$, we must
have that $c_{\tCE}((u_j \ot l_{\uk}) \ot f)$ is identified with 0 in the reduced Hochschild complex
$\tCE(\oc{k-1+m_1}{0},\oc{m_2}{n_2})$, which means precisely that 
 $\sum c_{\e}(u_j \ot g)$ is of the form $\sum c_{\e}(g'\ot u_{j'})$ as required. 
\end{proof}

The next result says that the action of Theorem~\ref{action}  is also natural in $(\e,\hat \e)$ in the following sense:

\begin{thm}\label{factorization}
Let $(\e,i),(\e',i')$ be  props with (unital) $\Ai$--multiplication and $\hat \e,\hat \e'$ be extensions of $\e,\e'$. 
Suppose that there is a symmetric monoidal functor $\hat{j}:\tCE\to\tCE'$ such that  
$i'=j\circ i:\Ai\to\e\to\e'$ for $j$ the restriction of $\hat j$ to $\e$.  
Then for any (h-)split monoidal functor $\Phi:\e'\to\Comp$, the (homotopy) $\tCE$--action of Theorem~\ref{action} on the pair 
$\big(j^*\Phi,C(j^*\Phi)\big)\cong\big(\Phi,C(\Phi)\big)$ factors through the $\tCE'$--action.

The same holds in the reduced case, replacing $C$ by $\bC$.  
\end{thm}

\begin{proof}
This follows directly from the naturality of the maps defining the action.
\end{proof}

\section{Examples and applications}\label{ex}

In this section, we apply Theorem~\ref{action} to specific categories $\e$. 
In \ref{Cos}, we consider the case $\e=\OO$, the open cobordism category of \ref{OO}. We show that the open-closed category $\OC$ of
section \ref{OCsec} is an extension of $\bC\OO$ in the sense of Section~\ref{Hact}. The application of Theorem~\ref{action} to this extension, stated as
Theorem~\ref{CosThm},  can be interpreted as a reformulation of Costello's Theorem~A (2-3) in \cite{costello07}. 
In \ref{ks_section}, we explain how reading off the action of $\OC$ obtained in the previous section on open field theories
$\Phi:\OO\to\Comp$ recovers the recipe given by Kontsevich-Soibelman in \cite{KS06}.
Sections \ref{twist_section} and \ref{pos} give determinant-twisted and positive boundary versions of Theorem~\ref{CosThm}.

In \ref{strict_section}, we consider the case of strict Frobenius algebras, with $\e=H_0(\OO)$. We show that the category $\SD$ of
Sullivan diagrams defined in \ref{SDsec} is an extension of 
$\bC(H_0(\OO))$.  The application of Theorem~\ref{action} in this case yields Theorem~\ref{strict}, which recovers Theorem 3.3 of \cite{Tradler-Z}, 
giving an action of Sullivan diagrams on the Hochschild complex of strict
Frobenius algebras. Using the projection $\OC\to \SD$, this produces a open-closed field theory though with much of the structure collapsed. 
At the end of the section, we give explicit
formulas for the product, coproduct, and $\Delta$-- (or $B$--)operator on the Hochschild complex in this case. 
In Section~\ref{strings} then gives an application to string topology in characteristic 0 using the models of Lambrechts-Stanley \cite{lambrechts_stanley}.

Finally, sections \ref{Ai} and \ref{AxP} consider the cases of $\e=\Ai^+$ and $\e=\Ass^+\x\pp$ for $\pp$ an operad.

\subsection{Open topological conformal field theories}\label{Cos}

Let $\OO$ be  the open cobordism category defined in \ref{OO}, with $i:\Ai^+\to\OO$ the inclusion of trees into all graphs, 
and $\OC$ the open-closed cobordism category of \ref{OCsec}. We have that $\OO$ is a
subcategory of $\OC$. The following lemma shows that $\OC$ is in fact an extension of the Hochschild core category of $\OO$:

\begin{figure}[h]
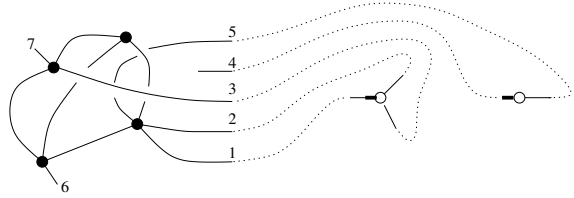

\begin{lpic}{decomp(0.5,0.5)}
\end{lpic}
\caption{Black and white graphs as elements in the iterated Hochschild complex of $\OO$}\label{decomp}
\end{figure}

\begin{lem}\label{OOext}
The category $\OC$ is an extension of $\bC\OO$. 
\end{lem}

\begin{proof}
We need to check that $\OC(\oc{m_1}{0},\oc{m_2}{n})\cong \bC^{n}(\OO(m_1,-))(m_2)$. Now 
$$\bC^{n}(\OO(m_1,-))(m_2)=\oplus_{k_1,\dots,k_{n}\ge 1} \OO(m_1,k_1+\dots+k_n+m_2)/U\ot L_{k_1}\ot\dots\ot L_{k_n}.$$
We describe a bijection between the generators of this complex and the generators of $\OC$: a generator of the complex above is identified with a black and white graph with $n$ white vertices and $m_1+m_2$ leaves by attaching the first $k_1+\dots+k_n$ 
outgoing leaves of generating graphs in $\OO$ to the leaves of the generating graphs $l_{k_1}, \dots,l_{k_n}$ of $L_{k_1}, \dots,L_{k_n}$, respecting
the ordering. (An example of this procedure is shown in Figure~\ref{decomp}.) The fact that
the only units allowed in $\OO$ are at the positions corresponding to the first leaf of an $L_{k_i}$ corresponds to the fact that the only unlabeled
leaves allowed in $\OC$ are those that are start-edges of white vertices. As the graphs
$l_{k_i}$ have a start-leaf, this is a reversible process whose target is exactly the generator of  $\OC(\oc{m_1}{0},\oc{m_2}{n})$. 
\end{proof}

Applying Theorem~\ref{action} to $\e=\OO$ with $\tCE=\OC$ then yields:

\begin{thm}\label{CosThm} 
Let $\Phi:\OO\to\Comp$ be an (h-)split symmetric monoidal functor. 
Then the pair $\big(\bC(\Phi)(0),\Phi(1)\big)$ is a $\Si$--equivariant (homotopy) $\OC$--module. 
\end{thm}

As morphisms in $\OO$ models the moduli space of cobordisms between 
({\em open strings}) (see Theorem~\ref{OOmod}), a split monoidal functor $\Phi:\OO\to\Comp$ is a model of an 
 {\em open topological conformal field theory}. Algebraically, such an object is an $\Ai$--version of a Frobenius algebra (see Section~\ref{Frobsec}). 
Similarly,  Theorem~\ref{cos_thm} shows that an equivariant $\OC$--module can be thought of as 
 a model for an open-closed topological conformal field theory. In particular, it includes an action of a chain model of the
moduli space of Riemann surfaces with fixed boundary parametrization on the value of the module at the circle.

We note that the category $\OC$ does \emph{not} include morphisms associated to the disk with one outgoing closed boundary component.  Consequently, algebras over the closed sector of this theory are not necessarily unital (the unit in the algebra would come from the generator of $H_0$ of the moduli of such disks).  That is, algebras over $\OC$ are inherently ``co-positive boundary'' topological conformal field theories.

The above theorem is essentially a reformulation of Costello's theorem \cite[Thm.~A (2-3)]{costello07}, though we 
obtain a more precise description of the action of the open-closed cobordism category. This allows us to recover 
the recipe given by Kontsevich-Soibelman for such an action in Section 11.6 of \cite{KS06}, which we expand on in the next section. 
Restricting to genus 0 surfaces, the statement includes the ``$\Ai$--cyclic Deligne conjecture'', which was also proved in \cite{War12}.

\subsection{Making the action explicit: the Kontsevich-Soibelman recipe} \label{ks_section}

Let $\Phi:\OO \to \Comp$ be a split monoidal functor with $\Phi(1)=A$.  Given $n_1$ Hochschild chains in $A$, $m_1$ elements  $A$ and a graph $\Ga$ in $\OC(\oc{m_1}{n_1},\oc{m_2}{n_2})$, that is: 
$$(a_0^1\ot\dots \ot a^1_{k_1}),\dots, (a_0^{n_1}\ot\dots \ot a^{n_1}_{k_{n_1}}),\  b_1,\dots, b_{m_1} \ \textrm{and}\  \Ga$$
the diagram in the proof of Theorem~\ref{action} gives an explicit way to obtain a sum of a tensor product of $n_2$ Hochschild chains in $A$ and $m_2$ elements of $A$. We apply here the sequence of maps given in the diagram to such elements and show how this recovers the recipe given by Kontsevich and Soibelman in \cite[pp 58--62]{KS06}. Figure~\ref{Aiproducts} below shows two examples of the construction. 

\begin{figure}[h]
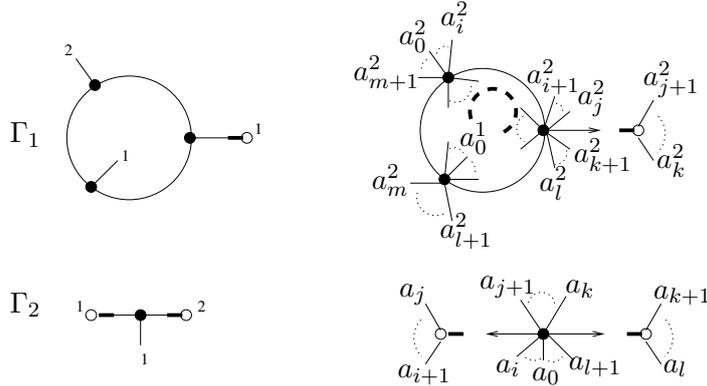

\begin{lpic}{products.Ainfty(0.5,0.5)}
 \lbl[b]{-10,58;$\Ga_1$} 
 \lbl[b]{-10,13;$\Ga_2$} 
\lbl[b]{108,56;$a^1_0$}
\lbl[b]{92,83;$a^2_0$}
\lbl[b]{103,88;$a^2_i$}
\lbl[b]{129,72;$a^2_{i+1}$}
\lbl[b]{139,66;$a^2_j$}
\lbl[b]{160,71;$a^2_{j+1}$}
\lbl[b]{160,50;$a^2_k$}
\lbl[b]{142,51;$a^2_{k+1}$}
\lbl[b]{129,44;$a^2_l$}
\lbl[b]{106,31;$a^2_{l+1}$}
\lbl[b]{86,45;$a^2_m$}
\lbl[b]{85,73;$a^2_{m+1}$}
\lbl[b]{127,-5;$a_0$}
\lbl[b]{117,-3;$a_i$}
\lbl[b]{95,-5;$a_{i+1}$}
\lbl[b]{92,16;$a_{j}$}
\lbl[b]{117,18;$a_{j+1}$}
\lbl[b]{136,18;$a_k$}
\lbl[b]{163,16;$a_{k+1}$}
\lbl[b]{161,-3;$a_l$}
\lbl[b]{140,-3;$a_{l+1}$}
\end{lpic}
\caption{General terms in a chosen product and coproduct induced by the graphs $\Ga_1$ and $\Ga_2$ on the Hochschild complex of $\Ai$--Frobenius algebras. The outputs are to be read along the white vertices after evaluating the operations defined by the graphs, which are elements of $\OO$, at their outputs (denoted by arrows in the figure).}\label{Aiproducts}
\end{figure}

The first map in the diagram  assembles all these terms as 
$$a_0^1\ot\dots \ot a^1_{k_1}\ot\dots\ot a_0^{n_1}\ot \dots \ot a^{n_1}_{k_{n_1}} \ot  b_1\ot\dots\ot b_{m_1}\ot l_{k_1\!+\!1}\ot\dots\ot l_{k_{n_{\!1}}\!+\!1}\ot \Ga$$
The following map, $\tilde\beta$, embeds these into the Hochschild complex of the bar construction. It gives terms of simplicial
degree 0 coming from
the canonical inclusion (adding an identity map in $\OO(k+m_1.k+m_1)$ to the above), plus additional terms of higher simplicial degrees.  
These elements of $C^{n_1}(B(\Phi,\OO,\OO))(m_1)$ are now reinterpreted as lying in 
$B(\Phi,\OO,\OC(-,\oc{m_1}{n_1}))$ just by considering $id_{k+m_1} \ot l_{k_1+1}\ot\dots\ot l_{k_{n_1}+1}$ as a graph with $n_1$ disjoint white
vertices of valences $k_1+1,\dots,k_{n_1}+1$ and $m_1$ additional disjoint leaves.

The bottom horizontal map in the diagram now glues this last graph to $\Ga$. The result of gluing is a sum of graphs $\Ga'$ which are obtained from
$\Ga$ by adding $k_i$ labeled leaves cyclically in all possible  manners on the $i$th closed incoming cycle of $\Ga$ for each $i$.  After
reinterpreting the new graphs as morphisms in $\OO$ attached to $n_2$ white vertices (as in Lemma~\ref{OOext}), the map $\al$---in simplicial degree 0---applies these morphisms of $\OO$ to the elements of $A$ and kills terms of higher simplicial degree. Finally, the resulting chain of $\Phi((k'_1+1)+\dots+(k'_{n_2}+1)+m_2)$ is reinterpreted as $n_2$ Hochschild chains in $A$ (around the white vertices) and $m_2$ elements of $A$. 
The terms of higher simplicial degrees produces by $\tb$ are killed by $\al$.

The appendix explains how to read signs for the operations. 
For  concrete examples of these operations in the case of a strict Frobenius algebra, we refer the reader to the end of section \ref{strict_section}.

\subsection{Twisting by the determinant bundle}\label{twist_section}

For a black and white graph $G$ defining a morphism in $\OC(\oc{m_1}{n_1},\oc{m_2}{n_2})$, we define its outgoing boundary $\del_{out}=\del_{out}G$ to be the union of its $n_2$
white vertices and the endpoints of its $m_2$ outgoing leaves, regarded as a subspace of the corresponding topological graph, also denoted $G$.  
We write $\det(G, \partial_{out})$ for the Euler characteristic of the relative homology
$H_*(G, \partial_{out})$, regarded as a graded abelian group:
$$\det(G, \partial_{out}) := \det(H_*(G, \partial_{out}))= 
\det(H_{0}(G, \partial_{out}))\ot \det(H_1(G, \partial_{out}))^*$$
here considered as a graded $\Z$--module, in degree $-\chi(G, \partial_{out})$.

For $d \in \Z$ , define a {\em $d$--orientation} for $G$ to be
a choice of generator of $$\det(\RR(V\sqcup H))\ot \det(G,\del_{out})^{\ot d}.$$ 
We define new categories $\OO_d$ and $\OC_d$ just like $\OO$ and $\OC$ but replacing the previously defined orientation of graphs by a $d$--orientation. 
So the objects of $\OO_d$ and $\OC_d$ are the same as those of $\OO$ and $\OC$, but the morphisms are now chain complexes generated by 
pairs $(G,o_d(G))$ for $G$ a graph
representing a morphism in $\OO$ or $\OC$ and $o_d(G)$ a $d$--orientation of $G$. 
The boundary of a $d$--oriented graph $(G,o_d(G))$ is the boundary of the graph $G$ as before together with the $d$--orientation induced
as before for its $\det(\RR(V\sqcup H))$--part.  For its $\det(G,\del_{out})$--part, we use the isomorphism between the determinant of $G$ and that of its boundary summand induced by a topological map (unique up to homotopy) contracting the blown-up of vertex with support in a small
neighborhood of that vertex. 

To define composition in $\OO_d$ and $\OC_d$, we need the following.  
Let $G_1,G_2$ be two graphs representing composable morphisms in $\OC$, with 
$(G_2\circ G_1)=\sum G$ their composition in $\OC$. As $G_2$ is a subgraph of each $G$, we have a triple 
$(G,G_2,\del_{out})$. Note also that $H_*(G,G_2)\cong H_*(G_1,\del_{out})$ as collapsing the copy of $G_2$ in any term $G$ of $G_2\circ G_1$ will
exactly recreate $G_1$ with its outer boundary collapsed.  
Given a short exact sequence of free abelian groups $U\hookrightarrow V\twoheadrightarrow W$, choosing a splitting $W\to V$ gives an isomorphism $\det(U)\ot \det(W)\to \det(V)$, and one can check that this isomorphism is independent of the choice of splitting. 
Now splitting the long exact sequence in homology for each triple $(G,G_2,\del_{out})$ into short exact sequences and
then choosing splittings for each of those short exact sequences, one gets an isomorphism
$$\det(G_1, \partial_{out}) \otimes \det(G_2, \partial_{out}) \to \det(G, \partial_{out})$$
for each term in the composition, which is natural and likewise independent of the choices of splitting. 
Explicitly, if $\beta$ denotes the connecting homomorphism in the long exact sequence associated to the  triple $(G,G_2,\del_{out})$, then this isomorphism is the following composition: 
\begin{align*}
& \det(H_{0}(G_1, \partial_{out}))\ot \det(H_1(G_1, \partial_{out}))^*\ot \det(H_{0}(G_2, \partial_{out}))\ot \det(H_1(G_2, \partial_{out}))^* \\
\to & \det(H_{0}(G_1, \partial_{out})) \ot \det(\ker\be)^*\ot \det(\operatorname{Im} \be)^*\ot \det(\operatorname{Im}\be)\ot \det(\operatorname{coker}\be)\ot \det(H_1(G_2, \partial_{out}))^*\\
\to & \det(H_{0}(G_1, \partial_{out}))\ot \det(\operatorname{coker}\be) \ot \det(\ker\be)^*\ot \det(H_1(G_2, \partial_{out}))^*\\
\to & \det(H_{0}(G, \partial_{out}))\ot \det(H_1(G, \partial_{out}))^*
\end{align*}
It is not difficult to check that this isomorphism is associative. 
One then can define composition in $\OO_d$ or $\OC_d$ as composition in $\OO$ or $\OC$, tensored with the $d^{\rm th}$ power of this
isomorphism. More precisely, 
the composition of $d$--oriented graphs $(G_1,o_d(G_1))$ and
$(G_2,o_d(G_2))$ 
is by the same gluing as before on the graphs, and via the composition 
$$\begin{array}{ll}& \det(\RR(V_1\sqcup H_1))\ot \det(G_1,\del_{out})^{\ot d}\ot \det(\RR(V_2\sqcup H_2))\ot \det(G_2,\del_{out})^{\ot d}\\
\to & \det(\RR(V_1\sqcup H_1))\ot \det(\RR(V_2\sqcup H_2))\ot \det(G_1,\del_{out})^{\ot d}\ot \det(G_2,\del_{out})^{\ot d}\\
\to &\det(\RR(V_1\sqcup H_1\sqcup V_2\sqcup H_2))\ot \det(G,\del_{out})^{\ot d}\end{array}$$
for each term $G$ in $G_2\circ G_1$, 
where the first arrow introduces a sign 
 $(-1)^{d|G_2|\chi(G_1,\del_{out})}$ and the second map is juxtaposition on the first factors as
in $\OO$ and $\OC$, and the $d$--power 
the above isomorphism on the last factors. 

The resulting categories $\OO_d$ and $\OC_d$ are symmetric monoidal.  

Note that $\OO_d$ admits a symmetric monoidal functor $i: \Ai \to \OO_d$, since  any
graph $G \in \Ai$  is a union of trees, each with exactly one outgoing boundary point, so $\det(G, \partial_{out})$ is of degree 0, with a canonical
generator, compatible under composition in $\Ai$.  Thus we are entitled to form the Hochschild complex of any functor $\Phi: \OO_d \to \Comp$.  
Lemma~\ref{OOext} extends to show that $\OC_d$ is an extension of $\bC\OO_d$, as the isomorphism 
$\bC^{n}(\OO(m_1,-))(m_2)\sta{\cong}{\rar}\OC(\oc{m_1}{0},\oc{m_2}{n})$ takes a graph in $\OO(m_1,k_1+\dots+k_n+m_2)$ to a graph with canonically isomorphic determinant in $\OC(\oc{m_1}{0},\oc{m_2}{n})$ as the added $l_i$'s are part of the outgoing boundary of the graph in $\OC$, glued to outgoing leaves of the graph in $\OO$. 
Hence by Theorem \ref{action}, we have 

\begin{cor}\label{detCor}

Let $\Phi:\OO_d\to\Comp$ be an (h-)split symmetric monoidal functor. 
Then the pair $\big(\bC(\Phi)(0),\Phi(1)\big)$ is a $\Si$--equivariant (homotopy) $\OC_d$--module. 

\end{cor}

\subsection{Positive boundary variations}\label{pos}

Recall from  \ref{posalg} the positive boundary subcategory $\OO^b \subseteq \OO$ whose morphisms are those satisfying that their
underlying surface has at least one outgoing boundary in each component.  
Define now $\OC^b \subseteq \OC$ to be the subcategory consisting of graphs with at least one outgoing boundary in each
component.  Recalling that the closed-to-closed morphisms of $\OC$ satisfy a ``co-positive''
boundary condition, namely that every component of the underlying surface has at least one incoming or free boundary, we have that the
closed-to-closed part of $\OC^b$ satisfies both the positive and free/co-positive boundary conditions.

\begin{lem}\label{OObext}
The category $\OC^b$ is an extension of $\bC(\OO^b)$.  
\end{lem}

\begin{proof}
Using the bijection in Lemma \ref{OOext}, we see that 
$$\bC^{n}(\OO^b(m_1,-))(m_2)=\oplus_{k_1,\dots,k_{n}\ge 1} \OO^b(m_1,k_1+\dots+k_n+m_2)/U\ot L_{k_1}\ot\dots\ot L_{k_n}$$
identifies with the subcomplex $\OC^b(\oc{m_1}{0},\oc{m_2}{n})$ of $\OC(\oc{m_1}{0},\oc{m_2}{n})$ as outgoing closed boundary components in $\OC$
correspond to non-empty outgoing boundary in each component of $\OO$ attached to it in the above decomposition.
\end{proof}

Applying Theorem~\ref{action} to $\e=\OO^b$ and $\tCE=\OC^b$ immediately gives

\begin{cor}\label{posCor}

If $\Phi:\OO^b \to \Comp$ is an (h-)split symmetric monoidal functor, then the pair $\big(\bC(\Phi)(0),\Phi(1)\big)$ is a $\Si$--equivariant (homotopy) $\OC^b$--module. 

\end{cor}

\subsection{Strict Frobenius algebras and Sullivan diagrams} \label{strict_section}

Recall from \ref{Frobsec} the category $H_0(\OO)$, whose morphisms are the 0-th homology groups of those of $\OO$, and which has the property that
$H_0(\OO)$--algebras are exactly  (strict) symmetric Frobenius algebras. We consider also the shifted version $H_{bot}(\OO_d)$ whose morphisms are
the bottom homology groups in each component of the morphisms of the category $\OO_d$ of Section~\ref{twist_section}, i.e. 
$$H_{bot}(\OO_d)=\coprod_{S\in\pi_0(\OO(n,m))}H_{-d\cdot\chi(S,\del_{out})}(\OO_{d,S}(n,m)).$$ 
We call $H_{bot}(\OO_d)$--algebras {\em dimension $d$ Frobenius algebras}. 

We show in this section that the category of Sullivan diagrams $\SD$ of Section~\ref{SDsec} is an extension of $C(H_0(\OO))$, and a shifted version
$\SD_d$ of $\SD$ is an extension of $C(H_{bot}(\OO_d))$, which gives the action of Sullivan diagrams on the Hochschild complex of Frobenius
algebras stated in Theorem~\ref{strict}. We then give explicit formulas for the product, coproduct and $\De$--operator on the Hochschild 
complex of Frobenius algebras coming out of our method, and check in Proposition~\ref{std_BV_prop} that, over a field, the Batalin-Vilkovisky coalgebra structure given by the
coproduct and $\De$--operator on Hochschild homology is dual to the Batalin-Vilkovisky structure on the Hochschild cohomology of the algebra defined
using the cup product and the dual to Connes' operator $B$. 

\smallskip

As already remarked in \ref{SDsec}, 
the components of the category $\SD$ of Sullivan diagrams are in 1-1 correspondence with those of $\OC$, namely the topological types of open-closed
cobordisms. For $S$ such a topological type, we denote by $\SD_S(\oc{m_1}{n_1},\oc{m_2}{n_2})$ the corresponding component. 
We define $\SD_d$ to be the dg-category obtained from $\SD$ by shifting the degree of the component $\SD_S(\oc{m_1}{n_1},\oc{m_2}{n_2})$ by 
$d.\chi(S,\del_{out})$; note that the shifts in degree are consistent with composition.
(The category $\SD_d$ is a  quotient of the category $\OC_d$ of \ref{twist_section}.)

\begin{lem} \label{SD_lem}
The category $\SD$ is an extension of $\bC(H_0(\OO))$ and more generally, $\SD_d$ is an extension of $\bC(H_{bot}(\OO_d))$. 
\end{lem}

\begin{proof}
We have 
$$\bC^{n}(H_0(\OO)(m_1,-))(m_2)=\oplus_{k_1,\dots,k_{n}\ge 1} H_0(\OO(m_1,k_1+\dots+k_n+m_2))/U\ot L_{k_1}\ot\dots\ot L_{k_n}$$
whose generators, by gluing the graphs in $\OO$ to the white vertices in the $L_{k_i}$'s, 
are black and white graphs with trivalent black vertices modulo the equivalence relation coming 1-cells in $\OO(m_1,k_1+\dots+k_n+m_2)$, i.e. from
blowing up 4-valent black vertices. But this corresponds exactly to the description of 
$\SD(\oc{m_1}{0},\oc{m_2}{n})$ in terms of quotient of black and white graphs given by
Theorem~\ref{SDequivthm}. 

Replacing $H_0(\OO)$ by $H_{bot}(\OO_d)$ in the above, we get $\SD_d(\oc{m_1}{0},\oc{m_2}{n})$ instead as the shifts in degree are the same. 
\end{proof}

For $\e=H_{bot}(\OO_d)$, taking $\tbCE=\SD_d$, Theorem~\ref{action} thus gives

\begin{thm}\label{strict}
Let $A$ be a symmetric Frobenius algebra of dimension $d$, then the pair $(\bC(A),A)$ is a $\Si$--equivariant $\SD_d$--module, where
$\bC(A)$ denotes the reduced Hochschild complex of the algebra $A$. 
\end{thm}

As a differential graded algebra with a non-degenerate inner product defines a symmetric Frobenius algebra, this recovers Theorem
3.3 of \cite{Tradler-Z} after dualization.  (See also \cite{TraZei07} which considers the open part as well as the closed part).

\smallskip

Using Theorem~\ref{factorization}, a consequence of Proposition~\ref{strictvanish} and the above theorem is the following:  

\begin{cor} \label{factor_action_cor}
For strict symmetric Frobenius algebras $A$, the TCFT structure on $\bC_*(A)$ defined by Costello and Kontsevich-Soibelman factors through an action of Sullivan diagrams. In particular, stable classes in the homology of the moduli space act trivially. 
\end{cor}

This results puts together the work of Costello and Kontsevich-Soibelman with that of Tradler-Zeinalian: we have shown that Costello's construction (which translates to that of Kontsevich-Soibelman when made explicit) of an action of moduli space on the Hochschild homology of a strict Frobenius algebra factors through an action of the complex of Sullivan diagrams as constructed by Tradler-Zeinalian \cite[Thm 3.3]{Tradler-Z}.

\medskip

We are also now able to give a proof of Proposition~\ref{non-vanish} which says that the projection map from $\OC$ to $\SD$, on the component of the multi-legged pair of pants with one incoming and $p$ outgoing boundary components, is a quasi-isomorphism.

\begin{proof}[Proof of Proposition~\ref{non-vanish}]
By Lemma~\ref{OOext}, we have that $\OC(\oc{0}{1},\oc{0}{p})\cong\OC(\oc{1}{0},\oc{0}{p})$ is the iterated Hochschild complex
$\bC^p(\OO(1,-)(0)$. Likewise, by Lemma~\ref{SD_lem}, we have that $\SD(\oc{0}{1},\oc{0}{p})\cong \SD(\oc{1}{0},\oc{0}{p})$ is the
iterated Hochschild complex $\bC^p(H_0\OO(1,-)(0)$. 
We have $\bC^p(\OO(1,-)(0)=\oplus_{k_1,\dots,k_p}\OO(1,k_1+\dots+k_p)\ot_{\Ai}(L_{k_1}\ot \dots\ot L_{k_p})$ and 
 the components of this complex corresponding to a surface of genus 0 with $p+1$ boundary components 
is the subcomplex $\oplus_{k_1,\dots,k_p}A(k_1,\dots,k_p)\ot_{\Ai}(L_{k_1}\ot \dots\ot L_{k_p})$ with $A(k_1,\dots,k_p)\subset \OO(1,k_1+\dots+k_p)$ the subcomplex of forests with the property that, once glued to $l_{k_1},\dots,l_{k_p}$, they form a tree (with $p$ white vertices). This is a condition on the labeling of the leaves of the forest. For  $\bC^p(H_0\OO(1,-)(0)$, we have a similar subcomplex $B(k_1,\dots,k_p)\subset H_0\OO(1,k_1+\dots+k_p)$ giving the corresponding component. Now the projection $\OO\to H_0\OO$ induces a quasi-isomorphism $A(k_1,\dots,k_p) \to B(k_1,\dots,k_p)$ for each $k_1,\dots,k_p$. Indeed, the latter complex is a free abelian group on its graph generators which are all in degree 0, and for each such generator, which is a union of trees, there is in $A$ a product of the corresponding cellular complex of the associahedra, which are contractible. 
Hence the map we are interested in can be described as 
$$\oplus_{k_1,\dots,k_p}A(k_1,\dots,k_p)\ot_{\Ai}(L_{k_1}\ot \dots\ot L_{k_p})\rar \oplus_{k_1,\dots,k_p}B(k_1,\dots,k_p)\ot_{\Ai}(L_{k_1}\ot \dots\ot L_{k_p})$$
induced by a quasi-isomorphism of multi-functor $A\arsim B$. 
The then result follows from a multivariable version of Proposition~\ref{qi}. 
\end{proof}

\medskip

The action on the Hochschild complex given by Theorem~\ref{action} is easy to implement explicitly in the case of strict
Frobenius algebras because operations involve fewer terms than in the general case. Figure~\ref{products} (a-c) gives examples of graphs representing the product (pair of pants with two inputs and one
output), the coproduct (pair of pants with one input and two outputs) and the $\Delta$--operator (degree 1 operator with one closed
input and one closed output). We give now the explicit formulas for the action of these graphs on the Hochschild complex of a strict
Frobenius algebra. Note that, because these operations are images of corresponding operations in $\OC$ generating a BV and a co-BV structure, we know that the product and $\Delta$ as well as the coproduct and $\De$ satisfy the BV relation. By Proposition~\ref{non-vanish}, the co-BV structure is a priori a non-trivial one. On the other hand, as we will see, the product in $\SD$ is rather trivial and hence the corresponding BV structure is rather trivial. 
We refer to Remark~\ref{GHrem} below for a product giving a less degenerate BV-structure in the case of commutative Frobenius algebras. 

\medskip

Let $A$ be a strict symmetric Frobenius algebra. 
To obtain the action of a (sum of) graph(s)  $G$ representing a chain in $\SD_d$, on a chain in the Hochschild complex of
$A$, we need to follow the prescription laid out in Section \ref{ks_section} (together with the appendix Section~\ref{signsub} for the signs).
In Figure~\ref{products}(a-c), we have made a choice of an ordering of the vertices, and of an orientation of the edges. The chosen
orientation of each graph is then the orientation corresponding to considering the graph as a composition of the operations at each
vertex in this ordering, with their canonical orientation (see Section~\ref{signsub}). 
Figure~\ref{products}(a'-c') shows the non-trivial graphs created when applying the procedure described in Section \ref{ks_section}.

\begin{figure}[h]
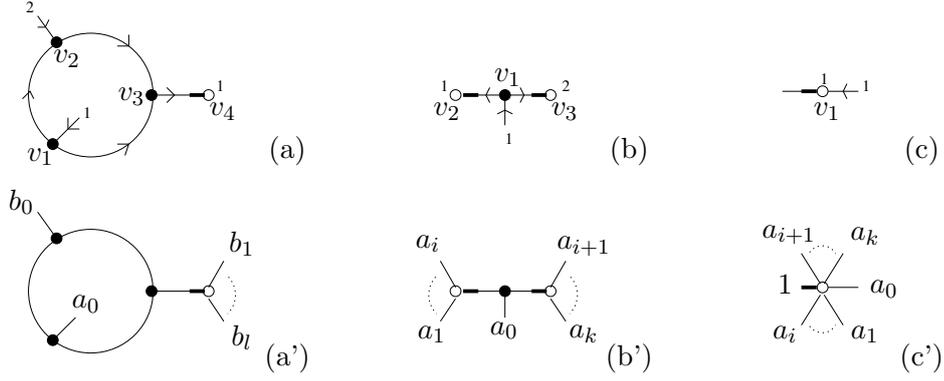

\begin{lpic}{products.2(0.5,0.5)}
 \lbl[b]{70,50;(a)}
 \lbl[b]{160,50;(b)}
 \lbl[b]{237,50;(c)}
 \lbl[b]{70,-5;(a')}
 \lbl[b]{160,-5;(b')}
 \lbl[b]{237,-5;(c')}
 \lbl[b]{5,50;$v_1$}
 \lbl[b]{12,76;$v_2$} 
 \lbl[b]{29,66;$v_3$}
 \lbl[b]{53,62;$v_4$}
 \lbl[b]{128,71;$v_1$}
 \lbl[b]{112,62;$v_2$} 
 \lbl[b]{143,62;$v_3$}
 \lbl[b]{212,62;$v_1$}      
\lbl[b]{17,11;$a_0$}
\lbl[b]{0,37;$b_0$}
\lbl[b]{58,26;$b_1$}
\lbl[b]{58,1;$b_l$}
\lbl[b]{127,4;$a_0$}
\lbl[b]{108,3;$a_1$}
\lbl[b]{107,27;$a_i$}
\lbl[b]{148,26;$a_{i+1}$}
\lbl[b]{148,3;$a_k$}
\lbl[b]{227,15;$a_0$}
\lbl[b]{201,16;$1$}
\lbl[b]{222,3;$a_1$}
\lbl[b]{201,3;$a_i$}
\lbl[b]{202,29;$a_{i+1}$}
\lbl[b]{222,29;$a_k$}
\end{lpic}
\caption{Representing graphs for the product, coproduct and $\Delta$--operator}\label{products}
\end{figure}

We denote as before a $k$--chain in the Hochschild complex of $A$ 
by $a_0\ot\dots\ot a_k$. Using the convention for the product and coproduct given in Section~\ref{signsub},
the graphs of
Figure~\ref{products} induce the following operations on $C_*(A)$: 

\smallskip

{\bf (a) Product:} 
$$(a_0\ot\cdots\ot a_k)\ot(b_0\ot\cdots\ot b_l)\mapsto
\left\{ \begin{array}{ll}0& k>0\\
\sum (-1)^{\epsilon}a_0''a_0'b_0\ot b_1\ot\cdots\ot b_l & k=0 
\end{array}\right.$$
where $\sum a_0'\ot a_0''$ denotes the coproduct of $a_0$ as an element of the
Frobenius algebra $A$ and 
$$\epsilon=|a_0'||a_0''|+d(|b_0|+\cdots+ |b_l|+l).$$
(The main part of this computation is done in detail in the appendix.)
Note in particular that, as the product is homotopy commutative, in homology it is 0
except on $HH_0(A,A)\ot HH_0(A,A)$.

\smallskip

{\bf (b) Coproduct:} 
$$(a_0\ot\cdots\ot a_k)\mapsto \sum (-1)^{\epsilon}(a''_0\ot a_1\ot\cdots\ot a_i)\ot(a_0'\ot a_{i+1}\cdots\ot a_k)$$
where $\epsilon=d(|a_1|+\dots+|a_k|+k)$. 

\smallskip

{\bf (c) $\Delta$--operator:} 
$$(a_0\ot\cdots\ot a_k)\mapsto \sum (-1)^\epsilon1\ot a_{i+1}\ot \cdots\ot a_k\ot a_0\ot a_1\ot \cdots\ot a_i$$
where $\epsilon=(|a_0|+\dots+|a_i|)(|a_{i+1}|+\dots+|a_k|)+ik$.

\begin{prop} \label{std_BV_prop}

If $A$ is a strict graded symmetric Frobenius algebra over a field $k$, the coproduct and $\Delta$ make $HH_*(A, A)$ into a Batalin-Vilkovisky coalgebra.  Moreover, this structure is dual to the BV-algebra structure on $HH^*(A, A)$, where the product is the cup product of Hochschild cochains, and the BV operator is dual to Connes' $B$--operator.

\end{prop}

The first part of this
proposition, before going to homology, recovers the cyclic Deligne conjecture as proved in \cite{kaufmann,Tradler-Z, paolo_cyclic}. 

The duality in this proposition is given on the chain level by a chain isomorphism $CH^*(A, A) \to \Hom(CH_*(A, A), k)$.  Degree-wise this is given by the map
$$\begin{array}{cc} \Hom(A^{\otimes n} , A) \to \Hom(A^{\otimes n+1}, k), & f \mapsto \widetilde{f} \end{array}$$
where $\widetilde{f}(a_0, \dots, a_n) = \langle a_0, f(a_1, \dots, a_n) \rangle$.

\begin{proof}

A BV coalgebra is an algebra over the cooperad whose $k$--ary operations are given by the homology of the moduli space of Riemann surfaces of genus 0
with one incoming and $k$ outgoing closed boundary components, with composition induced by gluing. As the corresponding component of
$\SD(\oc{0}{1},\oc{0}{k})$ is quasi-isomorphic to that of $\OC(\oc{0}{1},\oc{0}{k})$, the first part of the statement follows, independently of the
second part, from Theorems~\ref{cos_thm}
and \ref{strict}.

Now the duality carries $\Delta$ to $B$, since the $\Delta$--operator in $HH_*(A, A)$ given in (c) is precisely $B$, and the $\Delta$--operator on $HH^*(A, A)$ is defined by
transferring $B^*$ via $f \mapsto \widetilde{f}$.  (The signs in the formula for $B$ given in \cite[Sect.~2.4]{felix_thomas} differs from ours due to
different conventions. They match if we introduce a factor $(-1)^{a_0+\dots+a_k+k}$ passing the generator of $H_1(S^1)$ on the other side of the
Hochschild complex, and a factor $(-1)^{a_1+2a_2+\dots+ka_k}$ before and after the operation to compare the Hochschild complexes---this last factor sets the
degree $k$ shift of the Hochschild complex in between the $a_i$'s instead of at the end as we have it). 

So it suffices to check that the coproduct in (b) (which we will write as $\nu$) is dual to the
Hochschild cup product.  Let $f$ and $g$ be two Hochschild cochains; then (up to sign issues as above) 
\begin{eqnarray*}
\widetilde{f \cup g}\,(a_0, \dots, a_{p+q}) & = & \pm\, \langle a_0, f(a_1, \dots, a_p) \cdot g(a_{p+1}, \dots, a_{p+q}) \rangle \\
 & = & \pm \sum \langle a_0'', f(a_1, \dots, a_p)\rangle \cdot \langle a_0',  g(a_{p+1}, \dots, a_{p+q}) \rangle \\
& = &  \pm\, \nu^*(\widetilde{f} \otimes \widetilde{g})(a_0, \dots, a_{p+q})
\end{eqnarray*}
where the first equality is the definition and the third from the formula given in (b) above.  The second follows from
Figure~\ref{dually}, below, which relates the coproduct and product in the Frobenius algebra $A$ via the pairing. 
\vspace{0.4cm}
\begin{figure}[h]
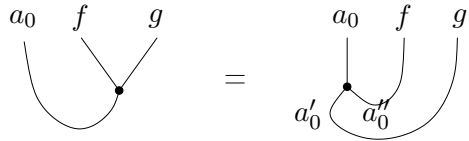

\begin{lpic}{dually(0.5,0.5)}
\lbl[b]{0,30;$a_0$}
\lbl[b]{15,29;$f$}
\lbl[b]{35,30;$g$}
\lbl[b]{85,30;$a_0$}
\lbl[b]{100,29;$f$}
\lbl[b]{115,30;$g$}
\lbl[b]{75,3;$a_0'$}
\lbl[b]{93,3;$a_0''$}
\end{lpic}
\caption{Duality of the cup product and coproduct}\label{dually}
\end{figure}
\end{proof}

\begin{rem}\label{GHrem}{\rm 
The product defined above is rather degenerate. If we assume that the Frobenius algebra is in addition commutative, then there is a less degenerate product
of degree 1, which also is part of a (now shifted) BV structure---see
\cite[Thm 4.7 and Cor 4.8]{Kla13B}. This product was also studied by Abbaspour on the homology level, see \cite[Sec 7]{Abb13A} or \cite[Thm 6.1]{Abb13B} and is expected to be related to the Goresky-Hingston product on the cohomology of the free loop space on a manifold \cite{GorHin09}. 
}\end{rem}

\subsection{String topology}\label{strings}

We apply in this section the results of the previous sections -- particularly Theorem \ref{action} and Corollary \ref{factor_action_cor} -- in order to control the (not yet entirely understood) operations in string topology in characteristic 0.  

Let $C^*(M)=C^*(M;\QQ)$ denote the {\em rational} singular cochain complex of a compact, oriented, simply connected manifold, and $H^*(M)$ its cohomology. 
It is well known (see \cite{jones}) that there is an isomorphism 
$$H^{-*}(LM) \cong HH_*(C^{-*}(M), C^{-*}(M))$$
from the cohomology of the free loop space $LM$ to the Hochschild homology of $C^{-*}(M)$, the cochains of $M$ seen as a chain complex in negative
degree.  $H_*(LM)$ is equipped with the structure of a BV-alegbra, extending to an action of certain spaces of non-degenerate string diagrams  \cite{CS99,CohJon02,CohGod,Cha05}. This structure has been expected to extend to the structure of a 
(positive boundary) homological conformal field theory (HCFT), i.e.~action of the moduli space of 
Riemann surfaces, possibly compactified, see eg.~\cite{ChaSul04,godin07,Sul07,Poi10,PoiRou,DPR15}. 
We will give here such an extension using the above algebraic model of $H^*(LM)$. It bears pointing out that this construction does not obviously agree with the geometric constructions above, although the underlying BV structures do agree.  See Remark \ref{luc_remark} for a subtler example.

We follow the prescription laid out by Lambrechts-Stanley \cite{lambrechts_stanley} and Felix-Thomas \cite{felix_thomas} to construct
this structure in Hochschild homology.  We note that $C^*(M)$ is quasi-isomorphic to a (simply connected) commutative differential
graded algebra $A$ (e.g., the algebra of differential forms on $M$ with rational coefficients), and that $H^*(A) \cong H^*(M)$ is a
strict Frobenius algebra.  Lambrechts-Stanley give a recipe for constructing, for any such $A$, a weakly equivalent algebra $B$ which
is itself a commutative differential graded Frobenius algebra; that is, $B$ itself satisfies Poincar\'e duality prior to application of
cohomology. If $M$ has dimension $d$, such an algebra $B$ is a dimension $d$ symmetric Frobenius algebra in our sense, that is it defines a functor 
$\Phi:H_{bot}(\OO_d)\to \Comp$. 
In particular, we can apply Theorem~\ref{strict} to $B$ and get an
action of Sullivan diagrams on its Hochschild homology.

Using the chain of isomorphisms
$$H^{-*}(LM) \cong HH_*(C^{-*}(M), C^{-*}(M)) \cong HH_*(A^{-*}, A^{-*}) \cong HH_*(B^{-*}, B^{-*}) \leqno{(*)}$$
we get an action of Sullivan diagrams on $H^*(LM)$, and hence an HCFT 
by precomposition with the map $\OC\to\SD$. We do not know for sure that this (somewhat collapsed) action is the one constructed by Godin, but
the next proposition says that it is an extension of Chas-Sullivan's string topology:

\begin{prop}\label{string_prop}

The co-BV operations on $H^*(LM, \QQ)$ dual to the Chas-Sullivan string topology BV operations on $H_*(LM, \QQ)$ of \cite{CS99} extend to an action of the closed
part of $H_{-*}(\SD_{-d},\QQ)$ (for $d = \dim M$) using Theorem \ref{strict}. 

\end{prop}

\begin{proof}
The co-BV structure (and $H_*(\SD_d)$--structure) we define on $H^*(LM,\QQ)$ is defined via an action on $HH_*(B,B)$, hence it is
equivalent to check that the action on  $HH_*(B, B)$ is dual to the string topology action. By Proposition \ref{std_BV_prop}, our co-BV structure is dual to
the BV structure on $HH^*(B,B)$ coming from the Hochschild cup product and the dual of Connes' operator $B$. Hence,  
by \cite[Prop.~1]{felix_thomas}, our structure is carried to the dual of the Chas-Sullivan structure by the isomorphism $(*)$.
\end{proof}

The HCFT structure we produce on  $H^*(LM,\QQ)$ is an action of moduli spaces of Riemann surfaces factoring through an action of Sullivan diagrams,
which immediately implies that  a substantial part of the action is trivial (see Proposition~\ref{strictvanish}). 
 In particular, we know that all stable classes in the homology of the moduli space act trivially, a fact known in the string topology setting by work
 of Tamanoi \cite{tamanoi}.  

\begin{rem} \label{luc_remark} {\rm 

It is worth issuing a caveat here: the main result of \cite{luc} implies that the BV structures on 
$$\begin{array}{ccc}
H_*(LS^2; \FF_2) & and & HH^*(H^*(S^2; \FF_2), H^*(S^2; \FF_2))
\end{array}$$
cannot be isomorphic (even though the underlying Gerstenhaber structures are), if we equip $H^*(S^2; \FF_2)$ with the Frobenius algebra structure coming from Poincar\'e duality.  Consequently, we cannot expect the construction given above to yield the HCFT structure on string topology if done integrally.

}\end{rem}

\begin{rem}{\rm 
As the Lambrechts-Stanley models for $C^*(M)$ used above are actually  commutative Frobenius algebras, the action of Sullivan diagrams on $H^*(LM)$ constructed here actually factors through an action of the complex of looped diagrams later constructed by Klamt in \cite{Kla13B}. The map from Sullivan diagrams to looped diagrams is not surjective on either the chain or homology level, so this gives new operations such as the higher coproduct already mentioned in Remark~\ref{GHrem}.  The map from Sullivan diagrams to looped diagrams corresponding to considering commutative Frobenius algebras as symmetric Frobenius algebras, is expected to be essentially injective, both on the chain and the homology level, so this further factorization is not expected to give much in terms of vanishing results. 
}\end{rem}

\subsection{A dual perspective and relationship to the work of Kaufmann-Penner} \label{KP_section}

In \cite{KauPen}, Kaufmann and Penner give a model of open and closed ``string interaction'' using arc systems in surfaces. 
We discuss here how their model relates to the open-closed cobordism category $\OC$ and the category of Sullivan diagrams $\SD$ occuring in the present paper, and gives to a dual approach to string topology. 

\medskip

The open-closed category $\OC$ is build out of fat graphs. A fat graph can be defined as an equivalence class of graphs embedded in a fixed surface $F$ of the same topological type, up to isotopy,  in such a way that $F$ is just a thickening of the graph. The equivalence relation is given by the action of the mapping class group of $F$. Now dual to such a graph is a system of arcs in $F$ going from boundary to boundary and cutting the surface into polygons---such families of arcs are called {\em filling}. To model the moduli space of Riemann surfaces, one can equivalently work with either fat graphs or equivalence classes of filling, or even {\em quasi-filling}\footnote{A family of arcs in a punctured surface is quasi-filling if the complements of the arcs in the surface is a union of polygons and once punctured polygons.}, arc systems. 

When modeling open-closed cobordisms, one starts with a {\em windowed} surface $F$, just as in Remark~\ref{KaufSD}, that is $F$ comes equipped with a marked point for each open and closed boundary component and a corresponding ``in'' and ``out'' labeling. Let $\De=\De_0\coprod\De_1\subset \del F$ denote the set of these marked points, with $\De_0$, a set of points alone in their boundary components corresponding to closed boundaries, and $\De^1$ the set of points corresponding to the open boundaries. Kaufmann and Penner work with arc families in such surfaces, and in their model, the open boundaries are the intervals that are in between the points of $\De_1$. This is motivated by the fact that the arcs have their endpoints in the complement of $\De$, and that arcs at such ``open windows'' should model the evolution of an open string. The dual graph on the other hand will have its endpoints at $\De$; our interpretation in this paper (just like in \cite{costello07}) is that these endpoints of leaves in the graph model the open boundaries. This is motivated by the model of the open cobordism category we worked with, and in particular the fact that the gluing along leaves model the gluing in moduli space induced by juxaposition of polygonal decomposition of the surfaces as explained in Section \ref{OO}.  This means that, in this dual picture, the role of the ``open'' and ``free'' intervals is switched!  The fact that such a switch in the roles of the intervals is possible comes from the fact that the diffeomorphism group which fixes a point on a boundary component is homotopy equivalent to the subgroup fixing the whole boundary. Hence the boundary components containing open boundaries can be considered as completely fixed, and if we subdivide a boundary component as a succession of ``open'' and ``free'' intervals, there is no difference from the point of view of moduli space between the one set of interevals (the open) and the other (the free). 
\begin{figure}[h]
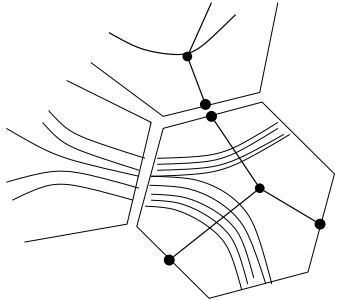

\begin{lpic}{KPdual(0.5,0.5)}
\end{lpic}
\caption{Arc family and dual fat graph}\label{KPdual}
\end{figure}

(We refer to \cite[Sec 3.2 and 4]{Ega15} for more details about the relationship between fat graphs and filling arc families in the bordered case and the relationship to black and white graphs, and \cite{Pen87} for a direct proof that quasi-filling arc families model the moduli space of Riemann surfaces.) 

\medskip

The above paragraphs indicate that we can switch between the fat graph and arc system models, but given that the role of open and free boundaries switches under this duality when there are open boundary components, the gluing along open boundaries is not going to be compatible, simply because we glue at different places! We expect nevertheless that these two  gluing are closely related, as we explain now. 

The idea of the gluing used by Kaufmann and Penner (first introduced in \cite{KauLivPen}) is to think of a weighted arc family, a point in the space of arc families,  as a collection of ribbons whose widths are given by the weights, and then gluing two such families by scaling so that the total weights match, possibly then discarding arcs that are not of an appropriate type. As arcs correspond to edges of the dual fat graphs, this gluing corresponds to a gluing that identifies edges in fat graphs. Gluing quasi filling arc families this way does not always produce a quasi filling arc family, but when gluing along closed boundaries and restricting to arcs families corresponding to admissible fat graphs, the gluing stays within admissible fat graphs and corresponds under the duality to the gluing of admissible metric fat graphs along boundary components which models the gluing of moduli space \cite[Lem 3.36, Lem 3.39 and Thm 3.30]{Ega15}. Moreover, this closed gluing 
 corresponds to the gluing of black and white graphs used in the present paper under the equivalence between admissible fat graphs and black and white graphs of \cite[Thm 4.41]{Ega15}. 

The gluing along open boundaries proposed by Kaufmann and Penner has not been studied  much yet, but 
it is likely that one can define more generally a category of ``admissible open-closed fat graphs'', with the corresponding admissible arc families under the duality, still modelling the moduli space of Riemann surfaces and with the property that the Kaufmann-Penner gluing  is well-defined and does model the gluing of moduli spaces also along open boundaries. In terms of fat graphs, this gluing would on open boundaries correspond to identifying certain edge sequences in between leaves in the graphs and thus would be as such different than the open gluing along leaves used in the present paper---we would again be gluing at a different place in the graph. However we expect that the result would simply be a different model of the open-closed gluing on moduli space once the role of the open boundaries is switched again. Note that if such a category of open-closed fat graphs exist, the open part will define a prop with $\Ai$--multiplication and by Theorem 3.1 and Corollary 2.2 of \cite{Wah12}, we will have that this category of admissible open-closed fat graphs is quasi-isomorphic to the prop of formal operations on the Hochschild complex of algebras over its open part.

\medskip

In the present paper, we have applied our open-closed cobordism category $\OC$ to string topology using on the open part a model of the algebras $C^*(M)$. This resulted in, at least rationally, a structure of algebra over $\OC$ factoring through our category $\SD$ of Sullivan diagrams for the pair $(C^*(M),C^*(LM))$, with  Jones' Hochschild model of $C^*(LM)$. 
As we have seen in Remark~\ref{KaufSD}, arc families where the arcs go from ``in'' to ``out'' boundaries, without any filling condition, and such that each outgoing boundary has arcs, correspond under the same duality as above to Sullivan diagrams when restricting to surfaces with only closed boundaries. For open boundaries, this duality defines a new version of {\em open Sullivan diagrams} (which should be a quotient of the above open admissible fat graphs). 
Given the nature of these open Sullivan diagrams, it is natural to expect that $C_*(\Om M)$, the chain complex of the based loop space $\Omega M$,  is an algebra over them. As these open Sullivan diagrams form a prop with $\Ai$--multiplication, 
one can then ask whether the whole open-closed category of arc families of Sullivan type is an extension, or at least up to quasi-isomorphism, of the Hochschild category of its open part (in our terminology). This would fit with the moduli space model proposed above as well as with our string topology computation in the previous section, and the fact that $C_*(C_*(\Om(M)),C_*(\Om(M)))$ is a model for $C_*(LM)$\cite{Goo85}.  The generalization to more branes should also allow to use more general path spaces in $M$. 
Such a construction should then recover \cite[Cor 6.7]{Kau10}.

\subsection{Hochschild homology of unital $\Ai$ algebras}\label{Ai}

In this section, we briefly consider what our construction gives when applied to 
the category $\e = \Ai^+$, equipped with the identity functor $id: \Ai^+ \to \Ai^+$.  

\begin{prop} \label{ann_prop}

The Hochschild complex $\bC^p(\Ai^+(m,-))(n)$ is isomorphic to the (split) subcomplex of $(\bar p,m+n)-\textrm{Graphs}$ consisting of fat graphs whose associated surface is a disjoint union of 

\begin{itemize} 

\item $n$ disks, each with precisely one outgoing open boundary, and 
\item $p$ annuli, each with precisely one closed outgoing boundary,

\end{itemize}

\noindent and with $m$ incoming open boundaries distributed on the free boundaries of these.  

\end{prop}

\begin{proof}
The gluing map
$$\bigoplus_{n_i\ge 1}\Ai^+(m,n_1+\dots+n_p+n)/U_{I} \ot L_{n_1}\ot\dots\ot L_{n_p} \to (\bar p,m+n)-\textrm{Graphs}$$
produces graphs which are a disjoint union of trees and trees attached to white vertices (see Figure~\ref{Ann}); the associated surfaces are as described.
\end{proof}

We therefore define an extension $\Ann$ of $\bC\Ai^+$ to be the subcategory of $\OC$ consisting of graphs whose associated surface is a disjoint union
of surfaces as in \ref{ann_prop}, or a closed-to-closed annulus.  Note that we cannot introduce any closed-to-open annuli in $\Ann$, for composites
would produce open-to-open morphisms that are not already present\footnote{Similarly there are no disks with a closed incoming boundary, since
  compositions would produce an open-to-open morphism with codomain $0$.} in $\bC \Ai^+$.  As $\Ann$ is an extension of $\bC \Ai^+$,  
by Theorem \ref{action}, we conclude:

\begin{thm}\label{Ann_Thm}

For any $\Ai^+$--algebra $A$, the pair $(\bC(A), A)$ is an $\Ann$--module.

\end{thm}

We examine the resulting $H_*(\Ann)$--structure on the pair $(HH_*(A, A), H_*(A))$, for $A$ a unital $\Ai$--algebra.

$\Ann$ evidently contains $\Ai^+ = \Ann \cap \OO$, and so the open sector of an $\Ann$--module remains (unsurprisingly) a unital  $\Ai$--algebra.  This equips $H_*(A)$ with the structure of a unital associative ring.  Write $m \in H_0(\Ann(\oc{2}{0}, \oc{1}{0}))$ for the class corresponding to the product, and $u \in H_0(\Ann(\oc{0}{0}, \oc{1}{0}))$ for the class corresponding to the unit.

Furthermore, since the mapping class group of an annulus with fixed boundaries is isomorphic to $\Z$, generated by the Dehn twist, the morphism
complex $\Ann(\oc{0}{1}, \oc{0}{1})$ is quasi-isomorphic to $C_*(B\Z) = C_*(S^1)$.  Up to homotopy, the only nontrivial operation $\oc{0}{1} \to
\oc{0}{1}$ is thus a class $\Delta$ of degree 1, corresponding to the fundamental class of the circle. This is Connes' operator $B$ explicitly given
at the end of Section~\ref{strict_section} (see Proposition~\ref{std_BV_prop})\footnote{Note here that the formula is the same for $\Ai$--algebras as for strictly associative
algebras as there are no black vertices in the graph generating the operation $\Delta$.}.   

One should also consider the interaction of the open and closed sectors.  There are no closed-to-open morphisms in $\Ann$, but there is a class $i \in
H_0(\Ann(\oc{1}{0}, \oc{0}{1}))$ coming from the annulus with one open incoming and one closed outgoing boundary. This map 
 $i:H_*(A) \to HH_*(A, A)$ is induced by the quotient map $A \to HH_0(A, A)$. 

\begin{prop}

The category $H_*(\Ann)$ is generated as a symmetric monoidal category by the operations $m$, $u$, $\Delta$, and $i$.

\end{prop}

\begin{rem}{\rm 

The Hochschild complex of a category $\e$ is functorial in $\e$; furthermore, it is not hard to see that a monoidal quasi-isomorphism $\e \to \e'$
induces a quasi-isomorphism of Hochschild complexes (using, e.g. the spectral sequence of a bicomplex).  Consequently the results above apply equally
to the category associated to the operad $Ass^+$ of unital associative algebras, since it is quasi-isomorphic to $\Ai^+$.

}\end{rem}

\subsection{Algebras over $\e=Ass^+\otimes \pp$ for an operad $\pp$}\label{AxP} 

Let $\pp$ be a chain operad, and consider the operad $Ass^+\otimes \pp$ whose algebras are unital associative algebras together with a commuting $\pp$--algebra structure.  By the work of Brun, Fiedorowicz, and Vogt \cite{BFV07}, if $\pp$ is the chain complex of the little disks operad $\mathcal{C}_n$, the resulting tensor product is an $E_{n+1}$--operad.  Furthermore, they show that the Hochschild complex of an $Ass^+ \otimes \pp$--algebra admits the structure of a $\pp$--algebra.

Explicitly, the action of $\pp$ on $C_*(A)$ is as follows: As $A$ is a unital associative algebra, we can consider $C_*(A)$ as the chain complex
associated to a simplicial chain complex $A_\bullet$ with $A_p=A^{\ot p+1}$ and degeneracy $s_i$ inserting a unit in position $i+1$. The $Ass^+\otimes
\pp$--structure of $A$ defines a simplicial $\pp$--structure on $A_\bullet$ by acting diagonally on $A^{\ot p+1}$, and this in turn induces a
$\pp$--structure on the associated total chain complex $C_*(A)$. This last structure can be made explicit via the Eilenberg-Zilber maps. The action of
a chain $p\in\pp(k)$ on $(a_0^1\ot\dots\ot  a_{p_1}^1)\ot\dots\ot(a_0^k\ot \dots \ot a_{p_k}^k)$ is of the form 
$$\sum\pm\,  p(a_0^1,\dots,a_0^k)\ot p(1,\dots,a_1,\dots,1)\ot \dots\ot  p(1,\dots,a_{p_1+\dots+p_k},\dots,1),$$ 
where the sum is over all possible shuffles of $(a_1^1,\dots, a_{p_1}^1),\dots,(a_1^k,\dots, a_{p_k}^k)$, with the resulting sequence denoted $a_1,\dots,a_{p_1+\dots+p_k}$, and $p(1,\dots,a_i,\dots,1)$ means take $a_i=a^j_k$ at the $j$th position and 1's everywhere else. 

By the results of the previous section, $HH_*(A, A)$ is a $H_*(\Ann)$--module.  It is natural, then, to ask how this interacts with the
Brun-Fiedorowicz-Vogt $\pp$--algebra structure. 
Comparing the above formula with the formula for Connes' $B$ operator (given at the end of Section~\ref{strict_section}) shows though that these two
structures do not interact very well, in particular because of the special role of the $a_0^j$'s in the $\pp$--action. 
One can though define an extension of the category $Ass\ot \pp$ with the free operad generated by $\pp$ and $B$ as ``closed-to-closed'' morphisms, subject to the relations in $\pp$ and $B^2 = 0$.

\section{Appendix: How to compute signs}\label{signsub}

Let $\Phi:\e\to\Comp$ be a split monoidal functor for $\e=\OO,\OO_d,\OC$ or $\OC_d$, with $\Phi(1)=A$ an $\Ai$--Frobenius algebra. 
Given an {\em oriented} graph $\Ga$ which is a morphism in $\e$, we want to read off an explicit formula of the associated operation on
$A$ or $C_*(A,a)$ {\em with
signs}.   The explicit formula will be given in terms of a chosen set of generating operations for $\OO$, for example in terms of the (co)product and
higher (co)products, the unit and the trace in $\OO$ (or $\OO_d$), and additionally the generator $l_n$ of Figure~\ref{graphsex} for $\OC$ (or $\OC_d$).

To be precise, one first needs to make a choice of which orientation should be thought of as the ``positive'' orientation for the
graphs representing the chosen basic operations. For the products and coproducts, we choose here the orientation $v\w h_1\w\dots\w h_k$ for $v$ the
vertex and  $h_1,\dots,h_k$ the half edges in their cyclic order starting at the first incoming half-edge. 
The unit and the trace are exceptional graphs with a canonical positive orientation. 
For $l_k$, we take the orientation $w\w h_1\w\dots\w h_k$ for $w$ the vertex,  $h_1,\dots,h_k$ the half edges in their cyclic order
starting at the start half-edge.

 Figure~\ref{prodsign} gives as an example the
convention we will use for the product in an algebra. \vspace{0.4cm}
\begin{figure}[h]
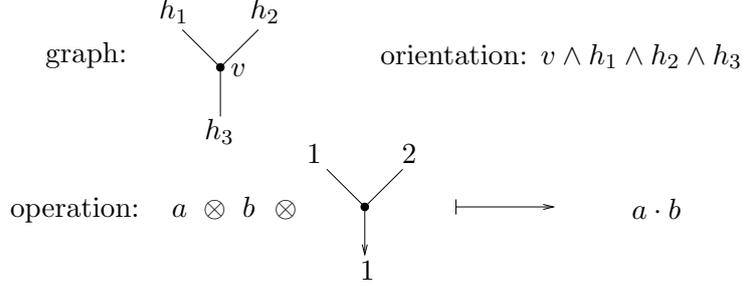

\begin{lpic}{prodsign(0.5,0.5)}
 \lbl[b]{-25,50;graph:}
 \lbl[b]{100,50;orientation: $v\w h_1\w h_2\w h_3$}
 \lbl[b]{-28,9;operation:}
 \lbl[b]{14,10;$a\ \ot\  b\ \ \ot $}
 \lbl[b]{125,10;$a\cdot b$}
 \lbl[b]{15,48;$v$}
 \lbl[b]{-2,62;$h_1$}
 \lbl[b]{22,62;$h_2$}
 \lbl[b]{10,30;$h_3$}
\lbl[b]{35,25;$1$}
\lbl[b]{60,25;$2$}
\lbl[b]{49,-6;$1$}
\end{lpic}
\caption{Sign convention for the product}\label{prodsign}
\end{figure}

Given a graph $\Ga$, we first need to write it as a composition of the chosen generating operations. This means choosing an orientation
of the internal edges and an ordering of the vertices, possibly introducing new vertices together with
unit or trace operations, and possibly using the symmetries of the category. (See Figure~\ref{composition} below for an example, and the proof of Proposition~\ref{positive_prop} for the case of $\OO^b$.)
Suppose $\Ga$ has vertices $v_1,\dots,v_k$ with half-edges $h_1^i,\dots,h^i_{n_i}$ at $v_i$ and $v_i\w h_1^i\w\dots\w h_{n_i}^i$ the chosen
orientation of the (chosen) operation $\mu_i$ associated to $v_i$. 
To interpret $\Ga$ as a composition of the operation at $v_1$, then at $v_2$ etc. requires writing the orientation of $\Ga$ as 
$\pm (v_1\w h_1^1\w\dots\w h_{n_1}^1)\w\dots\w (v_k\w h_1^k\w\dots\w h_{n_k}^k)$. 

\smallskip

Suppose we start from
$$a_1\ot\dots\ot a_n\ot (\Ga,o_d(\Ga))$$ in $A^{\ot n}\ot \OO_d(n,m)$, with  $\Ga$ as above and 
$$o_d(\Ga)=(v_1\w h_1^1\w\dots\w h_{n_1}^1)\w\dots\w
(v_k\w h_1^k\w\dots\w h_{n_k}^k)\ot \det(\Ga,\del_{out})^{\ot d}.$$ We rewrite this (with a Koszul sign!) as 
{\small $$a_1\ot\dots\ot a_n\ot\big((v_1\w h_1^1\w\dots\w h_{n_1}^1)\ot \det(\mu_1)^{\ot d}\big)\ot\dots\ot 
\big((v_k\w h_1^k\w\dots\w h_{n_k}^k)\ot \det(\mu_k)^{\ot d}\big)$$} 
 in 
$A^{\ot n}\ot \OO_d(n,p_1)\ot\dots\ot\OO_d(p_r,m)$, from which we can apply the first operation and then the next etc.
The final sign for the operation will come, in addition, from the signs occurring when using the symmetries in the category.

If the graph was an operation in $\OC_d$ instead, that is if we start with 
$$(a_0^1\ot\dots \ot a^1_{k_1}\ot l_{k_1})\ot\dots\ot (a_0^{n_1}\ot\dots \ot a^{n}_{k_{n}}\ot l_{k_n})\ot\  b_1\ot\dots\ot b_{m} \ot
(\Ga,o_d(\Ga))$$ in $C(A,A)^{\ot n})\ot A^{\ot m}\ot \OC_d(\oc{m}{n},\oc{m'}{n'})$, the principle is the same, but we have in addition to apply the procedure described in Section~\ref{ks_section}. 

\smallskip

We now give an explicit example with a graph of $\OO_d(2,1)$ which is used in the computations at the end of section \ref{strict_section}. 
In Figure~\ref{composition}, we give a graph with a choice of ordering of its vertices $v_1,v_2,v_3$, and a choice of orientation of
its internal edges $e_1,e_2,e_3$. We choose the orientation of the graph that corresponds to writing it as a composition of the operation
attached to $v_1$ (a coproduct), followed by the operation attached to $v_2$ and then $v_3$ (both products). Explicitly, it is given as 
$$(v_1\w h_1\w e_1\w e_2)\w(v_2\w \bar e_2\w h_2 \w e_3)\w(v_3\w \bar e_1\w \bar e_3 \w \bar h_1)$$
where $e_i$ and $\bar e_i$ are the start and end half-edges of $e_i$, $h_i$ is the $i$th incoming leaf, and $\bar h_1$ is the outgoing
leaf. 
\begin{figure}[h]
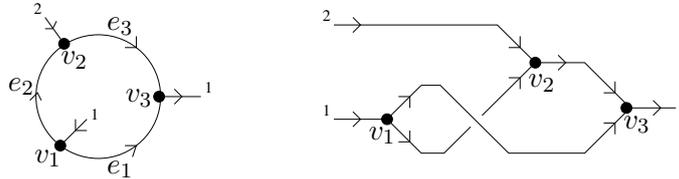

\begin{lpic}{composition(0.5,0.5)}
 \lbl[b]{5,-2;$v_1$}
 \lbl[b]{12,24;$v_2$} 
 \lbl[b]{29,14;$v_3$}
 \lbl[b]{93,4;$v_1$}
 \lbl[b]{135,18;$v_2$} 
 \lbl[b]{160,7;$v_3$}
 \lbl[b]{24,-5;$e_1$}
 \lbl[b]{-2,17;$e_2$} 
 \lbl[b]{24,33;$e_3$}
\end{lpic}
\caption{Writing a graph as a composition}\label{composition}
\end{figure}

The graph has relative Euler characteristic $\chi(\Ga,\del_{out})=-1$ which is also the relative Euler characteristic $\det(c)$ 
of the coproduct, while the products have
trivial relative Euler characteristic. As the products have degree 0, moving the determinant past the products does not produce a sign
and the operation associated to $\Ga$ with the above orientation is that of the composition 
 $$\big(((v_1\w h_1\w e_1\w e_2)\ot (\det{c})^{\ot d})\oplus id \big)\ot (\tau \oplus id)  \ot (v_2\w \bar e_2\w h_2 \w e_3)\ot(v_3\w \bar e_1\w \bar e_3 \w \bar h_1)$$
in $(\OO_d(1,2)\oplus \OO_d(1,1))\ot \OO_d(3,3)\ot \OO_d(3,2) \ot \OO_d(2,1)$, where $\tau$ denotes the twist map. 

The succession of operations (a comultiplication, a twist and two multiplications) applied to an pair $a\ot b$ is 
$$\begin{array}{lcl}a\ot b & \mapsto &(-1)^{|b|d}\sum a'\ot a''\ot b \\
& \mapsto & (-1)^{|b|d+|a'||a''|}\sum a''\ot a'\ot b \\
& \mapsto & (-1)^{|b|d+|a'||a''|}\sum a''\ot a' b \\
& \mapsto & (-1)^{|b|d+|a'||a''|}\sum a''a'b
\end{array}$$

\bibliography{biblio}

\end{document}